\documentclass[reqno]{amsart}  
\usepackage{amsmath,amssymb,amsthm}
\theoremstyle{plain}
\newtheorem{theorem}{Theorem}[section]

\newtheorem{remark}[theorem]{Remark}
\newtheorem{proposition}[theorem]{Proposition}

\newtheorem{corollary}[theorem]{Corollary}
\newtheorem{example}[theorem]{Example}

\numberwithin{equation}{section}
\allowdisplaybreaks[1]

\theoremstyle{definition}

\newtheorem{definition}[theorem]{Definition}

\theoremstyle{remark}

\newcommand{\sbm}[1]{\left[\begin{smallmatrix} #1
        \end{smallmatrix}\right]}

\newcommand{\cB}{{\mathcal B}}
\newcommand{\cC}{{\mathcal C}}
\newcommand{\cD}{{\mathcal D}}

\newcommand{\cG}{{\mathcal G}}
\newcommand{\cH}{{\mathcal H}}

\newcommand{\cJ}{{\mathcal J}}
\newcommand{\cK}{{\mathcal K}}
\newcommand{\cL}{{\mathcal L}}
\newcommand{\cM}{{\mathcal M}}

\newcommand{\cS}{{\mathcal S}}

\newcommand{\cU}{{\mathcal U}}

\newcommand{\cX}{{\mathcal X}}
\newcommand{\cY}{{\mathcal Y}}

\newcommand{\be}{{\mathbf e}}

\newcommand{\bA}{{\mathbf A}}
\newcommand{\bB}{{\mathbf B}}
\newcommand{\bC}{{\mathbf C}}
\newcommand{\bD}{{\mathbf D}}
\newcommand{\bU}{{\mathbf U}}
\newcommand{\bV}{{\mathbf V}}
\newcommand{\bY}{{\mathbf Y}}

\numberwithin{equation}{section}

\usepackage{amsfonts}

\begin{document}

\title[Schur--Agler and Herglotz--Agler classes]{Schur--Agler and
Herglotz--Agler classes of functions: positive-kernel
decompositions and transfer-function realizations}
\author[J.A.~Ball]{Joseph A. Ball}
\address{Department of Mathematics \\ Virginia Tech \\
Blacksburg, VA, 24061} \email{joball@math.vt.edu}
\author[D.S.~Kaliuzhnyi-Verbovetskyi]{Dmitry S.
Kaliuzhnyi-Verbovetskyi}\thanks{The authors were partially
supported by US--Israel BSF grant 2010432. The second author was
also partially supported by NSF grant DMS-0901628, and wishes to
thank the Department of Mathematics at Virginia Tech for
hospitality during his sabbatical visit in January--May 2013, when
a significant part of work on the paper was done.}
\address{Department of Mathematics \\
Drexel University\\
3141 Chestnut St.\\
  Philadelphia, PA, 19104}
\email{dmitryk@math.drexel.edu}
\date{}

\begin{abstract} We discuss transfer-function realization for
    multivariable holomorphic functions
    mapping the unit polydisk or the right polyhalfplane into the operator analogue of either the unit
    disk or the right halfplane (Schur/Herglotz functions over
    either the unit polydisk or the right polyhalfplane) which
    satisfy the appropriate stronger contractive/positive real part
    condition for the values of these functions on commutative
    tuples of strict contractions/strictly accretive operators
    (Schur--Agler/Herglotz--Agler functions over either the unit polydisk
    or the right polyhalfplane).  As originally shown by Agler, the
    first case (polydisk to disk) can be solved via unitary
    extensions of a partially defined isometry constructed in a
    canonical way from a kernel decomposition for the function (the
   {\em  lurking-isometry method}).  We show how a geometric reformulation of the
     lurking-isometry method (embedding of a given isotropic subspace of a Kre\u{\i}n
    space into a Lagrangian subspace---the {\em lurking-isotropic-subspace
    method}) can be used to handle the second two cases
    (polydisk to halfplane and polyhalfplane to disk), as well as the last
    case (polyhalfplane to halfplane) if an additional
    growth condition at $\infty$ is imposed.  For the general fourth
    case, we show how a linear-fractional-transformation change of
    variable can be used to arrive at the appropriate symmetrized
    nonhomogeneous Bessmertny\u{\i} long-resolvent realization.  We
    also indicate how this last result recovers the classical
    integral representation formula for scalar-valued holomorphic
    functions mapping the right halfplane into itself.
\end{abstract}


 \maketitle

 \section{Introduction}

 For $\cU$, $\cY$ coefficient separable Hilbert spaces, we define the
 operator-valued Schur class $\cS(\cU, \cY)$ (over the unit disk
 ${\mathbb D}$) to consist of all holomorphic functions $S$ on the
 unit disk ${\mathbb D}$ with values in the closed unit ball
  of the space
 $\cL(\cU, \cY)$ of bounded linear operators from $\cU$ to $\cY$,
 i.e.,  subject to $ \|S(\zeta)\| \le 1$  for all $\zeta \in
 {\mathbb D}$.
 The following result linking the theories of holomorphic functions,
 linear operators, and input/state/output linear systems is now well
 known (see e.g.~\cite{BB-Montreal} for a full discussion where
 multivariable extensions are also treated).

 \begin{theorem}  \label{T:Schur}  Given a function $S \colon
 {\mathbb D} \to \cL(\cU, \cY)$, the following are equivalent.

 \begin{itemize}
     \item[(1)] $S \in \cS(\cU, \cY)$.

     \item[(2)] The de Branges--Rovnyak kernel
  $$
  K_{S}(\omega, \zeta) = \frac{ I - S(\omega)^{*} S(\zeta)}{1 -
  \overline{\omega} \zeta}
  $$
  is a positive kernel on $\mathbb{D}$, i.e., there is an auxiliary Hilbert space
  $\cX$ and a holomorphic function $H \colon {\mathbb D} \to \cL(\cU,
  \cX)$ which gives rise to a {\em Kolmogorov decomposition} for
  $K_{S}$:
  $$
  K_{S}(\omega, \zeta) = H(\omega)^{*} H(\zeta).
  $$

  \item[(3)]  $S$ has a unitary transfer-function realization, i.e.,
  there is an auxiliary Hilbert state space $\cX$ and a unitary
  colligation matrix
  $$
  \bU  = \begin{bmatrix} A & B \\ C & D \end{bmatrix} \colon
  \begin{bmatrix} \cX \\ \cU \end{bmatrix} \to \begin{bmatrix} \cX\\
  \cY \end{bmatrix}
  $$
  so that
  $$
 S(\zeta) = D + \zeta C (I - \zeta A)^{-1} B \text{ for } \zeta
  \in {\mathbb D}.
  $$

  \item[(3$^{\prime}$)]  Condition (3) above holds where the
  colligation matrix $\bU$ is taken  to be any of (i) coisometric,
(ii)
  isometric, or (iii) contractive.
  \end{itemize}
  \end{theorem}

  It is natural to seek extensions of the Schur class to the
  multivariable setting where the disk ${\mathbb D}$ is replaced by
  the polydisk
  $$
  {\mathbb D}^{d} = \{ \zeta = (\zeta_{1}, \dots, \zeta_{d}) \in
  {\mathbb C}^{d} \colon |\zeta_{k}| < 1 \text{ for } k = 1, \dots,
  d\}.
  $$
  We therefore define the $d$-variable Schur class $\cS_{d}(\cU,
  \cY)$ to consist of holomorphic functions $S \colon {\mathbb D}^{d}
  \to \cL(\cU, \cY)$ subject to $\|S(\zeta) \| \le 1$ for all
  $\zeta \in {\mathbb D}^{d}$.  It was the profound observation of
  Agler \cite{Agler1990} that, unless $d \le 2$, a characterization
  of $\cS_{d}(\cU, \cY)$ of the same form as Theorem \ref{T:Schur} is
  not possible.  Instead, we define what is now called the
  Schur--Agler class, denoted as $\mathcal{SA}_{d}(\cU, \cY)$, to
  consist of holomorphic functions $S \colon {\mathbb D}^{d} \to
  \cL(\cU, \cY)$ such that $\|S(T_{1}, \dots, T_{d}) \| \le 1$
  whenever $T = (T_{1}, \dots, T_{d})$ is a commutative $d$-tuple of
  strict contraction operators on a fixed separable
  infinite-dimensional Hilbert space $\cK$.
Here the functional calculus defining $S(T_{1}, \dots, T_{d})$ can be
given by
$$
 S(T_{1}, \dots, T_{d}) = \sum_{n \in {\mathbb Z}^{d}_{+}} S_{n}
 \otimes T^{n}
$$
(convergence in the strong operator topology) where $S(\zeta) =
\sum_{n \in {\mathbb Z}^{d}_{+}} S_{n} \zeta^{n}$ is the
multivariable Taylor expansion for $S$ centered at the origin $0
\in {\mathbb D}^{d}$ and where we use standard multivariable
notation:
$$
\zeta^{n} = \zeta_{1}^{n_{1}} \cdots \zeta_{d}^{n_{d}}, \quad T^{n} =
T_{1}^{n_{1}} \cdots T_{d}^{n_{d}} \text{ if } n = (n_{1}, \dots,
n_{d}) \in {\mathbb Z}^{d}_{+}.
$$
The following result due to Agler \cite{Agler1990} (see also
\cite{AMcC99, BT})  has had a profound
impact on the subject over the years.

\begin{theorem}  \label{T:Agler}
     Given a function $S \colon
 {\mathbb D}^{d} \to \cL(\cU, \cY)$, the following are equivalent.

 \begin{itemize}
     \item[(1)] $S \in \mathcal{SA}_{d}(\cU, \cY)$.

     \item[(2)] $S$ has an {\em Agler decomposition} in the sense
     that there exists $d$ $\cL(\cY)$-valued positive kernels $K_{1},
     \dots, K_{d}$ on ${\mathbb D}^{d}$ such
     that
     \begin{equation} \label{DSchurAgdecom}
     I - S(\omega)^{*} S(\zeta) = \sum_{k=1}^{d} (1 - \overline{\omega}_{k}
     \zeta_{k}) K_{k}(\omega,\zeta).
     \end{equation}

  \item[(3)]  $S$ has a unitary Givone--Roesser $d$-dimensional
transfer-function realization, i.e.,
  there is an auxiliary Hilbert state space $\cX$ with a $d$-fold
orthogonal
  direct-sum decomposition $\cX = \cX_{1} \oplus \cdots \oplus
  \cX_{d}$ together with a unitary
  colligation matrix
  $$
  \bU  = \begin{bmatrix} A & B \\ C & D \end{bmatrix} \colon
  \begin{bmatrix} \cX \\ \cU \end{bmatrix} \to \begin{bmatrix} \cX\\
  \cY \end{bmatrix}
  $$
  so that
  \begin{equation}   \label{SAreal}
  S(\zeta) = D +  C (I - P(\zeta) A)^{-1} P(\zeta) B \text{ for }
\zeta
  \in {\mathbb D}^d.
  \end{equation}
  where we have set
  $$
   P(\zeta) = \zeta_{1} P_{1} + \cdots + \zeta_{d} P_{d}
   $$
   where $P_{k}$ is the orthogonal projection of $\cX$ onto $\cX_{k}$
   for each $k=1, \dots, d$.

  \item[(3$^{\prime}$)]  Condition (3) above holds where the
  colligation matrix $\bU$ is taken  to be any of (i) coisometric,
(ii)
  isometric, or (iii) contractive.
   \end{itemize}

  \end{theorem}

  The goal of this paper is to study parallel results for an
  assortment of linear-fractional transformed versions of the
  Schur--Agler class.  Specifically, we seek analogous
  characterizations of the following classes of holomorphic functions:
  \begin{enumerate}
      \item The Herglotz--Agler class over the unit polydisk ${\mathbb
      D}^{d}$, denoted $\mathcal{HA}({\mathbb D}^{d}, \cL(\cU))$
defined as the class of all holomorphic functions $F
      \colon {\mathbb D}^{d} \to \cL(\cU)$ such that $F(T)$ has positive real part
      $$
      F(T) + F(T)^{*} \ge 0
      $$
      for all commutative $d$-tuples $T = (T_{1}, \dots, T_{d})$ of
      strict contractions on the Hilbert space $\cK$.

      \item The Schur--Agler class over the right polyhalfplane
      $$ \Pi^{d} = \{ z = (z_{1}, \dots, z_{d}) \in {\mathbb C}^{d}
      \colon z_{k} + \overline{z}_{k} > 0 \text{ for } k = 1, \dots, d\},
      $$
      denoted by $\mathcal{SA}(\Pi^{d}, \cL(\cU, \cY))$,
      consisting of all holomorphic $\cL(\cU, \cY)$-valued functions
     $s$ on $\Pi^{d}$ such that $\|s(A)\| \le 1$ for all strictly
     accretive commutative $d$-tuples $A = (A_{1}, \dots, A_{d})$  of
operators
     on $\cK$ (i.e., such that $A_k+A_k^*\ge cI$ for some constant $c>0$, $k=1,\ldots,d$).

     \item The Herglotz--Agler class over $\Pi^d$, denoted by
      $\mathcal{HA}(\Pi^{d}, \cL(\cU))$, consisting of all
     holomorphic functions $f$ on $\Pi^{d}$ such that $f(A) + f(A)^{*}
     \ge 0$ for all  strictly accretive commutative $d$-tuples $A$ of operators on $\cK$.

     \item The Nevanlinna--Agler class over the upper polyhalfplane
     $$ (i\Pi)^{d} = \{ z = (z_{1}, \dots, z_{d}) \in
     {\mathbb C}^{d} \colon \frac{ z_{k} - \overline{z}_{k}}{2 i} > 0
\text{ for }
     k=1, \dots, d\},
     $$
     denoted by $\mathcal{NA}((i\Pi)^{d}, \cL(\cU))$,  consisting of
     all holomorphic $\cL(\cU)$-valued functions $\widetilde f$ on
$(i
     \Pi)^{d}$ such that $\frac{1}{2i} (\widetilde f(\widetilde A) -
\widetilde f(\widetilde
     A)^{*}) \ge 0$ whenever $\widetilde A = (\widetilde A_{1},
     \dots, \widetilde A_{d})$ is a commutative $d$-tuple of
     operators on $\cK$, each with strictly positive-definite
imaginary part (i.e., such that $\frac{1}{2i}(A_k-A_k^*)\ge cI$
for some constant $c>0$, $k=1,\ldots,d$).
       \end{enumerate}
 To be consistent with the more detailed notation used for these
 variants of the Schur--Agler class, we will also use the notation
 $\mathcal{SA}({\mathbb D}^{d}, \cL(\cU, \cY))$ for the Schur--Agler
 class $\mathcal{SA}_{d}(\cU, \cY)$ over the polydisk ${\mathbb
 D}^{d}$ discussed above.

 We note that our convention is to use the term {\em Herglotz} for
 functions with values having positive real part, and {\em Nevanlinna}
 for functions with values having positive imaginary part; we
 recognize that these
 conventions  are by no means universal (see e.g. \cite{ABT}).

  For the single-variable case such realization results have been
  explored in a systematic way in \cite{Staffans-ND} and \cite{BS}.
  For the multivariable setting, apart from the now classical
  Schur--Agler class over the polydisk $\mathcal{SA}_{d}(\cU,
  \cY)$, the only results along these lines which we are aware of are
  those in the recent paper of Agler--McCarthy--Young \cite{AMcCY} and
  of Agler--Tully-Doyle--Young \cite{ATDY1, ATDY2}.

  The approach in \cite{Staffans-ND} (in the single-variable
  setting) is to use a linear-fractional-transformation (LFT)
  change of variables (on the domain and/or the range side) to
  reduce the desired result to the corresponding result for the Schur
  class over the unit disk. This is also the main tool in
  \cite{AMcCY, ATDY1, ATDY2}: use an LFT Cayley-transform change of
  variables to reduce results for the Nevanlinna--Agler class to the
  corresponding known results for the Schur--Agler class.  However the
  procedure is rather intricate due to the added subtleties involved
  in  handling points at infinity in the multivariable case.

  In contrast, the approach in \cite{BS}
  is to apply a projective version of the lurking isometry argument
  (roughly, a lurking-isotropic-subspace argument in a
  Kre\u{\i}n-space setting) to arrive at the desired realization
  result via a direct but unified  Kre\u{\i}n-space geometric
  argument.  One of the main contributions of the present work is to
  extend this approach to the multivariable setting.  The main
  difficulty is to guarantee that a naturally defined isotropic
  subspace is actually a graph space with respect to a system of
  coordinates not coming from a fundamental decomposition of the
  ambient Kre\u{\i}n space.  We show how this difficulty can be
  overcome for the case of the ${\mathbb D}^{d}$-Herglotz--Agler class
  and $\Pi^{d}$-Schur--Agler class.  For the $\Pi^{d}$-Herglotz--Agler
  class, we are able to overcome the difficulty only in a special case
  (associated with the imposition of a growth condition at
  infinity), thereby recovering parallel results from \cite{AMcCY}.
  For the most general $\Pi^{d}$-Herglotz--Agler
  function $f$, we follow the LFT change-of-variable approach of
  \cite{ATDY1} combined with the more general realization formalism
  (Schur complement of an operator pencil) suggested by the work of
  Bessmertny\u{\i} (see \cite{Bes1, Bes2, Bes3, Bes4, K-V})
 to arrive at a realization formula for the most general
  $\Pi^{d}$-Herglotz--Agler function.  We note that  the original
Bessmertny\u{\i} class involved additional symmetries leading to
strong rigidity results. It was  conjectured in \cite{B} that an
appropriate weakening of the metric conditions for the
Bessmertny\u{\i} operator pencil should lead to a representation
for the most general $\Pi^{d}$-Herglotz--Agler function.  Here we
show that this conjecture is correct once one identifies the
appropriate modification:  one must allow the nonhomogeneous
skew-adjoint term in the nonhomogeneous Bessmertny\u{\i} operator
pencil to be unbounded (more precisely, a certain flip
$\Pi$-impedance-conservative system node in the sense of
\cite{Staffans-ND}).

  There has been a lot of work on transfer-function realization for
  the single-variable Schur and Herglotz classes over the right
  half plane.  The most influential for our point of view toward
  multivariable generalizations is the work of Arov-Nudelman
  \cite{AN} and of Staffans and collaborators (see
  \cite{StaffansMCSS, Staffans-ND, BS, MSW, Staffans2013} as well as the treatise \cite{Staffans}
  and the references there).  There is also a complementary approach
  to such realization theory (upper halfplane rather than right
  halfplane version)  with emphasis on the theory of selfadjoint
  extensions of densely defined symmetric operators  on a Hilbert
  space (see \cite{BHdS2008, BHdS2009, BHdST} as well as the recent
  book \cite{ABT} and the references there).

  The paper is organized as follows.  Section \ref{S:prelim} highlights
  the main ideas from Kre\u{\i}n-space geometry and from infinite-dimensional
  systems theory, in
  particular, the idea of a {\em system node}, which will be used in
  the later sections.  Section \ref{S:HA-D} presents our results for
  the Herglotz--Agler class over the polydisk while Section
  \ref{S:SA-Pi} does the same for the Schur--Agler class over the right
  polyhalfplane.  Section \ref{S:HA-Pi} presents our results for the
  restricted Herglotz--Agler class over the right polyhalfplane where
  a growth condition at infinity is imposed on the functions to be
  realized.  With this added restriction, the
  lurking-isotropic-subspace method from \cite{BS} adapts well to lead
  to a classical type realization (but with an in general unbounded
  $\Pi$-impedance-conservative system node) for the
  $\Pi^{d}$-Herglotz--Agler function.  Section 6 identifies the
 nonhomogeneous unbounded Bessmertny\u{\i} operator pencils which then
 lead to a realization for the most general Herglotz--Agler function
 over  $\Pi^{d}$.  We also mention that the results parallel to the
 results of this paper for the four classes under discussion
 (${\mathbb D}^{d}$-Schur--Agler, ${\mathbb D}^{d}$-Herglotz--Agler,
 $\Pi^{d}$-Schur--Agler and $\Pi^{d}$-Herglotz--Agler)  for the
 rational matrix-valued (Cayley) inner case, where the emphasis is on
 obtaining realizations with finite-dimensional state space, are
 obtained in our companion paper \cite{BK-V}.

 We shall have occasion to need a Cayley transform (with both scalar
 and operator argument) acting between the right halfplane and the
 unit disk.  Following the conventions in \cite{K-V} and \cite{BK-V},
 we shall make use of the following version:
 \begin{align}
  & \zeta \in {\mathbb D}  \mapsto w = \frac{1+ \zeta}{1-\zeta} \in
   \Pi, \text{ with inverse given by } \notag \\
   & w \in \Pi  \mapsto \zeta = \frac{w-1}{w+1} \in {\mathbb D}.
 \label{Cayley}
 \end{align}
 For $\zeta = (\zeta_{1}, \dots, \zeta_{d})$ a point in the unit
 polydisk ${\mathbb D}^{d}$, we continue to use the notation
 \begin{equation} \label{Cayley1}
     \frac{1 + \zeta}{1 - \zeta} : =
     \left( \frac{ 1 + \zeta_{1}}{1 - \zeta_{1}}, \dots,
       \frac{1 + \zeta_{d}}{1 - \zeta_{d}} \right)
  \end{equation}
  for the corresponding point in the right polyhalfplane $\Pi^{d}$.
  Similarly, given a point $w = (w_{1}, \dots, w_{d})$ in the right
  polyhalfplane $\Pi^{d}$, we use the notation
  \begin{equation}   \label{Cayley2}
      \frac{w-1}{w+1} := \left( \frac{ w_{1}-1}{w_{1}+1}, \dots,
      \frac{w_{d}-1}{w_{d}+1}\right)
   \end{equation}
for the associated point in the polydisk.

\subsection*{Acknowledgement.}  The authors are
thankful to the anonymous referee for his/her careful reading and
constructive remarks.

  \section{Preliminaries}  \label{S:prelim}

  \subsection{Decompositions of the identity} \label{S:decomid}

  Given a Hilbert space $\cX$, we shall say that a collection of $d$
  operators $(Y_{1}, \dots, Y_{d})$ on $\cX$ forms a {\em $d$-fold
  positive decomposition of the identity $I_{\cX}$} if each $Y_{k}$
  is a selfadjoint contraction ($0 \le Y_{k} \le I_{\cX}$ for $1 \le
  k \le d$) which together sum up to the identity ($\sum_{k=1}^{d} Y_{k} =
  I_{\cX}$).  In case $(Y_{1}, \dots, Y_{d}) = (P_{1}, \dots, P_{d})$
  consists of orthogonal projection operators (necessarily with
  pairwise orthogonal ranges) we shall say that $(P_{1}, \dots,
  P_{d})$ forms a {\em $d$-fold spectral decomposition of $I_{\cX}$}.
  Note that $d$-fold spectral decompositions $(P_{1}, \dots, P_{d})$
  arise in the realization formula for the Schur--Agler class in
  Theorem \ref{T:Agler}.  We shall see that the more general
  positive decompositions are needed for the realization formulas
  for functions in the $\Pi^{d}$-Schur--Agler class and in the
  $\Pi^{d}$-Herglotz--Agler class, as already discovered in
  \cite{AMcCY, ATDY1, ATDY2}.

  From the definitions we see that any spectral decomposition
  $(P_{1}, \dots, P_{d})$ is also a positive decomposition.
  There is also a result in the converse direction:  {\em if $(Y_{1},
  \dots, Y_{d})$ is a positive decomposition of $I_{\cX}$, then there
  exist a Hilbert space $\widetilde \cX$,
  a spectral decomposition  $(P_{1}, \dots, P_{d})$ of
  $I_{\widetilde \cX}$, and an isometric embedding
  $\iota\colon\mathcal{X}\to\widetilde{\mathcal{X}}$ such that}
  $$
     Y_{k} = \iota^*P_{k} \iota \ \text{ for } k = 1, \dots, d.
  $$
  This can be seen as a consequence of the Naimark dilation theorem
  (apply \cite[Theorem 4.6]{Paulsen} with the measurable space $X$ taken
  to be the finite set $\{k \in {\mathbb N} \colon 1 \le k \le d\}$).
  To prove the result for this simple case of the Naimark dilation
  theorem, simply define an isometric embedding of $\cX$ into
  $\bigoplus_{i=1}^{d} \cX$ by
  $$
   \iota = \left[ \begin{matrix} Q_{1} \\ \vdots \\ Q_{d}
\end{matrix} \right]
$$
where $Q_{k}$ provides a factorization $Y_{k} = Q_{k}^{*} Q_{k}$
and take $P_{k}$ equal to the projection onto the $k$-th block in the
direct-sum space $\bigoplus_{i=1}^{d} \cX$.

\subsection{Basics on the geometry of Kre\u{\i}n spaces}
\label{S:Krein}

In this Section we review some basics about the geometry of
Kre\u{\i}n spaces and Kre\u{\i}n-space operator theory which we shall
need in the sequel.  Other resources on this topic is a similar
survey section in the paper  \cite{BS} as well as the more complete
treatises \cite{Bognar, AI}.

A Kre\u{\i}n space by definition is a linear space $\cK$ endowed
with an indefinite inner product $[ \cdot, \cdot]$ which is
{\em complete} in the following sense:  there are two subspaces
$\cK_{+}$ and $\cK_{-}$ of $\cK$ such that the restriction of $[
\cdot, \cdot ]$ to $\cK_{+} \times \cK_{+}$ makes $\cK_{+}$ a
Hilbert space while the restriction of $-[ \cdot, \cdot]$ to
$\cK_{-} \times \cK_{-}$ makes $\cK_{-}$ a Hilbert space, and $\cK = \cK_{+} [\dot
+] \cK_{-}$ is a $[ \cdot, \cdot]$-orthogonal direct sum decomposition
of $\cK$.  In this case the decomposition $\cK = \cK_{+} [ \dot + ]
\cK_{-}$ is said to form a {\em fundamental decomposition} for $\cK$.
Fundamental decompositions are never unique except in the trivial
case where one of $\cK_{+}$ or $\cK_{-}$ is equal to the zero space.

Unlike the case of Hilbert spaces where closed subspaces all look
the same, there is a rich geometry for subspaces of a Kre\u{\i}n
space. A subspace $\cM$ of a Kre\u{\i}n space $\cK$ is said to be
{\em positive}, {\em isotropic}, or {\em negative} depending on
whether $[u, u] \ge 0$ for all $u \in \cM$, $[ u, u ] =0$ for all
$u \in \cM$ (in which case it follows that $[u,v] = 0$ for all
$u,v \in \cM$ as a consequence of the Cauchy-Schwarz inequality),
or $[u,u] \le 0$ for all $u \in \cM$.  Given any subspace $\cM$,
we define the Kre\u{\i}n-space orthogonal complement
$\cM^{[\perp]}$ to consist of all $v \in \cK$ such that $[u,v] =
0$ for all $u \in \cK$.  Note that the statement that $\cM$ is
isotropic is just the statement that $\cM \subset \cM^{[ \perp
]}$.  If it happens that $\cM = \cM^{[ \perp]}$, we say that $\cM$
is a {\em Lagrangian} subspace of $\cK$.

Examples of such subspaces arise from placing appropriate
Kre\u{\i}n-space inner products on the direct sum $\cH_{1} \oplus
\cH_{2}$ of two Hilbert spaces and looking at graphs of operators of
an appropriate class.

\begin{example}  \label{E:unitary}  {\em Suppose that $\cH'$ and
    $\cH$ are two Hilbert spaces and we take $\cK$ to be the
    external direct sum $\cH' \oplus \cH$ with inner product
    $$
    \left[ \begin{bmatrix} x \\ y \end{bmatrix}, \,
    \begin{bmatrix} x' \\ y' \end{bmatrix} \right] =
    \left\langle \begin{bmatrix} I_{\cH'} & 0 \\ 0 & -I_{\cH}
    \end{bmatrix} \begin{bmatrix} x \\ y \end{bmatrix}, \,
    \begin{bmatrix} x' \\ y' \end{bmatrix} \right\rangle_{\cH' \oplus \cH}
    $$
where $\langle \cdot, \cdot \rangle_{\cH' \oplus \cH}$ is the
standard Hilbert-space inner product on the direct-sum Hilbert
space $\cH' \oplus \cH$.  In this case it is easy to find a
fundamental decomposition: take $\cK_{+} = \sbm{ \cH \\ \{0\}}$
and $\cK_{-} = \sbm{ \{0\} \\ \cH' }$. Now let $T$ be a bounded
linear operator from $\cH$ to $\cH'$ and let $\cM$ be the graph of
$T$:
$$
  \cM = \cG_{T} = \left\{ \begin{bmatrix} T x \\ x \end{bmatrix}
  \colon x \in \cH \right\} \subset \cK.
  $$
  Then a good exercise is to work out the following facts:
  \begin{itemize}
      \item $\cG_{T}$ is negative if and only if $\|T\| \le 1$.

      \item $\cG_{T}$ is isotropic if and only if $T$ is isometric
      ($T^{*} T = I_{\cH}$).

      \item $\cG_{T}$ is Lagrangian if and only if $T$ is unitary:
      $T^{*} T = I_{\cH}$ and $T T^{*} = I_{\cH'}$.
   \end{itemize}
}\end{example}

\begin{example}  \label{E:skewadjoint} {\em  Let $\cH$ be a Hilbert space
    and set $\cK$ equal to the direct-sum space $\cK = \cH \oplus
    \cH$ with indefinite inner product given by
    $$ \left [ \begin{bmatrix} x \\ y \end{bmatrix},
    \begin{bmatrix} x' \\ y' \end{bmatrix} \right ] =
    \left \langle \begin{bmatrix} 0 & I_{\cH} \\ I_{\cH} & 0
    \end{bmatrix} \begin{bmatrix} x \\ y \end{bmatrix}, \,
    \begin{bmatrix} x' \\ y' \end{bmatrix} \right\rangle_{\cH \oplus
    \cH}.
$$
In this case a choice of fundamental decomposition is not so obvious;
one such choice is
$$
  \cK_{+} = \left\{ \begin{bmatrix} x \\ x \end{bmatrix} \colon x \in
  \cH \right\}, \quad
   \cK_{-} = \left\{ \begin{bmatrix} -y \\ y \end{bmatrix} \colon y \in
  \cH \right\}.
 $$
Then the exercise parallel to that suggested in Example
\ref{E:unitary} is to work out the following:  given a closed operator $T
\in \cL(\cH)$ with dense domain $\cD(T) \subset \cH$, let
$$
 \cG_{T} = \left\{ \begin{bmatrix} Tx \\ x \end{bmatrix} \colon x \in
 \cD(T) \right\} \subset \cK
 $$
 be its graph space.  Then:
 \begin{itemize}
     \item  $\cG_{T}$ is negative if and only if $T$ is dissipative:
     $\langle Tx, x \rangle + \langle x, Tx \rangle \le 0$ for all $x
     \in \cD(T)$.

     \item $\cG_{T}$ is {\em maximal negative}, i.e.,
     $\cG_{T}$ is negative and is not contained in any properly
     larger negative subspace, if and only if
     $T$ is {\em maximal dissipative}, i.e.,  $T$ is dissipative and  has no
     proper dissipative extension. An equivalent condition is $T$ is
     dissipative and the operator $I + T$ is onto (or equivalently
     $wI + T$ is onto for all $w$ in the right halfplane $\Pi^{+}$) (see
     \cite{Phillips59}).

     \item $\cG_{T}$ is isotropic if and only if $T$ is {\em
     skew-symmetric} or $T \subset -T^{*}$,
     i.e.  $\langle Tx, x \rangle_{\cH} + \langle x, T x
     \rangle_{\cH} = 0$ for all $x \in \cD(T)$.

     \item $\cG_{T}$ is Lagrangian if and only if $T$ is
     skew-adjoint or $T = -T^{*}$ (i.e., $T \subset -T^{*}$ and
     $T$ and $T^{*}$ have the same domain:  $y,z \in \cH$ such that
     $\langle Tx, y \rangle_{\cH} = \langle x, z \rangle_{\cH}$
     implies that $y \in \cD(T)$ and $z = -Ty$.  A closely related
     result is proved in Corollary \ref{C:impconserv} below.
 \end{itemize}
 }\end{example}

 We shall have use for the following connection between Examples
 \ref{E:unitary} and \ref{E:skewadjoint}.  Consider Example
 \ref{E:unitary} for the case where $\cH' = \cH$ and call this
 Kre\u{\i}n space $\cK_{1}$.  Let $\cK_{2}$ be the Kre\u{\i}n space
 $\cH \oplus \cH$ with Kre\u{\i}n-space inner product
 the negative of the inner product as in Example \ref{E:skewadjoint}:
  $$ \left [ \begin{bmatrix} x \\ y \end{bmatrix},
    \begin{bmatrix} x' \\ y' \end{bmatrix} \right ]_{\cK_{2}} =
    \left \langle \begin{bmatrix} 0 & -I_{\cH} \\ -I_{\cH} & 0
    \end{bmatrix} \begin{bmatrix} x \\ y \end{bmatrix}, \,
    \begin{bmatrix} x' \\ y' \end{bmatrix} \right\rangle_{\cH \oplus
    \cH}.
$$
Note that this modification of the inner product just interchanges
positive and negative subspaces and preserves isotropic and
Lagrangian subspaces.  Then the operator $\Gamma$ defined by
 $$
   \Gamma: = \frac{1}{\sqrt{2}} \begin{bmatrix} I_{\cH} & I_{\cH} \\ - I_{\cH}
   & I_{\cH} \end{bmatrix}
   \colon \begin{bmatrix} \cH \\ \cH \end{bmatrix} \to
   \begin{bmatrix} \cH \\ \cH \end{bmatrix}
  $$
  defines a Kre\u{\i}n-space isomorphism between the Kre\u{\i}n space
  $\cK_{1}$ as in Example \ref{E:skewadjoint} and the Kre\u{\i}n space
 $\cK_{2}$ as in Example \ref{E:unitary};  this is a consequence
 of the bijectivity of $\Gamma$ along with the identity
 $$
    \Gamma^{*} \begin{bmatrix} 0 & -I \\ -I & 0 \end{bmatrix} \Gamma =
    \begin{bmatrix} I & 0 \\ 0 & -I \end{bmatrix}.
$$
It follows that $\Gamma$ maps $\cK_{1}$-Lagrangian subspaces to
$\cK_{2}$-Lagrangian subspaces.  In particular, if $U$ is unitary and
does not have $1$ as an eigenvalue, we define the Cayley transform $Y$
of $T$ according to the formula
\begin{equation}   \label{defY}
 \cD(Y) = {\rm Ran} (I - U) \text{ and } Y = (I + U) (I - U)^{-1}.
\end{equation}
Note that
\begin{align*}
    \Gamma \cG_{U} & =  \frac{1}{\sqrt{2}}\begin{bmatrix} I & I \\ -I & I \end{bmatrix}
    \left\{
 \begin{bmatrix} U x \\ x \end{bmatrix} \colon x \in \cH \right\} \\
     & = \left\{ \begin{bmatrix} (I + U) x \\ (I - U) x \end{bmatrix}
     \colon x \in \cH \right\} \\
     & = \left\{ \begin{bmatrix} (I + U) (I - U)^{-1} \\ I
 \end{bmatrix} (I - U) x \colon x \in \cH\right\} \\
 & = \left\{ \begin{bmatrix} Y \\ I_{\cH} \end{bmatrix} y \colon y
 \in \cD(Y) \right\} = \cG_{Y}.
\end{align*}
There results the following fact concerning the Cayley transform map
of this type.

\begin{proposition}   \label{P:unitary/skewadj}  Suppose that $U$ is
a linear operator on a Hilbert space $\cH$ which does not have $1$ as
an eigenvalue.  Define the Cayley transform $Y$ of $U$ as in
\eqref{defY}.  Then $U$ is unitary if and only if $Y$ is skew-adjoint.
\end{proposition}

  \subsection{Well-posed linear systems and system nodes}
  \label{S:sys}

  A continuous-time input/state/output (i/s/o) linear system is a system of
  equations of the form
  \begin{equation}  \label{sys}
      \Sigma \colon \left\{
      \begin{array}{rcl}
      \dot x(t) & = & A x(t) + B u(t)  \\
      y(t) & = & C x(t) + D u(t)
  \end{array}  \right.
  \end{equation}
  where $x(t)$ takes values in the {\em state space} $\cX$, $u(t)$
  takes values in the {\em input space} $\cU$, and $y(t)$ takes values in
the
  {\em output space} $\cY$.  Under the assumption that the {\em system
  matrix} (sometimes also called the {\em colligation matrix})
  $$
    \Sigma: = \begin{bmatrix} A & B \\ C & D \end{bmatrix} \colon
    \begin{bmatrix} \cX \\ \cU \end{bmatrix} \to \begin{bmatrix} \cX
    \\ \cY \end{bmatrix}
  $$
  consists of {\em  bounded} operators,
  imposition of the initial condition $x(0) = 0$
  and application of the Laplace transform
  $$
  \widehat x(w) = \int_{0}^{\infty} e^{-wt} x(t)\,{\tt d}t, \quad
  \widehat u(w) = \int_{0}^{\infty} e^{-wt} u(t)\,{\tt d}t, \quad
  \widehat y(w) = \int_{0}^{\infty} e^{-wt} y(t)\,{\tt d}t
  $$
  leads to the input-output relation in the frequency domain
  $$
    \widehat y(w) = T_{\Sigma}(w) \widehat u(w)
  $$
  where
  \begin{equation*}
      T_{\Sigma}(w) = D + C (wI - A)^{-1} B
   \end{equation*}
   is the {\em transfer function} of the linear system $\Sigma$
   \eqref{sys}.
   The converse question of when an $\cL(\cU, \cY)$-valued
function
   $f$ on the right halfplane $\Pi$ can be realized as $f=
T_{\Sigma}$
   for some system $\Sigma =
   \left[ \begin{matrix} A & B \\ C & D
\end{matrix}\right]$ has generated much interest over the
years.  The obvious necessary condition is that $w \mapsto f(w)$
be analytic on some right halfplane but this is not sufficient if
we limit our attention to systems $\Sigma$ consisting only of
bounded operators $A,B,C,D$.  While it is clear that the right
generalization of the state operator $A$ is that it should be the
generator of a $C_{0}$-semigroup, exactly how to handle the
remaining operators $B,C,D$ so as to get a meaningful theory
containing compelling examples of interest  was not so clear, but
some progress was made already in the 1970s (see \cite{Fuhrmann,
Helton}).  It is now understood that a useful notion of
generalized system matrix $\Sigma$  is that associated with
so-called {\em well-posed} system.  Roughly, a {\em well-posed
linear system} is an i/s/o linear system for which the integral
form of the system operators ${\mathfrak A}, {\mathfrak B},
{\mathfrak C}, {\mathfrak D}$ satisfy natural compatibility
conditions and the integral form of the system matrix
\begin{equation}   \label{wellposedsys}
  \begin{bmatrix}  \mathfrak A^{t} & {\mathfrak B}_{0}^{t} \\
      {\mathfrak C}_{0}^{t} & {\mathfrak D}_{0}^{t} \end{bmatrix}
      \colon \begin{bmatrix} x(0) \\ u|_{[0,t)} \end{bmatrix} \mapsto
      \begin{bmatrix} x(t) \\ y|_{[0,t)} \end{bmatrix}
\end{equation}
makes sense as a bounded operator from $\cX \oplus
L^{2}_{\cU}([0,t))$ to $\cX \oplus L^{2}_{\cY}([0,t))$ for each
$t>0$ (see \cite{Staffans} for complete details).    The ``right''
infinitesimal object (the analogue of the system matrix $\Sigma =
\left[ \begin{matrix} A & B \\ C & D \end{matrix} \right]$
appearing in \eqref{sys}) is the notion of {\em system node} defined
as follows; this notion is well laid out in the work of Staffans
\cite{Staffans, Staffans-ND} where it is acknowledged that much of
the idea was already anticipated in the earlier work of Salamon \cite{Sala87} and Smuljan
\cite{Smuljan}.

We first make some preliminary observations.  A system node $\Sigma$
is still an operator from $\cX \oplus \cU$ to $\cX \oplus \cY$ but
now allowed to be unbounded with some domain $\cD(\Sigma) \subset \cX
\oplus \cU$.  We may then split $\Sigma$ in the form
$$
  \Sigma = \begin{bmatrix} \Sigma_{1} \\ \Sigma_{2} \end{bmatrix}
 $$
 where $\Sigma_{1} \colon \cD(\Sigma) \to \cX$ and $\Sigma_{2} \colon
 \cD(\Sigma) \to \cY$. However we allow the possibility that
 $\cD(\Sigma)$ does not split $\cD(\Sigma) = \left[
\begin{matrix}
 \cD(\Sigma)_{1} \\ \cD(\Sigma)_{2} \end{matrix} \right]$ as the
 direct sum of a linear manifold $\cD(\Sigma)_{1}$ in $\cX$ with a
 linear manifold $\cD(\Sigma)_{2}$ in $\cU$.  To keep the parallel
 with the classical case, we therefore write
 $$
   \Sigma = \begin{bmatrix} A \& B \\ C \& D \end{bmatrix}
   $$
   with the notation $A \& B$ (and similarly $C \& D$) suggesting that
 the common domain of $A \& B$  and $C \& D$ in $\cX \oplus \cU$ may
 not have a splitting.  However $A \& B$ will have a splitting
 $\begin{bmatrix} A_{|\cX} & B \end{bmatrix}$ where $A_{|\cX}$ and $B$
are
 operators  mapping the spaces $\cX$ and $\cU$ into a larger ``rigged'' space
 $\cX_{-1}$ which is part of a so-called {\em Gelfand triple} defined
 as follows.

 We assume that $A$ is an in general unbounded closed operator with dense
 domain $\cD(A): = \cX_{1}$
 in $\cX$ and with nonempty resolvent set.  Then $\cX_{1}$ is a Hilbert space in its own right with
 respect to the $\cX_{1}$-norm given by
 $$
    \| x \|_{1} : = \| (\alpha I - A) x \| _{\cX}
 $$
 where $\alpha$ is any fixed number in the resolvent set of $A$.
  While the norm depends on the choice of $\alpha$, any other choice
  $\alpha'$
 of $\alpha$ leads to the same space but with an equivalent norm, as
 long as $\alpha$ and $\alpha'$ are in the same connected component
 of the resolvent set of $A$.  If $A$ is the generator of a
 $C_{0}$-semigroup (the most interesting case for us), one can take
 the connected component of the resolvent set to be a right half plane.
 Note that $(\alpha I - A)|_{\cX_{1}}$ can be viewed as an isometry from $\cX_{1}$
 onto $\cX$.

 We next construct another Hilbert space $\cX_{-1}$ as the completion
 of $\cX$ in the norm
 $$
   \| x \|_{-1} : = \| (\alpha I - A)^{-1} x \| _{\cX}.
 $$
 If $\{x_{n}\}_{n \in {\mathbb Z}_{+}}$ is a sequence in $\cX_{1}$
 converging to $x \in \cX$ in $\cX$-norm, then the sequence
 $\{(\alpha I - A) x_{n}\}_{n \in {\mathbb Z}_{+}}$ is Cauchy in
 $X_{-1}$-norm and hence converges to an element $x^{-1} \in
 \cX_{-1}$.  One can check that this element $x^{-1} \in \cX_{-1}$ is
 independent of the choice of sequence $\{x_{n}\} \subset \cX_{1}$
 converging in $\cX$-norm to $x$; we denote this element $y \in
 \cX_{-1}$ by $y = (\alpha I - A) x$ and then define an extension
 $A_{|\cX} \colon \cX \to \cX_{-1}$ by
 $$
 A_{|\cX} \colon x \mapsto \alpha x -  (\alpha I - A)x \in \cX_{-1} \text{
 if } x \in \cX.
 $$
 We then have that the extended operator $(\alpha I - A)_{|\cX} $ is
 an isometry from $\cX$ onto $\cX_{-1}$ and we have
 the nested inclusions $\cX_{1} \subset \cX \subset
 \cX_{-1}$ with continuous and dense injections.  We will on occasion
 simplify the notation $A_{|\cX}$ to simply $A$ when the meaning is
 clear; thus for $x \in \cX$ and $x$ not necessarily in $\cX_{1} =
 \cD(A)$, the element $A x$ is still defined but as an element of
 $\cX_{-1}$.

 It is also useful to note the role of these spaces in duality
 pairings.  First, we note that the constructions in the previous
 paragraph can be carried out using the operator $A^{*}$ in place of
 $A$.  When this is done we get spaces $X_{1}^{\star} = (\overline{\alpha}
 I - A^{*})^{-1} \cX$ (with the norm $\|x\|_{1,\star}:=\|(\overline{\alpha}I-A^*)x\|_{\cX}$
 and $X_{-1}^{\star}$ equal to the completion of
 $\cX$ in the $\cX_{-1}^{\star}$-norm $\| x \|_{-1,\star} =
 \|(\overline{\alpha} I - A^{*})^{-1} x\|_{\cX}$ with the properties
 that $(\overline{\alpha} I - A^{*})^{-1}$ is an isometry from $\cX$ onto
 $\cX_{1}^\star$ and $(\overline{\alpha} I - A^{*})$ extends to an
 isometry from $\cX$ onto $\cX_{-1}^{\star}$ with the nesting $\cX_{1}^{\star}
 \subset \cX \subset \cX_{-1}^{\star}$.
 Given any $x \in \cX$, we can view $x$ as a linear
 functional on $\cX_{1}^{\star}$ using the $\cX$-pairing:
 \begin{equation}    \label{Xpairing}
   \ell_{x}(x_{1}^{\star}): = \langle x_{1}^{\star}, x \rangle_{\cX} \text{ for }
   x_{1}^\star \in \cX_{1}^\star.
 \end{equation}
 If we write $x_{1}^{\star} = (\overline{\alpha} I - A^{*})^{-1} y$ with
 $y \in \cX$, then
 \begin{align*}
    | \ell_{x}(x_{1}^{\star})| & =| \langle x_{1}^{\star}, x \rangle_{\cX}| \\
     & = |\langle (\overline{\alpha} I - A^{*})^{-1} y, x
     \rangle_{\cX}| \\
     & = |\langle y, (\alpha I - A)^{-1} x \rangle_{\cX}| \\
     & \le \| y \|_{\cX} \| (\alpha I - A)^{-1} x \|_{\cX} \\
     & =\| x_{1}^{\star}\|_{1,\star} \| x\|_{-1}
 \end{align*}
 with equality if  we take $x_{1}^{\star} = (\overline{\alpha} I -
 A^{*})^{-1}(\alpha I-A)^{-1} x$ (or $y = (\alpha I-A)^{-1}x$).  We conclude that the
 linear-functional norm of $\ell_{x}$ is equal to the $\cX_{-1}$-norm of $\cX$:
 $$
 \| \ell_{x} \|_{(\cX_{1}^{\star})^{*}} = \| x  \|_{-1}
 $$
 and $\cX$ can be identified with a subspace of
 $(\cX_{1}^{\star})^{*}$.  It is not difficult to see that this subspace
 is dense and hence, after taking completions, we have that
 $\cX_{-1}$ is naturally isomorphic to the dual space of
 $\cX_{1}^{\star}$ via the $\cX$-pairing \eqref{Xpairing}.
 For our application to infinite-dimensional linear systems, in
 practice the unbounded closed operator $A$ in this construction will
 also be taken to be the generator of a $C_{0}$-semigroup, so as to
 make sense of a differential equation of the form $\frac{dx}{dt}(t)
 = A x(t)$.

 We are now ready to introduce the notion of {\em system node}.

\begin{definition}  \label{D:sysnode}  By a {\em system node}
    $\Sigma$ on the collection of three Hilbert spaces $(\cU, \cX,
    \cY)$, we mean a linear operator
$$ \Sigma: = \begin{bmatrix} A \& B \\ C \& D \end{bmatrix} \colon
\begin{bmatrix} \cX \\ \cU \end{bmatrix} \supset \cD(\Sigma) \to
    \begin{bmatrix} \cX \\ \cY \end{bmatrix}
$$
such that:
\begin{enumerate}
    \item $\Sigma$ is a closed operator with domain $\cD(\Sigma)$ dense
    in $\left[ \begin{matrix} \cX \\
\cU \end{matrix} \right]$,
\item If we define an operator $A$ with domain $\cD(A) = \left\{ x \in \cX
\colon \sbm{ x \\ 0} \in \cD(\Sigma) \right\}$ by
\begin{equation}  \label{defA}
    A x = \Sigma
\sbm{ x \\ 0 } \text{ for } x \in \cD(A),
\end{equation}
then $A$ is the generator of a
$C_{0}$-semigroup on $X$.

     \item Let $A_{|\cX} \colon X \to \cX_{-1}$ be the extension of the
     operator $A \colon \cX_{1} = \cD(A) \to \cX$ (as in \eqref{defA})
     as described in the
     previous paragraph.  Then there is a bounded linear operator $B
     \colon \cU \to \cX_{-1}$ so that we recover the operator $A \&
     B$ as the restriction of the operator $\begin{bmatrix} A_{|\cX}
     & B \end{bmatrix} \colon \sbm{ \cX \\ \cU } \to \cX_{-1}$ to
     $\cD(\Sigma)$:
  $$
   \Sigma = \begin{bmatrix} A_{|\cX} & B \end{bmatrix}|_{\cD(\Sigma)}.
  $$
   \item $C \& D$ is a bounded operator from $\cD(\Sigma)$ to $\cY$
   $$
      C \& D \in \cL(\cD(\Sigma) , \cY)
   $$
 where $\cD(\Sigma)$ carries the graph norm.
\item The domain $\cD(\Sigma)$ of $\Sigma$ is characterized as
 \begin{equation}   \label{defdomain}
  \cD(\Sigma) = \left\{ \begin{bmatrix} x \\ u \end{bmatrix} \in
  \begin{bmatrix} \cX \\ \cU \end{bmatrix} \colon A_{| \cX} x + B u
      \in \cX \right\}.
\end{equation}
 \end{enumerate}
 \end{definition}

 A consequence of this definition of system node is the following
 fact:
\begin{equation}   \label{fact'}
    \text{ \em Given } u \in \cU, \text{ \em there exists } x_{u} \in
    \cD(\Sigma) \text{ \em so that } \sbm{ x_{u} \\ u } \in \cD(\Sigma).
\end{equation}
Indeed, it suffices to take $x_{u} = \left( (\alpha
I-A)_{|\cX}\right)^{-1} B u$. To see this, we use the criterion in
part (5) of Definition \ref{D:sysnode}  to check that this
$\sbm{x_{u}
\\ u}$ is in $\cD(\Sigma)$:
\begin{align*}
    \begin{bmatrix} A_{|\cX} & B \end{bmatrix} \sbm{ x_{u} \\ u } & =
    A_{|\cX} \left( (\alpha I - A)_{|\cX} \right)^{-1} B u + Bu \\
& = ( A - \alpha I)_{|\cX} \left( (\alpha I - A)_{|\cX}
\right)^{-1} B u + \alpha \cdot \left( (\alpha I - A)_{|\cX}
\right)^{-1} B u + Bu
\\
& = \alpha \cdot \left( (\alpha I - A)_{|\cX} \right)^{-1} B u \in
\cX.
\end{align*}

In the definition of system node, we took pains to write the top
component in the form $A \& B$ to indicate that its domain in
$\sbm{ \cX \\ \cU }$ does not split; yet in the end we found a
splitting by extending to a larger space $\sbm{  \cX_{-1} \\ \cU}$
and writing $A \& B$ as the restriction of an extended operator
$\begin{bmatrix} A_{|\cX} & B \end{bmatrix} \colon \sbm{ \cX_{-1}
\\ \cU} \to \cY$ whose domain does split.  Similarly, there is at
least a partial splitting for the operator $C \& D \colon
\cD(\Sigma) \to \cY$. Indeed, we have seen that $X_{1} = \cD(A) =
\left\{ x \in \cX \colon \sbm{ x \\ 0 } \in \cD(\Sigma\right\}$ is
a dense subset of $\cX$. We may therefore define an operator $C
\colon \cX_{1} \to \cY$ by
\begin{equation}   \label{Cdef}
  C x = C \& D \sbm{ x \\ 0} \text{ for } x \in \cX_{1}.
\end{equation}
Since $C\& D$ is a bounded operator from $\cD(\Sigma)$ (graph norm)
to $\cY$, it follows that $C$ so defined is a bounded operator from
$\cX_{1}$ (graph norm induced by the operator $A$) to $\cY$. In
practice we assign no independent meaning to the $D$ in $C \& D$
except under some additional hypotheses (e.g., for the case of a
{\em regular system}---see the paper of Weiss \cite{WeissTAMS94} or
\cite[Section 5.6]{Staffans}). To this point we have
at least versions of all the usual constituents for a linear
input/state/output linear system:
\begin{align}
& A \colon \cX_{1} \to \cX \text{ main operator or state dynamics,}
\notag  \\
& B \colon \cU \to \cX_{-1} \text{ input or control operator,} \notag
\\
& C \colon \cX_{1} \to \cY \text{ output or observation operator,}
\notag \\
& C \& D \colon \cD(\Sigma) \to \cY \text{ combined
observation/feedthrough operator,}
\label{sysops}
\end{align}
 and we lack in general an
independent well-defined feedthrough operator $D$.

The formula for $x_{u}$ in \eqref{fact'} can use any point $w$ in
the connected component (e.g., an appropriate right half plane) of
the resolvent set of $A$ containing $\alpha$. For clarity, let us
write $x_{u}(w) = \left( (wI - A)_{|\cX}\right)^{-1}Bu$ to indicate
the dependence of $x_{u}$ on the point $w$ in the right half
plane.  From the fact that $\sbm{ x_{u}(w) \\ u } \in \cD(\Sigma)
= \cD( C \& D)$ it follows that the expression
$$
T_{\Sigma}(w) \colon u \mapsto  C \& D \begin{bmatrix} x_{u}(w) \\ u
\end{bmatrix} = C \& D  \begin{bmatrix} \left( (wI -
A)_{|\cX}\right)^{-1}Bu \\ u \end{bmatrix}
$$
is well defined and defines the {\em transfer function} of the system
node.    Let $\alpha$ be any fixed point in the resolvent set
of $A$.   Then we can recover the value of the transfer function at
any point $w$ in the same connected component of the resolvent set of $A$ from its value
at the fixed point $\alpha$ according to the recipe
\begin{align}
  T_{\Sigma}(w) u & = \left( T_{\Sigma}(w) - T_{\Sigma}(\alpha) \right)u +
  T_{\Sigma}(\alpha) u \notag \\
  & = C (x_{u}(w) - x_{u}(\alpha) ) + T_{\Sigma}(\alpha) u \notag \\
  & = \left( (\alpha - w) C (w I - A)^{-1} \left((\alpha I - A)_{|\cX}\right)^{-1}B
  + T_{\Sigma}(\alpha) \right) u
  \label{transfunc'}
\end{align}

Conversely, start with any semigroup generator $A$ on $\cX$ with
domain $\cD(A) = \cX_{1}$ with induced Gelfand rigging $\cX_{1}
\subset \cX \subset \cX_{-1}$, any input operator $B \colon \cU \to
\cX_{-1}$, and an output operator $C \colon \cX_{1} \to \cY$ along
with a value $T_{\Sigma}(\alpha) \in \cL(\cU, \cY)$ for the transfer
function at the point $\alpha$. Define
\begin{align}
& \cD(\Sigma) = \left\{ \begin{bmatrix} x \\ u \end{bmatrix} \colon
A_{|\cX} x + B u \in \cX \right\} \text{ with } \notag \\
& A \& B = \begin{bmatrix} A_{|\cX} & B
\end{bmatrix}|_{\cD(\Sigma)}, \quad C \& D \sbm{ x \\ u } = C (x
-x_{u}(\alpha)) + T_{\Sigma}(\alpha) u. \label{sysnodeABCD}
\end{align}
Then $\Sigma$ so defined is a system node with value of its transfer
function at $\alpha$ equal to the prescribed value
$T_{\Sigma}(\alpha)$.  Note here that $x - x_{u}(\alpha)$ is in
$\cX_{1}$ since $\sbm{ x \\ u}$ and $\sbm{ x_{u}(\alpha) \\ u}$ are in
$\cD(\Sigma)$ and hence so also is
\begin{equation} \label{keyprop}
    \begin{bmatrix} x - x_{u}(\alpha) \\ 0 \end{bmatrix} =
    \begin{bmatrix}x \\ u \end{bmatrix} - \begin{bmatrix}
        x_{u}(\alpha) \\ u  \end{bmatrix}
    \end{equation}
resulting in $x - x_{u}(\alpha) \in \cX_{1}$.
Moreover, we recover the transfer function
$T_{\Sigma}$ at a general point from $A,B,C$ and $T_{\Sigma}(\alpha)$ via
the formula \eqref{transfunc'}.
(see \cite[Lemma 2.2]{StaffansMCSS} for more complete details).

Given a system node $\Sigma$, it is possible to make sense of the
associated system of differential equations
\begin{equation}   \label{sysnodesys}
\begin{bmatrix} \dot x(t) \\ y(t) \end{bmatrix} = \Sigma
    \begin{bmatrix} x(t) \\ u(t) \end{bmatrix}, \quad x(0) = x_{0}
\end{equation}
as long as $u \in C^{2}([0, \infty); \cU)$ and $\left[
\begin{matrix} x(0) \\ u(0) \end{matrix} \right] \in
\cD(\Sigma)$ (see \cite[Proposition 2.6]{MSW}).
 Application of the Laplace transform
 $$
 x(t) \mapsto \widehat x(w): = \int_{0}^{\infty} e^{-wt} x(t)\, {\tt
 d}t
 $$
 to the system equations \eqref{sysnodesys} leads us to the
 input-output property of the transfer function \eqref{transfunc'}:
 $$
  \widehat y(w) = C (wI - A)^{-1} x(0) + T_{\Sigma}(w) \widehat u(w)
 $$
 for $w$ with sufficiently large real part.

 We mention that any
 well-posed linear system \eqref{wellposedsys} is the integral form
 of the dynamical system associated with a system
 node (see e.g.~\cite{Staffans});  however there are system nodes for which the associated
 dynamical system \eqref{sysnodesys} is not well-posed (i.e., one or more of the block operators
 ${\mathfrak B}^{t}_{0}$, ${\mathfrak C}^{t}_{0}$, ${\mathfrak
 D}^{t}_{0}$ appearing in \eqref{wellposedsys} fail to exist as
 bounded operators between the appropriate spaces), despite the fact
 that the infinitesimal form of the system equations
 \eqref{sysnodesys} does make sense.

 The following examples of system nodes will be useful
 in the sequel.  In this discussion we make use of Kre\u{\i}n-space
 geometry notions discussion in Section \ref{S:Krein}.

 \begin{example} \label{E:scatconserv}
     {\em Suppose that $\Sigma = \left[ \begin{matrix} A \& B \\ C \& D
 \end{matrix} \right] \colon \left[ \begin{matrix} \cX \\
 \cU \end{matrix} \right] \supset \cD(\Sigma) \to \left[
 \begin{matrix} \cX \\ \cY \end{matrix} \right]$ is a
     closed operator such that its graph space
     $$ \cG(\Sigma) : = \begin{bmatrix} A \& B \\ C \& D \\
     \begin{matrix} I & 0 \end{matrix}  \\ \begin{matrix} 0 & I
     \end{matrix} \end{bmatrix} \cD(\Sigma) \subset \begin{bmatrix}
     \cX \\ \cY \\ \cX \\ \cU \end{bmatrix}
     $$
     is a Lagrangian subspace of ${\boldsymbol \cK}: =  \cX \oplus \cY \oplus \cX \oplus
     \cU$, where ${\boldsymbol \cK}$ is given a Kre\u{\i}n space
     structure using the signature operator
     $$ \cJ = \left[ \begin{matrix} 0 & 0 & I_{\cX} & 0 \\ 0 &
     I_{\cY} & 0 & 0 \\ I_{\cX} & 0 & 0 & 0 \\ 0 & 0 & 0 & -I_{\cU}
     \end{matrix} \right].
     $$
      Then $\Sigma$ is a system node (see
     \cite[Proposition 4.9]{BS}).  In fact,
     Schur-class functions over the right halfplane $
     \Pi$ are characterized as those functions $s$ having a
     realization $s(w) = T_{\Sigma}(w)$ as in \eqref{transfunc'} with a system node $\Sigma$ of
     this form (see \cite[Theorem 4.10]{BS}).
     These systems are also characterized by the energy-balance
     property that the block matrix
     $\sbm{ {\mathfrak A}^{t} &  {\mathfrak B}_{0}^{t} \\ {\mathfrak C}_{0}^{t}
     & {\mathfrak D}_{0}^{t}}$ associated with the integral
     form \eqref{wellposedsys} of the system equations is unitary for each $t$
     (see \cite{Staffans-ND}).  For this reason any such system node is said to be a
     {\em $\Pi$-scattering conservative system node}. In particular,
     any $\Pi$-scattering conservative system is well-posed.
    For more recent information concerning $\Pi$-scattering
     conservative system nodes and the related notion of
     $\Pi$-scattering dissipative system nodes (where the block
     matrix $\sbm{ {\mathfrak A}^{t} &  {\mathfrak
     B}_{0}^{t} \\ {\mathfrak C}_{0}^{t} & {\mathfrak D}_{0}^{t}}$ in
     \eqref{wellposedsys} is assumed only to be contractive), we refer to the
     recent work of Malinen, Staffans, and Weiss
    (\cite{MSW}, \cite{Staffans2013}, \cite{SW2012}).  In particular, the result
     of \cite{Staffans2013} is that a linear operator $S \colon
     \left[ \begin{matrix} \cX \\ \cU \end{matrix} \right]
     \to \left[ \begin{matrix} \cX \\ \cY \end{matrix}
     \right]$ is a $\Pi$-scattering passive system node if and only
     if it is closed with its graph equal to a maximal $\cJ$-negative
     subspace of $\cX \oplus \cY \oplus \cX \oplus \cU$.  Much of
     this work also has a focus of fitting physical examples into
     this framework; a recent accomplishment was to fit Maxwell's
     equations into this framework (see \cite{WS2013}).
      }\end{example}

  \begin{example} \label{E:impconserv}
  {\em  Suppose that  $\Sigma = \left[ \begin{matrix} A \& B \\ C \& D
 \end{matrix} \right] \colon \left[ \begin{matrix} \cX \\
 \cU \end{matrix} \right] \supset \cD(\Sigma) \to \left[
 \begin{matrix} \cX \\ \cU \end{matrix} \right]$ is a
     closed operator  with output space
     $\cY$ taken to be the same as the input space $\cU$ such that
     \begin{enumerate}
     \item
      its graph $\cG(\Sigma)$ is a Lagrangian
     subspace of ${\boldsymbol \cK}$, but where now ${\boldsymbol
     \cK}$ is given the Kre\u{\i}n-space inner product induced by the
     signature operator $\cJ'$ given by
     \begin{equation}   \label{cJ'}
     \cJ' = \left[ \begin{matrix}  0 & 0 & I_{\cX} & 0 \\
          0 & 0 & 0 & -I_{\cU} \\ I_{\cX} & 0 & 0 & 0 \\
      0 & -I_{\cU} & 0 & 0 \end{matrix} \right],
  \end{equation}
      and
   \item for each $u \in \cU$ there is an $x_{u} \in \cX$ so that
   $\left[ \begin{matrix} x_{u} \\ u \end{matrix} \right]$
   is in $\cD(\Sigma)$.
   \end{enumerate}
   Then $\Sigma$ is a system node (see \cite[Proposition 4.11]{BS}).
   In fact, Herglotz functions over the right halfplane  which also satisfy the growth condition at
   infinity
   \begin{equation*}
       \lim_{t \to +\infty} t^{-1} f(t) u = 0 \text{ for each } u \in
       \cU
    \end{equation*}
   are characterized as those functions $f$
   having a realization $f(w) = T_{\Sigma}(w)$ as in
   \eqref{transfunc'} with a system node $\Sigma$ of this form
   (see \cite[Theorem 4.12]{BS}).  The trajectories $(u(t), x(t),
   y(t))$ satisfy the alternative energy-conservation law
   $$
     \| x(t) \|^{2}_{\cX} - \| x_{0} \|^{2}_{\cX} = 2 \int_{0}^{t}
     {\rm Re}\, \langle y(t), u(t) \rangle_{\cU}\, {\tt d}t
   $$
   and $\Sigma$ is called a {\em $\Pi$-impedance-conservative system node} (see
   \cite{Staffans-ND}).
   As the transfer function for a $\Pi$-impedance-conservative system
   node need not be bounded in the right halfplane, it follows that
   $\Pi$-impedance-conservative system nodes need not be well-posed
   in general (see \cite{StaffansMCSS}).
   }\end{example}

   In connection with Example \ref{E:impconserv} we shall have use
   for the following additional fact.

   \begin{corollary}  \label{C:impconserv}  Suppose that $\bY  \colon
       \cD(\bY) \subset \left[ \begin{matrix} \cX \\ \cU
   \end{matrix} \right] \to \left[ \begin{matrix} \cX \\
   \cU \end{matrix} \right]$ is a   closed densely defined
   operator such that
   \begin{enumerate}
       \item $\bY$ is skew-adjoint:  $\bY  = -\bY^{*}$, and
       \item for each $u \in \cU$ there is an $x_{u} \in \cX$ such
       that $\left[ \begin{matrix} x_{u} \\ u \end{matrix}
       \right] \in \cD(\bY)$.
   \end{enumerate}
   Set $\cJ'_{0} = \left[ \begin{matrix} I_{\cX} & 0 \\ 0 & -
   I_{\cU} \end{matrix} \right]$ and set $\Sigma :
   = - \bY \cJ'_{0}$.  Then $\Sigma$ is a
   $\Pi$-impedance-conservative system node as in Example
   \ref{E:impconserv}.

   Conversely, if $\Sigma$ is a $\Pi$-impedance-conservative system node as in Example
   \ref{E:impconserv}, then $\bY:=-\Sigma
   J'_0\colon\cD(\bY)\subset\left[ \begin{matrix} \cX \\ \cU
   \end{matrix} \right] \to \left[ \begin{matrix} \cX \\
   \cU \end{matrix} \right]$ is a closed operator with the dense
   domain
   $$\cD(\bY)=\Big\{\left[ \begin{matrix} x \\ u   \end{matrix} \right] \in \left[ \begin{matrix} \cX \\
   \cU \end{matrix}\right]\colon J'_0\left[ \begin{matrix} x \\ u   \end{matrix}
   \right]=\left[ \begin{matrix} x \\ -u   \end{matrix}
   \right]\in\cD(\Sigma)\Big\}$$
   satisfying conditions (1) and (2).
  \end{corollary}

  \begin{proof}  Given any closed, densely defined operator $\Sigma
      \colon \cD(\Sigma) \subset \left[ \begin{matrix} \cX \\
      \cU \end{matrix} \right] \to \left[ \begin{matrix} \cX \\
      \cU \end{matrix} \right]$, the following translation of
      the definition of adjoint operator to graph spaces is well
      known:
  $$
  \left( \begin{bmatrix} \Sigma \\ I \end{bmatrix} \cD(\Sigma)
  \right)^{\perp} = \begin{bmatrix} I \\ -\Sigma^{*} \end{bmatrix}
  \cD(\Sigma^{*})
  $$
  where here the orthogonal complement is with respect to the
  standard Hilbert space inner product.  More generally,
  compute the $\cJ'$-orthogonal complement (where $\cJ' = \left[
  \begin{matrix} 0 & \cJ'_{0} \\ \cJ'_{0} & 0 \end{matrix}
      \right]$ by the definition \eqref{cJ'} of $\cJ'$) as follows:
      \begin{align*}
\begin{bmatrix} y_{1} \\ y_{2} \end{bmatrix} \in \left(
    \begin{bmatrix} \Sigma \\ I \end{bmatrix} \cD(\Sigma)
    \right)^{\perp \cJ'} & \Leftrightarrow
    \cJ' \begin{bmatrix} y_{1} \\ y_{2} \end{bmatrix} =
    \begin{bmatrix} \cJ_{0}' y_{2} \\ \cJ'_{0} y_{1}
    \end{bmatrix} \in \left( \begin{bmatrix} \Sigma \\ I
    \end{bmatrix} \cD(\Sigma) \right)^{\perp} = \begin{bmatrix} I \\
    - \Sigma^{*} \end{bmatrix} \cD(\Sigma^{*}) \\
 & \Leftrightarrow  \begin{bmatrix} y_{1} \\ y_{2} \end{bmatrix} \in
 \begin{bmatrix} -\cJ'_{0} \Sigma^{*} \cJ'_{0} \\ I \end{bmatrix}
     \cD(\Sigma^* \cJ'_{0}).
 \end{align*}
Thus condition (1) in Example \ref{E:impconserv} for $\Sigma$ to
be a $\Pi$-impedance-conservative system node translates to
$$
\begin{bmatrix} \Sigma \\ I \end{bmatrix} \cD(\Sigma) =
    \begin{bmatrix} - \cJ'_{0} \Sigma^{*} \cJ'_{0} \\ I \end{bmatrix}
    \cD(\Sigma^{*} \cJ'_{0})
$$
or simply
$$
    \Sigma = - \cJ'_{0} \Sigma^{*} \cJ'_{0}.
$$
An equivalent condition is:  {\em the operator $\bY: = \Sigma
\cJ'_{0}$ (or equivalently $-\bY = - \Sigma \cJ'_{0}$) is
skew-adjoint: $\bY = - \bY^{*}$.}

Conversely, if $\bY$ is any
skew-adjoint operator with dense domain in $\left[
\begin{matrix} \cX \\ \cU \end{matrix} \right]$, then
$\Sigma = - \bY \cJ'_{0}$ satisfies condition (1) in Example
\ref{E:impconserv}. Since
 $\left[
\begin{matrix} x \\ u \end{matrix} \right]\in\cD(\Sigma)$ if and
only if  $\left[
\begin{matrix} x\\ -u\end{matrix} \right]\in\cD(\bY)$, condition (2) in Example
\ref{E:impconserv} is equivalent to condition (2) in the statement
of the corollary.
\end{proof}

The next result gives a model for $\Pi$-impedance-conservative system
nodes as described in Corollary \ref{C:impconserv}.

\begin{proposition}   \label{P:impconserv}
    Let $(T, V_{0}, R)$ be a triple of operators such that:
    \begin{enumerate}
    \item
    $T$ is a  densely defined skew-adjoint operator on the
    Hilbert space $\cX$,
    \item  $V_{0} \in \cL(\cU, \cX)$ is a
    bounded linear operator from the input-output space $\cU$ into
    $\cX$, and
    \item $R$ is a bounded skew-adjoint operator on $\cU$.
    \end{enumerate}
    Define an operator $\Sigma \colon \cD(\Sigma) \subset
    \left[ \begin{matrix} \cX \\ \cU \end{matrix} \right]
    \to  \left[ \begin{matrix} \cX \\ \cU \end{matrix}
    \right]$  as follows.  Set
    \begin{equation}   \label{Sig-dom}
    \cD(\Sigma) = \left\{ \begin{bmatrix} x \\ u \end{bmatrix} \in
    \begin{bmatrix} \cX \\ \cU \end{bmatrix} \colon x - V_{0}u \in
    \cD(T) \right\}
\end{equation}
and then define
\begin{equation}  \label{Sig-form}
\Sigma \begin{bmatrix} x \\ u \end{bmatrix} = \begin{bmatrix} T(x
- V_{0}u) + V_{0} u \\ V_{0}^{*} x + V_{0}^{*}T(x - V_{0}u) + R u \end{bmatrix}
\in \begin{bmatrix} \cX \\ \cU \end{bmatrix} \text{ for } \begin{bmatrix}
    x \\ u \end{bmatrix} \in \cD(\Sigma).
\end{equation}
Equivalently,  $\Sigma$ can be defined as the system node
constructed from the data
\begin{align}
    & A = T \in \cL(\cX_{1}, \cX), \quad B = (I-T)V_{0} \in \cL(\cU,
\cX_{-1}), \notag  \\
& C = V_{0}^{*}(I - T)^{*} \in \cL(\cX_{1}, \cU),  \quad
\alpha = 1 \text{ and } T_{\Sigma}(1) = V_{0}^{*} V_{0} + i R \in \cL(\cU)
\label{dataset}
\end{align}
according to the recipe \eqref{sysnodeABCD}.
Then $\Sigma$ is a $\Pi$-impedance-conservative system node.

Conversely, any $\Pi$-impedance-conservative system node
arises in this way from a triple of operators $T, V_{0}, R$
satisfying conditions (1), (2), and (3) above.
\end{proposition}

\begin{remark} \label{R:impconserv}
    {  \em  We note that, in case $T$ is bounded, the operator $\Sigma$
    given by \eqref{Sig-dom} and \eqref{Sig-form} is simply
\begin{equation}  \label{fact}
\Sigma = \begin{bmatrix} I & 0 \\ 0 & V_{0}^{*} \end{bmatrix}
\begin{bmatrix} T & I-T \\ (I-T)^{*} & -T \end{bmatrix}
    \begin{bmatrix} I & 0 \\ 0 & V_{0} \end{bmatrix} +
    \begin{bmatrix} 0 & 0 \\ 0 & R \end{bmatrix}.
\end{equation}
One way to make sense of the first term in this formula for the general case
where $T$ is allowed to be unbounded
is as follows.  We may view $\widetilde \Sigma = \left[
\begin{matrix} T & I - T \\ (I - T)^{*} & -T \end{matrix}
    \right]  =
 \left[ \begin{matrix} T & I - T \\ I + T & -T
\end{matrix}\right]$ as   an operator from $\left[
\begin{matrix} \cX \\ \cX \end{matrix} \right]$ to
 $\left[
\begin{matrix} \cX_{-1} \\ \cX_{-1} \end{matrix} \right]$,
    where $\cX_{-1}$ is the rigged level-($-1$) space associated with
    the skew-adjoint operator $T$ as explained in discussion at the
    beginning of this section.  It is natural to introduce a domain
 \begin{equation}   \label{domain}
    \cD = \left\{
    \begin{bmatrix} x_{1} \\ x_{2} \end{bmatrix} \colon
    \widetilde \Sigma \begin{bmatrix} x_{1} \\ x_{2} \end{bmatrix}
    \in \begin{bmatrix} \cX \\ \cX \end{bmatrix} \right\} =
    \left\{ \begin{bmatrix} x_{1} \\ x_{2} \end{bmatrix} \colon
    T(x_{1} - x_{2}) \in \cX \right\}.
 \end{equation}
and define an operator $\widetilde \Sigma_{0} \colon \cD \subset
\left[ \begin{matrix} \cX \\ \cX \end{matrix} \right] \to \left[
\begin{matrix} \cX \\ \cX \end{matrix} \right]$ by $\widetilde
\Sigma_{0} = \widetilde \Sigma|_{\cD}$. Then we may define the
operator $\Sigma$ via the formula \eqref{fact} with domain
given by
$$
  \cD(\Sigma) = \left\{ \begin{bmatrix} x \\ u \end{bmatrix} \colon
  \begin{bmatrix} I & 0 \\ 0 & V_{0} \end{bmatrix} \begin{bmatrix} x
      \\ u \end{bmatrix} \in \cD(\widetilde \Sigma_{0}) \right\}.
$$
} \end{remark}

\begin{proof}[Proof of Proposition \ref{P:impconserv}]
    We first show that $\Sigma$ defined as in \eqref{Sig-dom},
    \eqref{Sig-form} satisfies conditions (1) and (2) in Corollary
    \ref{C:impconserv}.  As for condition (2), note that if we set
    $x_{u} = V_{0}u$ for each $u \in \cU$, then $x_{u} - V_{0}u = 0
    \in \cD(T)$, so condition (2) is satisfied.  It remains to
    verify condition (1).

    Toward this end, by the representation \eqref{fact} for $\Sigma$
    as explained in Remark \ref{R:impconserv}, it suffices to show
    that the operator
    $$
    \Sigma':= \begin{bmatrix} -I & 0 \\ 0 & I \end{bmatrix} \begin{bmatrix} T
     & I-T \\ I+T & -T + R \end{bmatrix} = \begin{bmatrix} -T & T-I \\
     T+I & -T + R \end{bmatrix}
$$
with domain $\cD$ as in \eqref{domain} is skew-adjoint.  Note that
$\left[ \begin{matrix} y_{1} \\ y_{2}
\end{matrix}\right] \in \cD\left( \left[ \begin{matrix} -T & T-I \\
T+I & -T  + R \end{matrix} \right]^{*} \right)$ means that the
sesquilinear form
\begin{align*}
& \left\langle \begin{bmatrix} -T & T-I \\ T+I & -T + R \end{bmatrix}
\begin{bmatrix} x_{1} \\ x_{2} \end{bmatrix}, \begin{bmatrix} y_{1}
    \\ y_{2} \end{bmatrix} \right\rangle \\
& \quad  = \langle T(x_{1} - x_{2}), -y_{1} + y_{2} \rangle
- \langle x_{2}, y_{1} \rangle + \langle  x_{1}, y_{2} \rangle
+ \langle R x_{2}, y_{2} \rangle
\end{align*}
defined for $\sbm{x_{1} \\ x_{2}}  \in
\cD\left( \sbm{ -T & T-I \\ T + I & -T + R} \right)$  is continuous in the argument
$\sbm{ x_{1} \\ x_{2} }$ in the $\cX \oplus \cX$ norm.
As the second, third, and fourth terms are automatically
    bounded in the $\sbm{ x_{1} \\ x_{2}}$-argument and the element $x_{1} - x_{2}$
is an arbitrary element of $\cD(T)$ (e.g., take $x_{2} = 0$ and
$x_{1} \in \cD(T)$), it follows that necessarily $y_{2} - y_{1} \in
\cD(T^{*}) = \cD(T)$ (i.e., $\sbm{ y_{1} \\
y_{2} } \in \cD$) and the calculation above
continues as
\begin{align*}
& \left\langle \begin{bmatrix} -T & T-I \\ T+I & -T + R \end{bmatrix}
\begin{bmatrix} x_{1} \\ x_{2} \end{bmatrix}, \begin{bmatrix} y_{1}
    \\ y_{2} \end{bmatrix} \right\rangle \\
    & \quad =
    \langle x_{1} - x_{2}, T^{*}(y_{2}-y_{1})\rangle - \langle x_{2}, y_{1}
    \rangle + \langle x_{1}, y_{2} \rangle  + \langle R x_{2}, y_{2}
    \rangle \\
 & = \left\langle \begin{bmatrix} x_{1} \\ x_{2} \end{bmatrix},
 \begin{bmatrix} -Ty_{2} + Ty_{1} + y_{2} \\ T y_{2} - Ty_{1} - y_{1}
     + R^{*} y_{2}  \end{bmatrix} \right\rangle
  = \left\langle \begin{bmatrix} x_{1} \\ x_{2} \end{bmatrix}, \,
  - \begin{bmatrix} -T & T-I \\ T+I & -T  + R \end{bmatrix}
  \begin{bmatrix} y_{1} \\ y_{2} \end{bmatrix} \right\rangle
\end{align*}
where we use $R$ is bounded with $R = -R^{*}$ in the last two steps.
The skew-adjointness of the operator $\Sigma'$ now follows as
wanted. As a consequence of Corollary \ref{C:impconserv} it
follows in particular that $\Sigma$ is a system node. The
equivalence of the original definition \eqref{Sig-form} of
$\Sigma$ with the alternative formulation based on the data set
\eqref{dataset} is a simple consequence of the identities
\eqref{sysnodeABCD} along with computation of the
value of the transfer function at the point $1 \in \Pi$:
$T_{\Sigma}(1) = V_{0}^{*} V_{0} + i R$; this in turn is a routine
verification which we leave to the reader.

For the converse statement, we let $\Sigma' = \left[ \begin{matrix} A' \& B' \\
C' \& D'
\end{matrix} \right]$ be any $\Pi$-impedance-conser\-vative system node.
As $\left[ \begin{matrix} -(A' \& B') \\ C' \& D'
\end{matrix} \right]$ is skew-adjoint, from the duality theory
of system nodes (see e.g.~\cite{Staffans}) one can see that
necessarily $A' = -A^{\prime*}$ and $C' = B^{\prime *}$; here $B'
\in \cL(\cU, \cX'_{-1})$, $C '\in \cL(\cX'_{1}, \cU)$ and
computation of the adjoint $B^{\prime *}$ is with respect to the duality
pairing between $\cX'_{-1}$ and $\cX'_{1}$ via the $\cX'$ pairing
\eqref{Xpairing} (note that $\cX'_{1} = \cX^{\prime \star}_{1}$ since $A' = -
A^{\prime *}$).  Set $T = A'$. As $T$ is skew-adjoint, both $1$
and $-1$ are in the resolvent set of $T$ and we may define a
bounded operator $V_{0} \in \cL(\cU, \cX')$ by $V_{0} = (I-T)^{-1}
B'$. We then have $B' = (I-T) V_{0}$ and $C' = V_{0}^{*} (I+T)$.
We now have all the ingredients to define another
$\Pi$-impedance-conservative system node $\Sigma_0 = \left[
\begin{matrix} I & 0  \\ 0 &  V_{0}^{*} \end{matrix}
    \right] \left[ \begin{matrix} T & I-T \\ I+T & -T
\end{matrix} \right] \left[ \begin{matrix} I & 0 \\ 0 &
V_{0} \end{matrix} \right]$ as in \eqref{Sig-dom},
\eqref{Sig-form} (with $R$ taken equal to zero).
When we write $\Sigma_{0}$ in the form
$$
  \Sigma_{0} = \begin{bmatrix} A_{0} \& B_{0} \\ C_{0} \& D_{0}
\end{bmatrix},
$$
we see that the construction gives that $\cD(\Sigma') =
\cD(\Sigma_{0})$, i.e., $\cD( A_{0} \& B_{0}) = \cD(A' \& B')$, with
$A_{0} \& B_{0} = A' \& B'$, so
\begin{align*}
& A_{0} : = A_{0} \& B_{0}|_{\cD(\Sigma_{0}) \cap
\left[ \begin{smallmatrix} \cX \\ 0 \end{smallmatrix} \right] } =
A' \& B'|_{\cD(\Sigma_{0}) \cap
\left[ \begin{smallmatrix} \cX \\ 0 \end{smallmatrix} \right]} =: A',
\quad B_{0} = B' \in \cL(\cU, \cX_{-1}), \\
&  C_{0} : = C_{0} \& D_{0}|_{\cD(\Sigma_{0}) \cap
\left[ \begin{smallmatrix} \cX \\ 0 \end{smallmatrix}
\right]} = C' \& D'|_{\cD(\Sigma) \cap
\left[ \begin{smallmatrix} \cX \\ 0 \end{smallmatrix}
\right]} =: C'.
\end{align*}
As observed in \eqref{keyprop}, for any fixed choice of $\alpha \in
\Pi$ (e.g., $\alpha = 1$) an element $\left[ \begin{smallmatrix} x \\ u \end{smallmatrix}
\right] \in \cD(\Sigma')$ can be decomposed as
$$
\begin{bmatrix} x \\ u \end{bmatrix} = \begin{bmatrix} x - (\alpha I
    - A)^{-1} B u \\ 0 \end{bmatrix} + \begin{bmatrix} (\alpha I -
    A)^{-1} B u \\ u \end{bmatrix}
$$
where each summand is again in $\cD(\Sigma') = \cD(\Sigma_{0})$.  Then
making use of the formulas \eqref{sysnodeABCD} we compute
$$
 \Sigma \begin{bmatrix} x \\ u \end{bmatrix} = \begin{bmatrix} A x +
 B u \\ C( x - (\alpha I - A)^{-} B u) \end{bmatrix} + \begin{bmatrix}
 0 \\ T_{\Sigma}(\alpha) u \end{bmatrix}
$$
and similarly
$$
  \Sigma_{0} \begin{bmatrix} x \\ u \end{bmatrix} = \begin{bmatrix}
  Ax + B u \\ C (x - (\alpha I - A)^{-1} B u) \end{bmatrix} +
  \begin{bmatrix} 0 \\ T_{\Sigma_{0}}(\alpha) u \end{bmatrix}
$$
from which we read off
$$
\Sigma \begin{bmatrix} x \\ u \end{bmatrix} = \Sigma_{0}
\begin{bmatrix} x \\ u \end{bmatrix} + \begin{bmatrix} 0 \\ \left(
    T_{\Sigma}(\alpha) - T_{\Sigma_{0}}(\alpha) \right) u
\end{bmatrix}.
$$
Thus $\Sigma = \Sigma_{0} + \left[ \begin{smallmatrix} 0 & 0 \\ 0 & R
\end{smallmatrix} \right]$ with $R: = T_{\Sigma}(\alpha) -
T_{\Sigma_{0}}(\alpha)$ equal to a bounded operator on $\cU$. As both
$ \Sigma \left[ \begin{smallmatrix} I & 0 \\ 0 & -I \end{smallmatrix}
\right]$ and $ \Sigma_{0} \left[ \begin{smallmatrix} I & 0 \\ 0 & -I \end{smallmatrix}
\right]$ are skew-adjoint, it is necessarily the case that $R =
-R^{*}$ as well.
\end{proof}

 We need a generalization of the formula for the transfer function
 \eqref{transfunc'} to the setting where the resolvent
 $(w I - A)^{-1}$ is replaced by the structured resolvent $( Y(w) -
 A)^{-1}$.  For the statement of this result it is convenient to
 assume that $A$ is {\em maximal dissipative}.
 Recall from the discussion in Example \ref{E:skewadjoint}
 above that a densely defined closed Hilbert-space operator $A$ is
 maximal dissipative if it  is dissipative (${\rm Re} \langle Ax, x
 \rangle \le 0$ for all $x \in \cD(A)$) and $I+ A$ is onto; it then
 follows that $wI + A$ is onto for all $w \in \Pi$ and that  $A$
 is the generator of a contractive semigroup (see \cite{Phillips59}).

 \begin{proposition}   \label{P:strucresol}
     Let $A$ be a maximal dissipative operator on the Hilbert space
     $\cX$ (and hence $A$ is the generator of a contractive
     $C_{0}$-semigroup with resolvent set containing the right
     halfplane $\Pi$), and suppose that $(Y_{1}, \dots, Y_{d})$ is a
     $d$-fold positive decomposition of $I_{\cX}$.  For $w = (w_{1},
     \dots, w_{d}) \in {\mathbb C}^{d}$, we set $Y(w) = w_{1} Y_{1} +
     \cdots + w_{d} Y_{d}$.  Then the following observations hold
     true:
     \begin{enumerate}
     \item
     For $w \in \Pi^{d}$, $Y(w) - A$ is
     invertible with
     \begin{equation}  \label{est}
     \|( Y(w) - A)^{-1} \| \le \frac{1}{ \operatorname{min}_{j}
     \operatorname{Re} w_{j} }.
     \end{equation}

     \item
   For $w \in \Pi^{d}$, the operator $(Y(w) - A)^{-1}$ is a
   bicontinuous bijection from $\cX$ onto $\cX_{1}$.

   \item For $w \in \Pi^{d}$, the operator $Y(w) - A \colon
   \cX_{1} \to \cX$ has an extension
   $$
   (Y(w) - A)_{|\cX} \colon \cX \to \cX_{-1}
   $$
   which is a bicontinuous bijection from $\cX$ onto $\cX_{-1}$.

    \item For any system node $\Sigma = \left[ \begin{matrix} A
    \& B \\ C \& D \end{matrix} \right]$ containing $A$ as its
    state operator/semi\-group generator, it holds that
    $$  \begin{bmatrix} ( (Y(w) - A)_{_{|\cX}})^{-1} B u \\ u
\end{bmatrix} \in \cD(\Sigma).
$$
\end{enumerate}
\end{proposition}

    \begin{proof}

    \textbf{(1)}  For $w = (w_{1}, \dots, w_{d}) \in \Pi^{d}$ and $x
    \in \cD(A)$, we have
    \begin{align}
        \| (Y(w)  - A) x \|  \cdot \|x\| & \ge | \langle (Y(w) - A) x, x
        \rangle | \notag  \\
        & \ge \operatorname{Re} \langle (Y(w) - A) x, x \rangle
        \notag \\
        & \ge \operatorname{Re} \langle Y(w) x, x \rangle \notag  \\
        & = \sum_{k=1}^{d} \langle ( \operatorname{Re} w_{k})
        Y_{k} x, x \rangle  \notag \\
        & \ge \operatorname{min}_{j} ( \operatorname{Re} w_{j} )
        \sum_{k=1}^{d} \langle Y_{k} x, x \rangle  \notag \\
        & = \operatorname{min}_{j} ( \operatorname{Re} w_{j} )
        \| x\|^{2}.
        \label{compute}
    \end{align}
Thus $Y(w) - A$ is bounded below and has a left inverse.  A similar
argument with $Y(w) - A$ replaced by $(Y(w) - A)^{*}$ shows that
$Y(w) - A$ also has a right inverse.  The computation  \eqref{compute}
gives the estimate \eqref{est}.

\smallskip

\textbf{(2)}
Next note that for $w \in \Pi^{d}$ and $x \in \cX$,
$$
  A (Y(w) - A)^{-1} x = - x + Y(w) (Y(w) - A)^{-1} x  \in \cX
$$
from which we conclude that $(Y(w) - A)^{-1}$ maps $\cX$ into
$\cX_{1}$.  Conversely if $x \in \cX_{1}$, then $ y = (Y(w) - A) x \in \cX$
and we recover $x$ as $x = (Y(w) - A)^{-1} y$.  We conclude that
$(Y(w) - A)^{-1}$ maps $\cX$ bijectively to $\cX_{1}$.  The fact that
$(Y(w) - A)^{-1}$ is bicontinuous is then a consequence of the open
mapping theorem.

\smallskip

\textbf{(3)}  We use the dual version of a result from part (2):  for
$w \in \Pi^{d}$,
$$
Y(w)^{*} - A^{*} \colon \cX_{1}^\star \to \cX
$$
is a bicontinuous bijection from $\cX_{1}^\star$ (with the
$(1,\star)$-norm) onto $\cX$.
    If we then take adjoints with respect
to the $\cX$-pairing, we get an operator
$$
(Y(w)^{*} - A^{*})^{*} : \cX \to \cX_{-1}
$$
which must also be a bicontinuous bijection, but now from $\cX$ to
$\cX_{-1}$.  Since the duality is with respect to the
$\cX$-pairing \eqref{Xpairing}, it is clear that this map provides
an extension of the operator $Y(w) - A \colon \cX_{1} \to \cX$:
$$
(Y(w) - A)_{|\cX} : = ( Y(w)^{*} - A^{*})^{*} \colon \cX \to
\cX_{-1}.
$$

\textbf{(4)} To show that $\left[ \begin{matrix} ( (Y(w) -
A)_{|\cX})^{-1} B u \\ u \end{matrix} \right]$ is in
$\cD(\Sigma)$, by the characterization of $\cD(\Sigma)$ in
Definition \ref{D:sysnode} we need only show that
\begin{equation}  \label{toshow}
  \begin{bmatrix} A_{|\cX} & B \end{bmatrix} \begin{bmatrix} ( (Y(w)
- A)_{|\cX})^{-1} B u \\ u
\end{bmatrix} \in \cX.
\end{equation}
 As $Bu \in \cX_{-1}$, $(Y(w) - A)^{-1}$ maps $\cX_{-1}$ to $\cX$ by
 part (3) and $Y(w)$ is a bounded operator on
$\cX$, we conclude that
$$
  Y(w) ( (Y(w) - A)_{|\cX})^{-1} B u \in \cX.
$$
We now rewrite the expression in \eqref{toshow} as
 \begin{align*}
   & \begin{bmatrix} A_{|\cX} & B \end{bmatrix} \begin{bmatrix} (
(Y(w) - A)_{|\cX})^{-1} B u \\ u
\end{bmatrix} \\
& = [ (A - Y(w))_{|\cX} + Y(w)] ( (Y(w) -
A)_{|\cX})^{-1} B u + Bu \\
& = Y(w) ( (Y(w) - A)_{|\cX})^{-1} B u
\end{align*}
to conclude that  \eqref{toshow} holds as desired.
\end{proof}

\begin{remark}  \label{R:sysnodes} {\em
    We note that as a consequence of property (4) in Proposition
    \ref{P:strucresol} it is possible to define the transfer function associated
    with the structured resolvent $(Y(w) - A)^{-1}$ via
    $$
    T_{\Sigma, \{Y_{k}\}}(w) = [ C \& D ] \ \begin{bmatrix}
   \left( (Y(w) - A)_{|\cX} \right)^{-1} Bu \\ u \end{bmatrix}.
   $$
   It is tempting to view this as the transfer function for a
   multidimensional linear system
   \begin{equation}   \label{CT-FM} \begin{array}{rcl}
   \sum_{k=1}^{d} Y_{k}  \frac{ \partial x}{\partial t_{k}}(t) & = & A x(t) + B u(t) \\
     y(t) & = & C x(t) + D u(t).  \end{array}
    \end{equation}
   where $t = (t_{1}, \dots, t_{d})$ (a continuous-time version of a
   multidimensional linear system of Fornasini--Marchesini type---see
   \cite{FM}).  However, in general it is not clear how to extend the
   operators $Y_{1}, \dots, Y_{d} \in \cL(\cX)$ to operators
   ${Y_{k}}_{|\cX_{-1}}$ in a sensible way.
   Special cases where this is possible are: (1) the case where $B$ and
   $C$ are bounded, i.e.,  $B \in \cL(\cU, \cX)$ and $C \in \cL(\cX, \cY)$:
   in this situation the system
  \eqref{CT-FM} makes sense with state space taken simply to be
  $\cX$, and (2) in case the operators $Y_{k}$ commute with $A$: in this
  case one can extend $Y_{k}$ to $\cX_{-1}$ via the formula
  $$
    {Y_{k}}_{|\cX_{-1}} = (\alpha I - A) Y_{k} (\alpha I - A)^{-1}
    \colon \cX_{-1} \to \cX_{-1}.
  $$
  We note that continuous-time counterparts of Fornasini--Marchesini
  models of a somewhat different form have been considered in the
  literature (see \cite{BM} and the references there).
}\end{remark}

  \section{The Herglotz--Agler class over the polydisk}  \label{S:HA-D}

  In this section we present our results for the Herglotz--Agler class
  over the polydisk ${\mathbb D}^{d}$.

  \begin{theorem}  \label{T:HA-D}
      Given a function $F \colon {\mathbb D}^{d} \to \cL(\cU)$, the
      following are equivalent.

      \begin{enumerate}
      \item $F$ is in the Herglotz--Agler class
      $\mathcal{HA}({\mathbb D}^{d}, \cL(\cU))$.

      \item $F$ has a ${\mathbb D}^{d}$-Herglotz--Agler
      decomposition, i.e., there exist $\cL(\cU)$-valued positive
      kernels $K_{1}, \dots, K_{d}$ on ${\mathbb D}^{d} $ such that
 \begin{equation}   \label{Dd-HerAgdecom}
     F(\omega)^{*} + F(\zeta) = \sum_{k=1}^{d} (1 - \overline{\omega}_{k}
     \zeta_{k}) K_{k}(\omega, \zeta).
  \end{equation}

  \item There exist a Hilbert space $\cX$, a $d$-fold
  spectral decomposition $(P_{1}, \dots, P_{d})$ of $I_{\cX}$ with
  associated operator pencil $P(\zeta) = \zeta_{1}P_{1} + \cdots +
  \zeta_{d} P_{d}$ (see Section \ref{S:decomid}),  and a bounded colligation matrix
  $$
  \bU = \begin{bmatrix} A & B \\ C & D \end{bmatrix} \colon
  \begin{bmatrix} \cX \\ \cU \end{bmatrix} \to \begin{bmatrix} \cX \\
      \cU \end{bmatrix}
   $$
  with block matrix entries $A,B,C,D$ satisfying the relations
  \begin{equation}   \label{id}
      A^{*}A = AA^{*} = I_{\cX}, \quad B = AC^{*},
      \quad D+D^{*} = CC^{*}\ (= B^{*}B)
  \end{equation}
  such that
  \begin{equation} \label{DdHerreal}
      F(\zeta) = D + C (I - P(\zeta) A)^{-1} P(\zeta) B.
  \end{equation}

\end{enumerate}
\end{theorem}

\begin{proof}

   (1)$\Rightarrow$(2):  Use that $F \in \mathcal{HA}({\mathbb
    D}^{d}, \cL(\cU))$ if and only if  $S(\zeta): = [F(\zeta) - I]
[ F(\zeta) + I]^{-1}$ is in the Schur--Agler class
    $\mathcal{SA}({\mathbb D}^{d}, \cL(\cU))$.  Then $S$ has a
    Schur--Agler decomposition
    $$
    I - S(\omega)^{*} S(\zeta) = \sum_{k=1}^{d} (1 - \overline{\omega}_{k}
    \zeta_{k}) \widetilde K_{k}(\omega, \zeta)
    $$
    for $\cL(\cU)$-valued positive kernels $\widetilde K_{1}, \dots,
    \widetilde K_{d}$ on ${\mathbb D}^{d}$.
    A routine computation gives
    $$
    I - S(\omega)^{*} S(\zeta) = 2 [ F(\omega)^{*} + I]^{-1}
    \left(F(\omega)^{*} + F(\zeta) \right) [ F(\zeta) + I]^{-1}.
    $$
    This leads us to the ${\mathbb D}^{d}$-Herglotz--Agler
    decomposition \eqref{Dd-HerAgdecom} with
    $$
    K_{k}(\omega, \zeta) = \frac{1}{2} [F(\omega)^{*} + I]
   \,  \widetilde K_{k}(\omega, \zeta) \,  [F(\zeta) + I].
    $$
Notice that Agler in \cite{Agler1990} proved the implications
(1)$\Rightarrow$(2) in Theorems \ref{T:Agler} and \ref{T:HA-D}
simultaneously, while we show here that they are, in fact,
equivalent.

    \smallskip
    (2)$\Rightarrow$(3):  Since each kernel $K_{k}$ is positive, each
    $K_{k}$ has a Kolmogorov decomposition, i.e., there is a function
    $H_{k} \colon {\mathbb D}^{d} \to \cL(\cU, \widetilde \cX_{k})$
    (for some auxiliary Hilbert space $\widetilde \cX_{k}$) so that
    we have the factorization
    $K_{k}(\omega, \zeta) = H_{k}(\omega)^{*} H_{k}(\zeta)$.  We let
    $H(\zeta) = \operatorname{Col}_{k=1, \dots, d} [ H_{k}(\zeta)]$
    be the associated block column matrix function defining a function on
    ${\mathbb D}^{d}$ with values in $\cL(\cU, \widetilde \cX)$,
    where we set $\widetilde \cX = \widetilde \cX_{1} \oplus \cdots
    \oplus \widetilde \cX_{d}$ (written as columns).  We set
$P(\zeta)
    =\left[  \begin{smallmatrix}  \zeta_{1} I_{\widetilde \cX_{1}} &
    & \\ & \ddots & \\ & & \zeta_{d} I_{\widetilde \cX_{d}}
\end{smallmatrix} \right]$.
    We now may rewrite the Agler decomposition
    \eqref{Dd-HerAgdecom} in the form
    \begin{equation}   \label{Agdecom'}
    F(\omega)^{*} + F(\zeta) = H(\omega)^{*} (I_{\widetilde \cX} -
    P(\omega)^{*} P(\zeta) ) H(\zeta).
    \end{equation}
    We consider the subspace
    $$
    \widetilde{\cG} = \overline{\operatorname{span}} \left\{ \begin{bmatrix}
    H(\zeta) \\ F(\zeta) \\ P(\zeta) H(\zeta) \\ I_{\cU}
\end{bmatrix} u \colon u \in \cU, \, \zeta \in {\mathbb D}^{d}
 \right\} \subset \begin{bmatrix} \widetilde \cX \\ \cU \\ \widetilde
\cX \\ \cU \end{bmatrix}
$$
where the ambient space $\widetilde \cX \oplus \cU \oplus
\widetilde \cX \oplus \cU$ is given the Kre\u{\i}n-space inner
product induced by the signature matrix $\widetilde{\cJ}$ given by
$$
\widetilde{\cJ} = \begin{bmatrix} -I_{\widetilde \cX} & 0 & 0 & 0 \\
      0 & 0 & 0 & I_{\cU} \\ 0 & 0 & I_{\widetilde \cX} & 0 \\ 0 &
      I_{\cU} & 0 & 0 \end{bmatrix}.
$$
Then one can check
\begin{align*}
  &  \left\langle \begin{bmatrix} -I_{\widetilde \cX} & 0 & 0 & 0 \\
0
    & 0 & 0 & I_{\cU} \\ 0 & 0 & I_{\widetilde \cX} & 0 \\ 0 &
    I_{\cU} & 0 & 0 \end{bmatrix}
    \begin{bmatrix} H(\zeta) \\ F(\zeta) \\ P(\zeta) H(\zeta) \\
    I_{\cU} \end{bmatrix} u, \,
    \begin{bmatrix} H(\omega) \\ F(\omega) \\  P(\omega) H(\omega) \\
        I_{\cU} \end{bmatrix} u' \right\rangle   \\
        & \quad =
    \left\langle \left[ -H(\omega)^{*} H(\zeta) + F(\omega)^{*} +
    H(\omega)^{*} P(\omega)^{*} P(\zeta) H(\zeta) +
    F(\zeta)\right] u,\, u' \right\rangle = 0
\end{align*}
where the last equality follows as a consequence of the Agler
decomposition \eqref{Agdecom'}.  We conclude that
$\widetilde{\cG}$ is a $\widetilde{\cJ}$-isotropic subspace.  We
then extend $\widetilde{\cG}$ to a $\cJ$-Lagrangian subspace
$\cG$, where the ambient space $\widetilde \cX \oplus \cU \oplus
\widetilde \cX \oplus \cU$ extended to a space of the form $\cX
\oplus \cU \oplus \cX \oplus \cU$ with $\cX \supset \widetilde
\cX$ and the Kre\u{\i}n-space Gramian matrix $\cJ$ of the same
block form as $\widetilde \cJ$ above.

We claim that
\begin{equation*}
    \cG \cap \left[ \begin{smallmatrix} \cX \\ \cU \\ \{0\} \\ \{0\}
    \end{smallmatrix} \right] = \{0\}.
 \end{equation*}
 Indeed, suppose that $\left[ \begin{smallmatrix} x \\ u \\ 0 \\ 0
 \end{smallmatrix} \right] \in \cG$.  As $\cG$ is isotropic, we then
 must have, for all $u' \in \cU$ and $\zeta \in {\mathbb D}^{d}$,
 $$
  0 = \left\langle \cJ \left[ \begin{smallmatrix} x \\ u \\ 0 \\ 0
  \end{smallmatrix} \right], \left[ \begin{smallmatrix} H(\zeta) \\
  F(\zeta) \\
  P(\zeta) H(\zeta) \\ I_{\cU} \end{smallmatrix} \right]
  u' \right\rangle = \langle -H(\zeta)^{*} P_{\widetilde \cX} x + u,
  u' \rangle_{\cU}.
  $$
  Hence $u = H(\zeta)^{*} P_{\widetilde \cX} x$ for all $\zeta \in
  {\mathbb D}^{d}$; in particular, $u = H(0)^{*} P_{\widetilde \cX}
  x$.  Thus our element of $\cG$ has the form $\left[
  \begin{smallmatrix} x \\ H(0)^{*} P_{\widetilde \cX} x \\ 0 \\ 0
      \end{smallmatrix} \right]$.  We must also have
  $$
  0 = \left\langle \cJ \left[ \begin{smallmatrix} x \\ H(0)^{*}
  P_{\widetilde \cX} x \\ 0 \\ 0 \end{smallmatrix} \right], \left[
  \begin{smallmatrix}  x \\ H(0)^{*} P_{\widetilde \cX}x \\ 0 \\ 0
      \end{smallmatrix} \right] \right\rangle = -\|x\|^{2}_{\cX}
  $$
    which enables us to conclude that $x=0$ and hence also $\left[
    \begin{smallmatrix} x \\ H(0)^{*}P_{\widetilde \cX}x \\ 0 \\ 0
    \end{smallmatrix} \right] = 0$, and the claim follows.

    We are now able to conclude that  $\cG$ is a graph space:
$$
\cG = \left\{ \begin{bmatrix} A \& B \\ C \& D \\ \begin{matrix}
I_{\cX} & 0 \\ 0 & I_{\cU} \end{matrix} \end{bmatrix}
\begin{bmatrix} x \\ u
\end{bmatrix} \colon \begin{bmatrix} x \\ u \end{bmatrix} \in \cD
\right\}
$$
for some closed linear operator $\left[ \begin{smallmatrix} A \& B \\
C \& D \end{smallmatrix} \right]$ with domain $\cD \subset \left[
\begin{smallmatrix} \cX \\ \cU \end{smallmatrix} \right]$.
    Furthermore, by construction the vector $\sbm{ P(\zeta) H(\zeta)
    u \\ u}$ is in $\cD$ for all $\zeta \in {\mathbb D}^{d}$ and $u
    \in \cU$; in particular, by setting $\zeta = 0 \in {\mathbb
    D}^{d}$ we see that $\sbm{ 0 \\ u } \in \cD$ for all $u \in \cU$
    and hence $\cD$ splits:  $\cD = \sbm{ \cD_{1} \\ \cU }$ for some
    linear manifold $\cD_{1}  \subset \cX$.  We are now in position
    to apply Lemma 3.4 from \cite{BS} to conclude that in fact $\left[
    \begin{smallmatrix} A \& B \\ C \& D \end{smallmatrix} \right] =
    \left[ \begin{smallmatrix} A & B \\ C & D \end{smallmatrix}
    \right]$ is bounded with domain equal to all of $\left[
    \begin{smallmatrix} \cX \\ \cU \end{smallmatrix} \right]$,
        and moreover the identities \eqref{id} hold.   We have
        now produced a colligation matrix $\bU = \left[
        \begin{smallmatrix} A & B \\ C & D \end{smallmatrix}
        \right]$ so that
$$
\cG = \begin{bmatrix} A & B \\ C & D \\ I_{\cX} & 0 \\ 0 & I_{\cU}
\end{bmatrix} \begin{bmatrix} \cX \\ \cU \end{bmatrix}.
$$
In particular, it follows that, for any $u \in \cU$, there is a
corresponding $\left[ \begin{smallmatrix} x' \\ u' \end{smallmatrix}
\right]$ in $\left[ \begin{smallmatrix} \cX \\ \cU \end{smallmatrix}
\right]$ so that
\begin{equation}   \label{id1}
 \begin{bmatrix} H(\zeta) \\ F(\zeta) \\ P(\zeta)
H(\zeta) \\ I_{\cU} \end{bmatrix}  u =
\begin{bmatrix} A & B \\ C & D \\ I_{\cX} & 0 \\ 0 & I _{\cU}
\end{bmatrix}   \begin{bmatrix} x' \\ u'
\end{bmatrix}.
\end{equation}
From the bottom two components of \eqref{id1} we read off
$$
 x' = P(\zeta) H(\zeta) u, \quad u' = u.
$$
Then the top two components of \eqref{id1} give
$$
  \begin{bmatrix} A & B \\ C & D \end{bmatrix} \begin{bmatrix}
  P(\zeta) H(\zeta) u \\ u \end{bmatrix} = \begin{bmatrix}
  H(\zeta) u \\ F(\zeta) u \end{bmatrix} \text{ for all } u \in \cU
$$
which we rewrite as a linear system of operator equations
\begin{equation}   \label{colsys}
    \left\{ \begin{array}{rcl}
    A P(\zeta) H(\zeta) + B & = & H(\zeta) \\
    C P(\zeta) H(\zeta) + D & = & F(\zeta). \end{array}  \right.
\end{equation}
Solving the first equation in \eqref{colsys} for $H(\zeta)$ gives
$$
   H(\zeta) = (I_{\cX} - A P(\zeta) )^{-1} B.
$$
(We note that the inverse on the right hand side exists since $A$ is
unitary and $\| P(\omega) \| < 1$ for $\omega \in {\mathbb D}^{d}$.)
Plugging this last expression into the second of equations
\eqref{colsys} then yields
\begin{align*}
    F(\zeta) & = D + C P(\zeta) (I - A P(\zeta))^{-1} B \\
    & = D + C (I - P(\zeta) A)^{-1} P(\zeta) B
 \end{align*}
 and condition (3) in the statement of Theorem \ref{T:HA-D} follows.

 \smallskip
 (3)$\Rightarrow$(2):  We assume that $F$ has the representation
 \eqref{DdHerreal} where the coefficient matrices $A, B, C, D$
 satisfy the relations \eqref{id}.  Then we compute
 \begin{align*}
   &  F(\zeta)^{*} + F(\omega)  =
     D^{*} + B^{*}(I - P(\zeta)^{*}A^{*})^{-1} P(\zeta)^{*} C^{*} + D
     + C P(\omega) (I - A P(\omega))^{-1} B   \\
     & \quad  = [D + D^{*}] + B^{*} (I - P(\zeta)^{*} A^{*})^{-1}
     P(\zeta)^{*} A^{*} B + B^{* } A P(\omega) (I - A P(\omega))^{-1}
     B \\
     & \quad =B^{*} (I - P(\zeta)^{*} A^{*})^{-1} X (I - A
     P(\omega))^{-1} B
 \end{align*}
 where we have set $X$ equal to
 \begin{align*}
     X &  = (I - P(\zeta)^{*} A^{*}) (I - A P(\omega)) + P(\zeta)^{*}
     A^{*} (I - A P(\omega)) + (I - P(\zeta)^{*} A^{*}) A P(\omega)
\\
     & = I - P(\zeta)^{*} A^{*} - A P(\omega) + P(\zeta)^{*}
     P(\omega)  \\
     & \quad + P(\zeta)^{*} A^{*} - P(\zeta)^{*} P(\omega) + A
     P(\omega) - P(\zeta)^{*} P(\omega) \\
     & = I - P(\zeta)^{*} P(\omega).
  \end{align*}
  We conclude that the Agler decomposition \eqref{Dd-HerAgdecom}
  holds with
  $$
   K_{k}(\zeta, \omega) = B^{*} (I - P(\zeta)^{*} A^{*})^{-1} P_{k}
   (I - A P(\omega))^{-1} B,
   $$
   where $P_k=P_{\widetilde{\cX}_k}$ is the orthogonal projection
   of $\widetilde{\cX}$ onto $\widetilde{\cX}_k$,
   and condition (2) in Theorem \ref{T:HA-D} follows.

   \smallskip

   (2)$\Rightarrow$(1):  Given an Agler decomposition
   \eqref{Dd-HerAgdecom}, we may rewrite it in the form
   \eqref{Agdecom'}.
   From the proof of
   (2)$\Rightarrow$(3)$\Rightarrow$(2) we see that we may assume that
   $H(\zeta)$ is holomorphic in $\zeta \in {\mathbb D}^{d}$.  If $T =
   (T_{1}, \dots, T_{d})$ is a commutative $d$-tuple of strict
   contractions on $\cK$, it is straightforward to verify that the
   formula \eqref{Agdecom'} leads to
   $$
   F(T)^{*} + F(T) = H(T)^{*} \left( I_{\cX \otimes \cK} - P(T)^{*}
   P(T) \right) H(T).
   $$
   From the diagonal form of $P(\zeta) = \sum_{k=1}^{d} \zeta_{k}
   P_{k}$, we see that $P(T) = \sum_{k=1}^{d} P_{k} \otimes T_{k}$
has
   $\| P(T) \| < 1$.  Hence $F(T)^{*} + F(T) \ge 0$ and we conclude
   that $F \in \mathcal{HA}({\mathbb D}^{d}, \cL(\cU))$.
  \end{proof}

  \begin{remark}  \label{R:HAdisk}  {\em  Given a representation for
      $F \in \mathcal{HA}({\mathbb D}^{d}, \cL(\cU))$ as in
      \eqref{DdHerreal} and \eqref{id}, let us separate out the selfadjoint and
skew-adjoint parts of $F(0) = D$ to rewrite the formula
      \eqref{DdHerreal} as
    $$
    F(\zeta) = {\rm Re}\,F(0) + C (I - P(\zeta) A)^{-1}
    P(\zeta) B + R
    $$
    where ${\rm Re}\,F(0)=\frac{1}{2}(D+D^*)$ and we set  $R = \frac{1}{2} (D - D^{*}) = - R^{*}$.
    From the relations \eqref{id} we see that
       \begin{align}
    F(\zeta) - R & =
    \frac{1}{2} B^{*}\left(I + 2A (I - P(\zeta) A)^{-1}
    P(\zeta) \right) B   \notag \\
& = \frac{1}{2} B^{*} \left( I + 2A P(\zeta) (I - A P(\zeta))^{-1} \right) B   \notag \\
    & = V^{*}(I - A
    P(\zeta))^{-1} (I + A P(\zeta)) V,   \label{DdHerreal2'}
    \end{align}
 where  $V:= \frac{1}{\sqrt{2}} B$ is such that
 $V^*V=  \operatorname{Re} F(0)$.

 We note that Agler \cite{Agler1990} obtained the representation
 \eqref{DdHerreal2'} for a function in the Herglotz--Agler class
 $\mathcal{HA}({\mathbb D}^{d}, \cL(\cU))$ starting with the realization
 $S(\zeta) = \bD + \bC(I - P(\zeta) \bA)^{-1} P(\zeta) \bB
 $ (with $\bU = \left[ \begin{smallmatrix} \bA & \bB \\ \bC & \bD
\end{smallmatrix} \right]$ unitary as in part (3) of Theorem
\ref{T:Agler})
for the associated function
\begin{equation}\label{E:F-S}
  S(\zeta) = (F(\zeta) - I) (F(\zeta) + I)^{-1}
\end{equation}
in the Schur--Agler class $\mathcal{SA}({\mathbb D}^{d},
\cL(\cU))$ as follows.  The fact that $S$ arises as the Cayley
transform of a function in the Herglotz--Agler class implies that
$I - S(\zeta) = 2 (F(\zeta) + I)^{-1}$ is injective:  in
particular, $\bD = S(0)$ has the property that $I-\bD$ is
invertible; moreover, the fact that $\bU$ is unitary implies that
the $\bU_{0}:= \bA + \bB(I - \bD)^{-1} \bC$ is also unitary. This
follows from the two identities, \begin{align*}
\begin{bmatrix}
\bA & \bB  \\ \bC & \bD
\end{bmatrix}
\begin{bmatrix}
I \\ (I-\bD)^{-1} \bC
\end{bmatrix} &=\begin{bmatrix}
\bU_{0}\\
(I-\bD)^{-1}\bC
\end{bmatrix},\\ \begin{bmatrix}
\bA^* & \bC^*\\
\bB^* & \bD^*
\end{bmatrix}\begin{bmatrix}
I\\
(I-\bD^*)^{-1} \bB^*
\end{bmatrix} &=\begin{bmatrix}
\bU_{0}^*\\
(I-\bD^*)^{-1}\bB^*
\end{bmatrix},
\end{align*}
which, when combined with that fact that $\sbm{ \bA & \bB \\ \bC &
\bD}$ is unitary, imply that both $\bU_0$ and $\bU_0^*$ are isometries. In
terms of our notation, the result of Agler \cite[Proof of theorem
1.8]{Agler1990} is that then $F$ has the representation
\eqref{DdHerreal2'} with
\begin{equation}   \label{Agler-rep}
 A = \bU_{0}, \quad V= \frac{1}{\sqrt{2}} \bB.
\end{equation}

 The formula \eqref{DdHerreal2'} can be further adjusted as follows:
 \begin{align}
     F(\zeta) - R & = V^{*}(I - A
     P(\zeta))^{-1} (I + AP(\zeta))V  \notag \\
          & = V^{*}(U - P(\zeta))^{-1} (U + P(\zeta)) V
     \label{DdHerreal2}
  \end{align}
where we set  $U = A^{*}$. Here we still have $V^{*}V =
\operatorname{Re} F(0)$ and $U = A^{*}$ is unitary.  We shall use
this representation for a ${\mathbb D}^{d}$-Herglotz--Agler
function in Section \ref{S:Bes}. For an alternate direct
derivation of the realization \eqref{DdHerreal2} for a ${\mathbb
D}^{d}$-Herglotz--Agler function, see \cite{BK-V}, where the
result is given for the rational matrix-valued case with the
additional constraint that $F$ have zero real part on the unit
circle; in this case one can arrange that the state space $\cX$ is
finite-dimensional. }
\end{remark}

  \section{The Schur--Agler class over the right polyhalfplane}
  \label{S:SA-Pi}

  In this section we present our realization results for the
  Schur--Agler class over the right polyhalfplane $\Pi^{d}$.

  \begin{theorem}  \label{T:SA-Pi}
      Given a function $s \colon \Pi^{d} \to \cL(\cU, \cY)$, the
      following are equivalent.

      \begin{enumerate}
      \item $s$ is in the right polyhalfplane Schur--Agler class
      $\mathcal{SA}(\Pi^{d}, \cL(\cU, \cY))$.

      \item $s$ has a $\Pi^{d}$-Schur--Agler decomposition, i.e.,
      there exist positive kernels $K_{1}, \dots, K_{d}$ on
      $\Pi^{d} $ such that
     \begin{equation}  \label{Pi-SchurAgdecom}
         I - s(z)^{*} s(w) = \sum_{k=1}^{d} (\overline{z}_{k} + w_{k})
         K_{k}(z,w).
 \end{equation}
 for all $z,w \in \Pi^{d}$.

 \item There exists a state space $\cX$ and a $d$-fold positive
 decomposition of $I_{\cX}$ $(Y_{1}, \dots, Y_{d})$ (see Section
 \ref{S:decomid}) together with a $\Pi$-scattering-conservative system node (see Example
 \ref{E:scatconserv})
 $$
 \bU = \begin{bmatrix} A \& B \\ C \& D \end{bmatrix} \colon \cD
 \subset \begin{bmatrix} \cX \\ \cU \end{bmatrix} \to \begin{bmatrix}
 \cX \\ \cY \end{bmatrix}
 $$
 so that
 \begin{equation}   \label{real'}
   s(w) = C \& D  \begin{bmatrix} \left((Y(w) -A)_{|\cX}\right)^{-1}
B \\ I_{\cU}
   \end{bmatrix},
 \end{equation}
 where we have set $Y(w) = w_{1}Y_{1} + \cdots + w_{d}Y_{d}$.
  \end{enumerate}
 \end{theorem}

 \begin{proof}
     (1)$\Rightarrow$(2):
     The proof uses the component-wise multivariable Cayley transform
     from $\Pi^{d}$ to ${\mathbb D}^{d}$; for this purpose it is
     convenient to use the condensed notation \eqref{Cayley1}
     for the point $\frac{1 + \zeta}{1- \zeta} = \left( \frac{1 +
     \zeta_{1}}{1 - \zeta_{1}}, \dots, \frac{1 + \zeta_{d}}{1 -
     \zeta_{d}} \right)$ in the polydisk ${\mathbb D}^{d}$ associated
     with the point $\zeta = (\zeta_{1}, \dots, \zeta_{d}) \in
     \Pi^{d}$.

     For $s \in \mathcal{SA}(\Pi^{d}, \cL(\cU,\cY))$ in the
     Schur--Agler class over $\Pi^{d}$ we associate the function $S \in \mathcal{SA}({\mathbb D}^{d},
     \cL(\cU, \cY))$ in the Schur--Agler class over ${\mathbb D}^{d}$
     via
     $$
     S(\zeta) = s\left( \frac{1+\zeta}{1 - \zeta} \right)
     $$
     Then by Theorem \ref{T:Agler} $S$ has a ${\mathbb D}^{d}$-Schur--Agler decomposition
     $$
I_{\cU} - S(\omega)^{*} S(\zeta) = \sum_{k=1}^{d} (1 - \overline{\omega}_{k} \zeta_{k})
\widetilde K_{k}(\omega, \zeta).
     $$
     Using the relation
     $$
     s(w) = S\left( \frac{w+1}{w-1} \right)
     $$
     (where we use the convention \eqref{Cayley2}),
     we next get that
  $$
  I_{\cU} - s(z)^{*} s(w) = \sum_{k=1}^{d} \left( 1 -
  \frac{\overline{z}_{k}-1}{\overline{z}_{k}+1} \cdot
  \frac{w_{k}-1}{w_{k}+1}\right) \widehat K_{k}(z,w)
  $$
  where we set
  $$ \widehat K_{k} (z,w) =
   \widetilde K_{k}\left(
  \frac{z-1}{z+1},
  \frac{w-1}{w +1} \right).
  $$

  This in turn leads to
  $$
  I_{\cU} - s(z)^{*} s(w) = \sum_{k=1}^{d} (\overline{z}_{k} + w_{k})
  K_{k}(z,w)
  $$
  with $K_{k}$ given by
  $$
  K_{k}(z,w) = 2 \frac{1}{\overline{z}_{k}+1} \widehat K_{k}(z,w)
  \frac{1}{w_{k}+1}
  $$
  and (2) follows.

  \smallskip
  (2)$\Rightarrow$(3):  We use the Kolmogorov decompositions
  $K_{k}(z,w) = H_{k}(z)^{*} H_{k}(w)$ (where $H_{k} \colon \Pi^{d}
  \to \cL(\cU, \widetilde \cX_{k})$) of the positive kernels
  $K_{k}(z,w)$ to rewrite the $\Pi^{d}$-Schur--Agler decomposition
  \eqref{Pi-SchurAgdecom} in the form
  \begin{equation}   \label{Pi-SchurAgdecom'}
      I - s(z)^{*} s(w) = H(z)^{*} (P(z)^{*} + P(w)) H(w)
  \end{equation}
  where we set $H(w) = \operatorname{Col}_{1 \le k \le d} H_{d}(w)$
  with $P(w) = \left[ \begin{smallmatrix} w_{1} I_{\widetilde
  \cX_{1}} & & \\ & \ddots & \\ & & w_{d} I_{\widetilde \cX_{d}}
\end{smallmatrix} \right]$. We introduce the subspace
\begin{equation}   \label{defcX0}
\widetilde \cX_{0} : = \overline{\operatorname{span}} \left\{
\operatorname{Ran} H(w) \colon w \in \Pi^{d}\right\} \subset
\bigoplus_{k=1}^{d} \widetilde \cX_{k}=\widetilde{\cX}
\end{equation}
and we introduce operators $\widetilde Y_{1}, \dots, \widetilde
Y_{d}$ on $\widetilde
\cX_{0}$ by
$$
  \widetilde Y_{k} = P_{\widetilde \cX_{0}} P_{k}|_{\widetilde
\cX_{0}}
$$
where $P_{\widetilde \cX_{0}}$ is the orthogonal projection of
$\widetilde \cX$ onto its subspace $\widetilde \cX_{0}$
\eqref{defcX0} and where $P_k=P_{\widetilde \cX_{k}}$ is the
orthogonal projection of $\widetilde \cX$ onto $\widetilde
\cX_{k}$. It is easily verified that $\widetilde Y_{1}, \dots,
\widetilde Y_{d}$ form a positive decomposition of the identity on
$\widetilde \cX_{0}$.

We next view \eqref{Pi-SchurAgdecom'}
as the statement that the subspace
$$
\widetilde \cG: = \overline{\operatorname{span}} \left\{
\begin{bmatrix} \widetilde Y(w) H(w) \\ s(w) \\ H(w) \\ I_{\cU}
\end{bmatrix} u \colon w \in \Pi^{d},
\, u \in \cU \right\}
\subset
\begin{bmatrix} \widetilde \cX_{0} \\ \cY \\ \widetilde \cX_{0} \\
\cU \end{bmatrix}
 =: \widetilde{\boldsymbol{\cK}}
$$
is an isotropic subspace of $\widetilde{\boldsymbol{\cK}}$, where
$\widetilde{\boldsymbol{\cK}}$ is considered as a Kre\u{\i}n space
with inner product induced by the indefinite Gramian matrix
\begin{equation}  \label{Pi-J}
\widetilde \cJ: = \begin{bmatrix}  0 & 0 & I_{{\widetilde \cX}_{0}} &
0 \\
0 & I_{\cY} & 0 & 0 \\ I_{{\widetilde \cX}_{0}} & 0 & 0 & 0 \\ 0 & 0
& 0 &
-I_{\cU} \end{bmatrix}.
\end{equation}

We next check that $\widetilde \cG$ can be expressed as a graph space
$$
  \widetilde \cG = \begin{bmatrix} \widetilde A \& \widetilde B \\
     \widetilde C \& \widetilde D \\
     \begin{matrix} I_{{\widetilde \cX}_{0}} & 0 \end{matrix} \\
     \begin{matrix} 0 & I_{\cU} \end{matrix} \end{bmatrix}
         \widetilde \cD_{0}
$$
associated with a closed operator $\widetilde \bU = \left[
\begin{smallmatrix}  \widetilde A \& \widetilde B \\
     \widetilde C \& \widetilde D  \end{smallmatrix} \right] $ with
dense
domain $\widetilde \cD_{0} \subset \widetilde \cX_{0} \oplus \cU$
mapping into $\widetilde \cX_{0} \oplus \cY$.

To this end we first
check the necessary condition that
$$
\widetilde \cG \cap \begin{bmatrix} \widetilde \cX \\ \cY \\ \{0\} \\
\{0\} \end{bmatrix} = \{0\}
$$
as follows. We suppose that $x' \oplus y' \oplus 0 \oplus 0 \in
\widetilde \cG$. As $\widetilde \cG$ is isotropic, it follows that
\begin{equation*}
 0 = \left\langle \widetilde \cJ \left[ \begin{smallmatrix} x' \\ y'
\\ 0 \\ 0
 \end{smallmatrix} \right], \left[ \begin{smallmatrix} x' \\ y' \\ 0
 \\ 0 \end{smallmatrix} \right] \right\rangle = \| y'\|^{2}
\end{equation*}
from which we see that $y' = 0$ and $x' \oplus 0 \oplus 0 \oplus 0
\in \widetilde \cG$.  As $\widetilde \cG$ is isotropic, we must
then also have
$$
 0 = \left\langle \widetilde \cJ \left[ \begin{smallmatrix}
\widetilde Y(w) H(w)
 \\ s(w) \\ H(w) \\ I_{\cU} \end{smallmatrix} \right] u,
 \left[ \begin{smallmatrix} x' \\ 0 \\ 0 \\ 0 \end{smallmatrix}
 \right] \right\rangle = \langle H(w) u, x' \rangle_{{\widetilde
X}_{0}}
 $$
 for all $u \in \cU$ and $w \in \Pi^{d}$.  We now use the
 condition \eqref{defcX0} to conclude that necessarily $x' = 0$ as
 well.

 We next observe that \eqref{Pi-SchurAgdecom'} for
 $z=w=t\mathbf{e}$, where $\mathbf{e}=(1,\ldots,1)$, becomes
 $$I-s(t\mathbf{e})^*s(t\mathbf{e})=2tH(t\mathbf{e})^*H(t\mathbf{e}),$$
so for any $u\in\mathcal{U}$ we have
$$\|H(t\mathbf{e})u\|^2=\frac{1}{2t}(\|u\|^2-\|s(t\mathbf{e})u\|^2)\to
0\ {\rm as}\ t\to\infty.$$ Therefore
\begin{multline*}
\overline{\rm span}_{u,w}\Big\{\begin{bmatrix}
H(w)\\
I_{\cU}
\end{bmatrix}u\Big\}\supset\overline{\rm span}_{u,w,t}\Big\{\begin{bmatrix}
H(w)-H(t\mathbf{e})\\
0
\end{bmatrix}u\Big\}
\supset\overline{\rm span}_{u,w}\Big\{\begin{bmatrix}
H(w)\\
0
\end{bmatrix}u\Big\}\\
=\begin{bmatrix}
\widetilde{\cX_0}\\
\{0\}
\end{bmatrix},
\end{multline*}
  and we now see that $
 \left\{ \left[\begin{smallmatrix} H(w) \\
I_{\cU} \end{smallmatrix} \right] u \right\}$ has dense span in
$\widetilde \cX_{0} \oplus \cU$.  We conclude that $\widetilde
\cG$ is indeed a graph space as claimed.

By Proposition 2.5 in \cite{BS}, we may embed $\widetilde \cG$ into a
$\cJ$-Lagrangian subspace $\cG$ of $\boldsymbol{\cK}: = \cX \oplus
\cY \oplus \cX \oplus \cU$ where we may arrange that $\cX$ is a
Hilbert space containing $\widetilde \cX$ and where we set
\begin{equation}\label{eq:J}
\cJ = \begin{bmatrix} 0 & 0 & I_{\cX} & 0 \\ 0 & I_{\cY} & 0 & 0 \\
I_{\cX} & 0 & 0 & 0 \\ 0 & 0 & 0 & - I_{\cU} \end{bmatrix}.
\end{equation}
Furthermore, it is argued there that one can arrange that this
(possibly) enlarged Lagrangian subspace $\cG$ is also a
graph space:
$$
\cG \cap \begin{bmatrix} \cX \\ \cY \\ \{0\} \\ \{0\} \end{bmatrix} =
\{0\}.
$$
Hence there is a closed operator
$$
\bU = \begin{bmatrix} A \& B \\ C \& D \end{bmatrix} \colon
\cD(\bU)
\subset \begin{bmatrix} \cX \\ \cU \end{bmatrix} \to \begin{bmatrix}
\cX \\ \cY \end{bmatrix}
$$
so that
\begin{equation}   \label{cG-form}
\cG = \begin{bmatrix} \bU \\ I \end{bmatrix} \cD(\bU) =
\begin{bmatrix} A \& B \\ C \& D \\ \begin{matrix} I_{\cX} & 0 \\ 0 &
    I_{\cU} \end{matrix} \end{bmatrix} \cD(\bU).
\end{equation}
As discussed in Example \ref{E:scatconserv}, $\bU$
is a $\Pi$-scattering-conservative system node.
It is shown in the proof of Proposition 4.9 from \cite{BS} that then the main
operator $A$ of $\bU$ (given by
\eqref{defA} with $\bU$ in place of $\Sigma$)
is maximal dissipative (as defined in Example \ref{E:skewadjoint} above).
We let $Y_{1}, \dots, Y_{d}$ be a positive decomposition of the
identity on $\cX$ which extends $\widetilde Y_{1}, \dots, \widetilde Y_{d}$;
e.g., one way to do this is
\begin{align*}
& Y_{1}|_{\cX \ominus \widetilde \cX_{0}} = I_{\cX \ominus \widetilde
 \cX_{0}}, \quad Y_{k}|_{\cX \ominus \widetilde \cX_{0}} = 0 \text{
 for } k = 2, \dots, d, \\
& Y_{k}|_{{\widetilde \cX}_{0}} = Y_{k} \text{ for } k=1, \dots, d
\end{align*}
and extend by linearity.  We then set $Y(w) = w_{1} Y_{1} + \cdots +
w_{d} Y_{d}$.

As $\cG$ contains $\widetilde \cG$, we conclude that
$$
\begin{bmatrix} Y(w) H(w) u \\ s(w) u \\ H(w) u \\ u \end{bmatrix}
    \in \begin{bmatrix} A \& B \\ C \& D \\ \begin{matrix} I_{\cU} &
    0 \\ 0 & I_{\cU} \end{matrix} \end{bmatrix} \cD(\bU)
$$
for each $w \in \Pi^{d}$ and $u \in \cU$.  Thus, for each such $w$
and $u$ there is a $\left[ \begin{smallmatrix} x_{w,u} \\ u_{w,u}'
\end{smallmatrix} \right] \in \cD(\bU)$ so that
\begin{equation}   \label{id1'}
    \begin{bmatrix} Y(w) H(w) u \\s(w) u \\ H(w) u \\ u \end{bmatrix}
    = \begin{bmatrix} A \& B \\ C \& D  \\
    \begin{matrix} I_{\cX} & 0 \\ 0 & I_{\cU} \end{matrix} \end{bmatrix}
\begin{bmatrix}
    x_{w,u} \\ u_{w,u}' \end{bmatrix}.
\end{equation}
From the bottom two rows of \eqref{id1'} we read off
$$
H(w) u = x_{w,u}, \quad u = u_{w,u}'.
$$
Plugging these back into the first two rows of \eqref{id1'} gives
\begin{equation} \label{colsys'}
 \begin{array}{rcl}
A_{|\cX} H(w) u + B u & = &  Y(w) H(w) u  \\
C \& D \begin{bmatrix} H(w) u \\ u \end{bmatrix} & = & s(w) u.
\end{array}
\end{equation}
Since $\bU$ is a $\Pi$-scattering-conservative system node, a
consequence of Proposition \ref{P:strucresol} is that $Y(w) - A$
is invertible for each $w \in \Pi^{d}$ with an extension $(Y(w) -
A)_{|\cX} \colon \cX \to \cX_{-1}$ having the property that $(
(Y(w) - A)_{|\cX})^{-1} \colon \cX_{-1} \to \cX$.   We may
therefore solve the first of equations \eqref{colsys'} for $H(w)
u$ to get
\begin{equation}  \label{Hsolve}
    H(w) u = \left( (Y(w) - A)_{|\cX}\right)^{-1} B u.
\end{equation}
and furthermore
$$
\begin{bmatrix} \left( (Y(w) - A)_{|\cX} \right)^{-1} B u \\ u
\end{bmatrix} \in \cD(C \& D )  \text{ for each } w \in \Pi^{d}.
$$
We can now substitute \eqref{Hsolve} into the second of equations
\eqref{colsys'} to arrive at the desired realization formula
\eqref{real'} for $s$, and (3) follows.

\smallskip

(3)$\Rightarrow$(2): Assume that $s$ has a realization as in
\eqref{real'}.  For $w \in \Pi^{d}$, set
$$
 H(w) =  \left( (Y(w) - A)_{|\cX} \right)^{-1}B.
$$
Observe that
\begin{align*}
A_{|\cX} H(w) u + Bu  & = A \& B \begin{bmatrix} \left( (Y(w) -
A)_{|\cX}\right)^{-1} B \\ I \end{bmatrix} u \\
& =  A_{|\cX} \left( (Y(w) - A)_{|\cX} \right)^{-1} B u + B u  \\
& = -B u + P(w) \left( (Y(w) - A)_{|\cX}\right)^{-1} B u + B u \\
& = Y(w) \left( (Y(w) - A)_{|\cX}\right)^{-1} B u \\
& = Y(w) H(w)u.
\end{align*}
 Combining this with \eqref{Pi-SchurAgdecom'} gives
 \begin{equation}   \label{use1}
   \bU \begin{bmatrix} H(w) u \\ u \end{bmatrix} =
   \begin{bmatrix} A \& B \\ C \& D \end{bmatrix} \begin{bmatrix}
       H(w) u \\ u \end{bmatrix} = \begin{bmatrix}  Y(w) H(w) u \\
       s(w) u \end{bmatrix}.
 \end{equation}
By Proposition 4.9 from \cite{BS}, the fact that $\bU$ is a
$\Pi$-scattering-conservative system node tells us that the graph
of $\bU$ is a Lagrangian subspace of $\cX \oplus \cY \oplus \cX
\oplus \cU$ in the Kre\u{\i}n-space inner product induced by $\cJ$
as in \eqref{eq:J}. In particular it holds that
\begin{align}
    0 & = \left\langle \cJ \begin{bmatrix} A \& B \\ C \& D \\
    \begin{matrix} I_{\cX} & 0 \\ 0 & I_{\cU} \end{matrix}
    \end{bmatrix} \begin{bmatrix} H(w) u \\ u \end{bmatrix},  \,
    \begin{bmatrix} A \& B \\ C \& D \\
    \begin{matrix} I_{\cX} & 0 \\ 0 & I_{\cU} \end{matrix}
    \end{bmatrix} \begin{bmatrix} H(z) u' \\ u' \end{bmatrix}
    \right\rangle   \notag \\
    & = \left\langle \cJ \begin{bmatrix}  Y(w) H(w) u \\ s(w) u \\
H(w) u
    \\ u \end{bmatrix}, \, \begin{bmatrix}  Y(z) H(z) u' \\ s(z) u'
    \\ H(z) u' \\ u' \end{bmatrix} \right\rangle \notag  \\
    & = \langle \left[ H(z)^{*} (Y(z)^{*} + Y(w)) H(w) +
(s(z)^{*}s(w) -
    I_{\cU})\right] u, u' \rangle.
    \label{use2}
  \end{align}
By the arbitrariness of $u, u' \in \cU$, we conclude that
\begin{equation}  \label{contractiveAgdecom}
H(z)^{*}\left(\sum_{k=1}^{d} (\overline{z}_{k} + w_{k})
Y_{k}\right) H(w) = I_{\cU } - s(z)^{*} s(w).
\end{equation}
Then \eqref{contractiveAgdecom} leads to the Agler decomposition
\eqref{Pi-SchurAgdecom} with
$$
K_{k}(z,w) = H(z)^{*} Y_k H(w)=[Y_k^{1/2}H(z)]^*[Y_k^{1/2}H(w)].
$$

 \smallskip

 (2)$\Rightarrow$(1):  We write the Agler decomposition
 \eqref{Pi-SchurAgdecom} in the form \eqref{Pi-SchurAgdecom'}.  Then,
 if $\mathbf{A} = (A_{1}, \dots, A_{d})$ is a commutative $d$-tuple
 of strictly accretive operators, the functional calculus gives
 $$
 I - s({\mathbf A})^{*} s({\mathbf A}) = H({\mathbf A})^{*}
(P({\mathbf
 A})^{*} +  P({\mathbf A})) H({\mathbf A}).
 $$ From the diagonal form $P(w) = w_{1} P_{1} + \cdots + w_{d}
 P_{d}$ of $P(w)$, we see that $P({\mathbf A})^{*} + P({\mathbf A})
 \ge 0$, and hence $I - s({\mathbf A})^{*} s({\mathbf A}) \ge 0$, or
 $\|s({\mathbf A}) \| \le 1$.  We conclude that $s \in
 \mathcal{SA}(\Pi^{d}, \cL(\cU, \cY))$ and (1) follows.
\end{proof}

  \section{The Herglotz--Agler class over the polyhalfplane}
  \label{S:HA-Pi}

 In this section we present our realization results for a restricted
 class of Herglotz--Agler functions over the right polyhalfplane
 $\Pi^{d}$, where a growth condition \eqref{growth} is imposed at
 infinity.

  \begin{theorem} \label{T:HA-Pi}
      Given a function $f \colon \Pi^{d} \to \cL(\cU)$, the following
      are equivalent:
      \begin{enumerate}
      \item $f \in \mathcal{HA}(\Pi^{d}, \cL(\cU))$ and also $f$
      satisfies the growth condition at $+\infty$:
      \begin{equation}   \label{growth}
          \lim_{t \to +\infty} t^{-1} f(t {\mathbf e}) u = 0
          \text{ for each } u \in \cU.
        \end{equation}
    where  ${\mathbf e} = (1, \dots, 1) \in \Pi^{d}$.

    \item $f$ has a $\Pi^{d}$-Herglotz--Agler decomposition, i.e.,
    there exist $\cL(\cU)$-valued positive kernels $K_{1}, \dots,
    K_{d}$ on $\Pi^{d}$ such that
 \begin{equation}   \label{PiHerAgdecom}
     f(z)^{*} + f(w) = \sum_{k=1}^{d} (\overline{z}_{k} + w_{k}) K_{k}(z,w)
 \end{equation}
 and in addition $f$ satisfies the growth condition \eqref{growth}.

 \item There exists a Hilbert state space $\cX$ and a positive
 decomposition of the identity $(Y_{1}, \dots, Y_{d})$ on $\cX$
 along with an $\Pi$-impedance-conservative system node (see Example
 \ref{E:impconserv})
 $$
  \bY = \begin{bmatrix} A \& B \\ C \& D \end{bmatrix} \colon
  \cD(\bY) \subset \begin{bmatrix} \cX \\ \cU \end{bmatrix} \to
  \begin{bmatrix} \cX \\ \cU \end{bmatrix}
 $$
 such that
 \begin{equation}   \label{real''}
     f(w) = C \& D \begin{bmatrix} \left( (Y(w) -
     A)_{|\cX}\right)^{-1} B \\ I_{\cU} \end{bmatrix}.
  \end{equation}
  \end{enumerate}
\end{theorem}

\begin{proof}
    We show

(1)$\Rightarrow$(2)$\Rightarrow$(3)$\Rightarrow$(2)$\Rightarrow$(1).

    \smallskip

    \textbf{(1)$\Rightarrow$(2):}  We note that $f$ is in the
Herglotz--Agler
    class $\mathcal{HA}(\Pi^{d}, \cL(\cU))$ if and only if $s = (f-I) (f+I)^{-1}$ is in the Schur--Agler class
    $\mathcal{SA}(\Pi^{d}, \cL(\cU))$.  Thus by Theorem \ref{T:SA-Pi}
    this $s$ has a $\Pi^{d}$-Schur--Agler decomposition
    $$
    I - s(z)^{*} s(w) = \sum_{k=1}^{d} (\overline{z}_{k} + w_{k})
    \widetilde K_{k}(z,w)
    $$
    for some $\cL(\cU)$-valued positive kernels $\widetilde K_{k}$ on $\Pi^{d}$.  A standard computation gives
    $$
     I - s(z)^{*} s(w) = 2 (f(z)^{*}+I)^{-1} [f(z)^{*} + f(w)] (f(w)+I)^{-1}
     $$
     from which we see that \eqref{PiHerAgdecom} holds with
     $$
     K_{k}(z,w) = \frac{1}{2} (f(z)^{*}+I) \widetilde K_{k}(z,w) (f(w)+I).
     $$

     \smallskip

     \textbf{(2)$\Rightarrow$(3):}  We use the Kolmogorov
decompositions
     $K_{k}(z,w) = H_{k}(z)^{*} H_{k}(w)$ of the positive kernels
     $K_{k}$ (with $H_{k} \colon \Pi^{d} \to \cL(\cU, \widetilde
     \cX_{k})$ say) to rewrite the Agler decomposition
     \eqref{PiHerAgdecom} in the form
     \begin{equation}   \label{PiHerAgdecom'}
     f(z)^{*} + f(w) = H(z)^{*} ( P(z)^{*} + P(w)) H(w)
   \end{equation}
   where we have set
   $$
    H(w) = \begin{bmatrix} H_{1}(w) \\ \vdots \\ H_{d}(w)
\end{bmatrix}, \quad P(w) = \begin{bmatrix} w_{1} I_{\widetilde
\cX_{1}} & & \\ & \ddots & \\ & & w_{d} I_{\widetilde \cX_{d}}
\end{bmatrix}.
$$
Just as in the proof of Theorem \ref{T:SA-Pi},
we introduce the subspace
\begin{equation}  \label{defcX0'}
    \widetilde \cX_{0}:= \overline{\operatorname{span}} \{
    \operatorname{Ran} H(w) \colon w \in \Pi^{d} \} \subset
    \bigoplus_{k=1}^{d} \widetilde \cX_{k}=:\widetilde{X}
\end{equation}
and introduce operators $\widetilde Y_{1}, \dots, \widetilde Y_{d}$
on $\widetilde \cX_{0}$ by
$$
\widetilde Y_{k} = P_{\widetilde \cX_{0}} P_{k}|_{\widetilde \cX_{0}}.
$$
We then view \eqref{PiHerAgdecom} as the statement that the subspace
$$
\widetilde \cG = \overline{\operatorname{span}} \left\{
\begin{bmatrix} \widetilde{Y}(w) H(w) \\ f(w) \\ H(w) \\I_{\cU} \end{bmatrix} u
    \colon u \in \cU, \, w \in \Pi^{d} \right\} \subset
    \begin{bmatrix} \widetilde \cX_{0} \\ \cU \\ \widetilde \cX_{0}
\\ \cU
    \end{bmatrix} =:\widetilde{\boldsymbol \cK}
$$
is an isotropic subspace of $\widetilde{\boldsymbol{\cK}}$ when
$\widetilde{\boldsymbol{\cK}}$ is given the Kre\u{\i}n-space inner
product induced by the Gramian matrix
\begin{equation}   \label{PiHer-tildeJ}
\widetilde \cJ = \begin{bmatrix} 0 & 0 & I_{\widetilde \cX_0}& 0 \\
0 & 0 & 0 & -I_{\cU} \\ I_{\widetilde \cX_0} & 0 & 0 & 0 \\ 0 &
-I_{\cU} & 0 & 0 \end{bmatrix}.
\end{equation}

We show next that $\widetilde \cG$ is a graph space, i.e.,
\begin{equation}  \label{graph''}
    \widetilde \cG \cap \begin{bmatrix} \widetilde \cX_0 \\ \cU \\
\{0\}
    \\ \{0\} \end{bmatrix} = \{0\}.
 \end{equation}
 Toward this end, suppose that $x \oplus u \oplus 0 \oplus 0\in\widetilde \cG$.
 Our goal is to show that then necessarily $x=0$
 and $u = 0$.  As $\widetilde \cG$ is $\widetilde \cJ$-isotropic, we
necessarily have, for all $u' \in \cU$,
 \begin{align}
     0 & = \left\langle \widetilde \cJ \begin{bmatrix} \widetilde
Y(w) H(w) u' \\ f(w) u' \\ H(w) u' \\ u' \end{bmatrix},
\begin{bmatrix} x \\  u \\ 0 \\ 0 \end{bmatrix} \right\rangle  \notag \\
     & = \left\langle \begin{bmatrix} H(w)u' \\  -u' \end{bmatrix},
\, \begin{bmatrix} x \\ u \end{bmatrix} \right\rangle   \notag \\
     & = \langle u', H(w)^{*} x - u \rangle
     \label{ortho}
\end{align}
from which we conclude that
\begin{equation}  \label{have}
    u = H(w)^{*} x \text{ for all } w \in \Pi^{d}.
 \end{equation}
 We note that, as a consequence of \eqref{PiHerAgdecom'},

     $$
  H(t \be)^{*} H(t \be) = \frac{ f(t \be)^{*} + f(t \be)}{ 2 t}.
  $$
  The growth assumption \eqref{growth} then implies that $H(t \be) \to
  0$ strongly as $t \to +\infty$.

To conclude the proof of \eqref{graph''}, we now need only
specialize \eqref{have} to the case $w = t \be$ and take a weak
limit as $t \to +\infty$ to show that $u=0$.  It then follows from
\eqref{ortho} that $x$ is orthogonal to
$\overline{\operatorname{span}} \{ H(w) u \colon w \in \Pi^{d}, \,
u \in \cU\}$. As a consequence of \eqref{defcX0'}, this in turn
forces $x=0$ and \eqref{graph''} follows.

We next embed $\widetilde \cG$ into a $\cJ$-Lagrangian subspace $\cG$
of $\cX \oplus \cU \oplus \cX \oplus \cU$, where $\cX$ is a Hilbert
space containing $\widetilde \cX$ as a subspace and the indefinite
Gramian matrix has the same form as $\widetilde \cJ$ in
\eqref{PiHer-tildeJ} above:
\begin{equation}   \label{PiHer-J}
    \cJ = \begin{bmatrix}  0 & 0 & I_{\cX} & 0 \\
    0 & 0 & 0 & - I_{\cU} \\ I_{\cX} & 0 & 0 & 0 \\ 0 & -I_{\cU} & 0
    & 0 \end{bmatrix}
\end{equation}
in such a way that $\cG$ is still a graph subspace.
That this is possible follows via a minor adjustment of Proposition
2.5 in \cite{BS} as indicated in the proof of Theorem 4.12 there.
We also note that condition (2) in Example \ref{E:impconserv} is
automatic since the subspace $\widetilde \cG$ satisfies this
condition by construction.
It then follows that $\cG$ has the form \eqref{cG-form} for a closed
operator
$$
\bU = \begin{bmatrix} A \& B \\ C \& D \end{bmatrix} \colon
\cD(\bU)
\subset \begin{bmatrix} \cX \\ \cU \end{bmatrix} \to \begin{bmatrix}
\cX \\ \cY \end{bmatrix}.
$$
That $\bU$ is a $\Pi$-impedance-conservative system node follows
from the fact that $\cG$ is $\cJ$-Lagrangian (with $\cJ$ given by
\eqref{PiHer-J}).  We also extend the positive decomposition of
the identity $\widetilde Y_{1}, \dots, \widetilde Y_{d}$ on
$\widetilde \cX_{0}$ to a positive decomposition of the identity
$Y_{1}, \dots, Y_{d}$ on $\cX$ just as in the proof of
(2)$\Rightarrow$(3) in Theorem \ref{T:SA-Pi} above. That we
recover $f(w)$ as the transfer-function for the system node $\bU$,
i.e., the formula \eqref{real''} holds, now follows exactly as in
the proof of  Theorem \ref{T:SA-Pi}.

\smallskip

\textbf{(3)$\Rightarrow$(2):}  We follow the proof of
(3)$\Rightarrow$(2) in
Theorem \ref{T:SA-Pi}. If we define $H(w) = A \& B \left[
\begin{smallmatrix} \left( Y(w) - A_{|\cX}\right)^{-1} B \\ I
\end{smallmatrix} \right]$ (well-defined by Proposition
\ref{P:strucresol}), we arrive at \eqref{use1} and
\eqref{use2}, but with  $\cJ$ given by \eqref{PiHer-J}
 rather than by \eqref{Pi-J}, leading to the adjusted final conclusion
 $$
     0 = \langle \left[ H(z)^{*} (Y(z)^{*} + Y(w)) H(w) - (f(z)^{*} +
     f(w)) \right] u, u' \rangle.
 $$
This leads to the Agler
 decomposition \eqref{PiHerAgdecom} with
 $$
   K_{k}(z,w) = H(z)^{*}Y_k H(w)=[Y_k^{1/2}H(z)]^*[Y_k^{1/2}H(w)]
 $$
 as in the proof of (3)$\Rightarrow$(2) in Theorem \ref{T:SA-Pi}.

 It remains to show that the growth condition \eqref{growth}
 necessarily holds if $f$ has a realization \eqref{real''} from a
 $\Pi$-impedance-conservative system node.  To see this, we note that then
 the single-variable Herglotz function $f(s \be)$ ($s \in \Pi)$ has
 an impedance-conservative system-node realization.  That the growth
 condition \eqref{growth} holds now follows from the result for the
 single-variable case (see Theorem 7.4 in \cite{Staffans-ND}).

 \smallskip

 \textbf{(2)$\Rightarrow$(1):}  The proof is parallel to the proofs
of
 (2)$\Rightarrow$(1) in Theorems \ref{T:HA-D} and \ref{T:SA-Pi}.
 Write the Agler decomposition \eqref{PiHerAgdecom} in the form
 \eqref{PiHerAgdecom'} and observe that the functional calculus gives
 $$
 f({\mathbf A})^{*} + f({\mathbf A}) =
  H({\mathbf A})^{*} (P({\mathbf A)}^{*} + P({\mathbf A}) )
  H({\mathbf A}).
  $$
  If ${\mathbf A}$ is a strictly accretive commutative $d$-tuple,
  from the diagonal form of $P(w)$ we see that
  $P({\mathbf A})^{*} + P({\mathbf A}) \ge 0$.  We conclude that $f
  \in \mathcal{HA}(\Pi^{d}, \cL(\cU))$, and (1) follows.
  \end{proof}

  In \cite{AMcCY} a criterion was given for when a
  $\Pi^{d}$-Herglotz--Agler function has a realization involving the
  structured resolvent $(P(w) - A)^{-1}$ coming from a spectral
  decomposition $(P_{1}, \dots, P_{d})$ (so $P(w) = w_{1} P_{1} +
\cdots
  + w_{d} P_{d}$) rather than just a positive decomposition $(Y_{1},
  \dots, Y_{d})$ of the identity).  We give our version of a result of
  this type, with realization in terms of a
  $\Pi$-impedance-conservative system node rather than in the form
  presented in \cite{AMcCY}.

  \begin{theorem} \label{T:HA-Pi'}
      Suppose that $f \colon \Pi^{d} \to \cL(\cU)$ is a
      $\Pi^{d}$-Herglotz--Agler function satisfying the growth
      condition \eqref{growth}.  Then the following are equivalent:
      \begin{enumerate}
       \item   There exists a Hilbert space $\cX$ and a
      $d$-fold spectral decomposition $(P_{1}, \dots,$ $ P_{d})$ of
      $I_{\cX}$
 along with a $\Pi$-impedance-conservative system node
 $$
  \bY = \begin{bmatrix} A \& B \\ C \& D \end{bmatrix} \colon
  \cD(\bU) \subset \begin{bmatrix} \cX \\ \cU \end{bmatrix} \to
  \begin{bmatrix} \cX \\ \cU \end{bmatrix}
 $$
 such that
 \begin{equation*}
     f(w) = C \& D \begin{bmatrix} \left( (P(w) -
     A)_{|\cX}\right)^{-1} B \\ I_{\cU} \end{bmatrix}.
  \end{equation*}

  \item If $S(\zeta) = \cC(f)(\zeta):= \left[
  f\left(\frac{1+\zeta}{1- \zeta}\right) - I \right]
  \left[f\left(\frac{1+\zeta}{1-\zeta}\right) + I\right]^{-1}$
(where by convention \eqref{Cayley1}
  $\frac{1+ \zeta}{1 - \zeta} =
  \left(\frac{1 + \zeta _{1}}{1 - \zeta_{1}}, \dots,
  \frac{1+ \zeta_{d}}{1 - \zeta_{d}}\right)$ if $\zeta = (\zeta_{1},
  \dots, \zeta_{d}) \in {\mathbb D}^{d}$), then $S$ is in the
  Schur--Agler class
  $\mathcal{SA}({\mathbb D}^{d}, \cL(\cU))$ and $S$ has a realization
  as in \eqref{SAreal} where $\bU = \left[\begin{smallmatrix} \bA &
  \bB \\ \bC & \bD \end{smallmatrix} \right]$ is unitary with the
  additional property that $1$ is not in the point spectrum of
  $\bU$.
\end{enumerate}
\end{theorem}

\begin{proof}

\textbf{(1)$\Rightarrow$(2):}  We suppose that we are given a
$\Pi$-impedance-conservative system node $\bY$ as in condition (1).
We set
$$
 H(w) = \left( (P(w) - A)_{|\cX}\right)^{-1} B \colon \cU \to \cX
$$
(well-defined by Proposition \ref{P:strucresol}) and verify that
\begin{equation}   \label{bYeq}
 \begin{bmatrix} A \& B \\ C \& D \end{bmatrix} \begin{bmatrix} H(w)
u
     \\ u \end{bmatrix} = \begin{bmatrix} P(w) H(w) u \\ f(w) u
 \end{bmatrix}
 \end{equation}
for all $u \in \cU$.  Furthermore, working as in the proof of (3)
$\Rightarrow$ (2) in Theorem \ref{T:HA-Pi}, we see that $H(w)$ so
defined provides an Agler decomposition for $f$:
$$
 f(z)^{*} + f(w) = \sum_{k=1}^{d} (\overline{z}_{k} + w_{k}) H(z)^{*}
 P_{k} H(w).
$$
If we set $\widetilde \bY = \left[ \begin{smallmatrix} - (A \& B)
\\ C \& D \end{smallmatrix} \right]$, then the fact that $\bY =
\left[
\begin{smallmatrix} A \& B \\ C \& D \end{smallmatrix} \right]$ is an
impedance-conservative system
node means that $\widetilde \bY$ is skew-adjoint (see Corollary
\ref{C:impconserv}):
$\widetilde \bY^{*} = - \widetilde \bY$.  Moreover we can rewrite the
identity \eqref{bYeq} in the form
\begin{equation}  \label{Yeq}
 \widetilde \bY \begin{bmatrix} H(w) u \\ u \end{bmatrix}
=\begin{bmatrix}  -P(w)  H(w) u \\ f(w)
 u \end{bmatrix}
\end{equation}

As $\widetilde \bY$ is skew-adjoint,  easily verified properties
of the Cayley transform imply that $\bU: = (\widetilde \bY -
I)(\widetilde \bY +I)^{-1}$ is unitary and the point $1$ is not in
the point spectrum of $\bU$ (this is another version of Proposition
\ref{P:unitary/skewadj}).  It remains only to check that $\bU
=: \left[ \begin{smallmatrix}  \bA & \bB \\ \bC & \bD
\end{smallmatrix} \right]$ provides a ${\mathbb
D}^{d}$-scattering-conservative realization of $S:= \cC(f)$.

We rewrite \eqref{Yeq} in terms of $\bU$ as
$$
(I + \bU) (I - \bU)^{-1} \begin{bmatrix} H(w) u \\ u \end{bmatrix}
= \begin{bmatrix} -P(w) H(w) u \\ f(w)u \end{bmatrix}
$$
or equivalently
$$
(I + \bU) \begin{bmatrix} H(w) u \\ u \end{bmatrix} = (I - \bU)
\begin{bmatrix} -P(w) H(w) u \\ f(w) u \end{bmatrix}
    $$
We reorganize this using linearity to get
\begin{equation}  \label{Ueq}
    \bU \begin{bmatrix}  (P(w)-I) H(w) u \\ -(f(w)+I) u
\end{bmatrix} = \begin{bmatrix} ( P(w) + I) H(w) u \\ -(f(w)-I) u
\end{bmatrix}.
\end{equation}
We introduce $\cL(\cU, \cX_{k})$-valued functions $\widetilde H_{k}$
on the polydisk ${\mathbb D}^{d}$ according to the relation
(where we again use the convention \eqref{Cayley2})
$$
 H_{k}(w): = P_{k} H(w) = \frac{1}{w_{k}+1} \widetilde H_{k}\left( \frac{w-1}{w+1}
 \right) (f(w)+I).
$$
Then we note that
\begin{align*}
    (P(w) - I) H(w) & = \sum_{k=1}^{d} (w_{k}-1) P_{k} H(w)  \\
    & = \sum_{k=1}^{d} \frac{w_{k}-1}{w_{k}+1} \widetilde
    H_{k}(\zeta) (f(w) + I) \\
    & = \sum_{k=1}^{d} \zeta_{k}  \widetilde H_{k}(\zeta) (f(w)
    + I) \\
    & = P(\zeta) \widetilde H(\zeta) (f(w) + I)
\end{align*}
where we set
$$
\zeta = \frac{w-1}{w+1} \text{ (as in \eqref{Cayley2}) for } w \in
\Pi^{d}
$$
and
$$
    \widetilde H(\zeta) = \sum_{k=1}^{d}  \widetilde
    H_{k}(\zeta).
$$
Similarly one can verify that
$$
(P(w) + I) H(w) = \sum_{k=1}^{d}\frac{w_{k}+1}{w_{k}+1} \widetilde
H_{k}(\zeta)(f(w) + I) = \widetilde H(\zeta) (f(w) + I),
$$
and we have arrived at the pair of identities
\begin{align*}
    & (P(w) -I) H(w) = P(\zeta) \widetilde H(\zeta) (I + f(w)) \\
    &  (P(w)+I) H(w) = \widetilde H(\zeta) (I + f(w)).
\end{align*}
Fix a vector $v  \in \cU$.  Define $u_{w} = (I + f(w))^{-1} v$ so
$v = (I + f(w)) u_{w}$. Then \eqref{Ueq} with $u = u_{w}$ can be
rewritten in the form
\begin{equation}   \label{Ueq'}
    \bU \begin{bmatrix}  P(\zeta) \widetilde H(\zeta) v \\ -v
\end{bmatrix} = \begin{bmatrix} \widetilde H(\zeta) v \\ -S(\zeta) v
\end{bmatrix}.
\end{equation}
Writing out $\bU = \left[ \begin{smallmatrix} \bA & \bB \\ \bC & \bD
\end{smallmatrix} \right]$, one can now solve \eqref{Ueq'} in the
standard way to arrive at
$$
 S(\zeta) = \bD + \bC P(\zeta) (I - \bA P(\zeta))^{-1} \bB,
$$
i.e., the unitary colligation matrix $\bU$ with the additional
property that $\bU$ does not have $1$ as an eigenvalue provides a
${\mathbb D}^{d}$-scattering-conservative realization for $S =
\cC(f)$.

\smallskip

\textbf{(2)$\Rightarrow$(1):}  We suppose that $S = \cC(f)$ has a
${\mathbb D}^{d}$-scattering-conservative realization \eqref{SAreal} where the
associated unitary colligation matrix $\bU = \left[
\begin{smallmatrix} \bA & \bB \\ \bC & \bD \end{smallmatrix} \right]$
    does not have $1$ as an eigenvalue.  Then we know that $\bU$
    also has the defining property
\begin{equation}   \label{Ueq''}
    \bU \begin{bmatrix} P(\zeta) \widetilde H(\zeta) \\ I
\end{bmatrix} u = \begin{bmatrix} \widetilde H(\zeta) \\ S(\zeta)
\end{bmatrix} u
\end{equation}
for all $u \in \cU$ and $\zeta \in {\mathbb D}^{d}$
where $\widetilde H(\zeta) = \left[ \begin{smallmatrix}
\widetilde H_{1}(\zeta) \\ \vdots \\ \widetilde H_{d}(\zeta)
\end{smallmatrix}
\right]$ provides a ${\mathbb D}^{d}$-Schur--Agler decomposition
\eqref{DSchurAgdecom}.  Since $1$ is not an eigenvalue of $\bU$ by
assumption we may form the Cayley transform
$$ \widetilde \bY : = (I + \bU) (I - \bU)^{-1}
 \text{ with } \cD(\widetilde \bY) = {\rm Ran} (I- \bU).
$$
By Proposition \ref{P:unitary/skewadj}, $\widetilde \bY$
is skew-adjoint:
\begin{equation}  \label{tildeYskew}
    \widetilde{\bY} = - \widetilde{\bY}^{*}.
\end{equation}

    By construction we have
 \begin{equation}   \label{Yeq'}
 \widetilde \bY \colon (I - \bU) \begin{bmatrix} x \\ u \end{bmatrix}
 \mapsto (I + \bU) \begin{bmatrix} x \\ u \end{bmatrix}.
\end{equation}
From \eqref{Ueq''} we note that
\begin{align*}
  & (I - \bU) \begin{bmatrix} P(\zeta) \widetilde H(\zeta) \\ I
\end{bmatrix} u
= \begin{bmatrix} (P(\zeta) - I) \widetilde H(\zeta) u \\ (I -
S(\zeta)) u \end{bmatrix}, \\
 & (I + \bU) \begin{bmatrix} P(\zeta) \widetilde H(\zeta) \\ I
\end{bmatrix} u
 = \begin{bmatrix} (P(\zeta) + I) \widetilde H(\zeta) \\ (I +
S(\zeta)) u \end{bmatrix}.
\end{align*}
Notice that $(I - S(\zeta)) = 2 (f(w) + I)^{-1}$ is invertible.
Hence \eqref{Yeq'} leads to
\begin{equation}   \label{Yeq''}
   \widetilde \bY \begin{bmatrix} (P(\zeta) - I) \widetilde
   H(\zeta) (I - S(\zeta))^{-1} \\ I \end{bmatrix} u =
   \begin{bmatrix} (P(\zeta) + I) \widetilde H(\zeta) (I - S(\zeta))^{-1} u
       \\ (I + S(\zeta)) (I - S(\zeta))^{-1} u
       \end{bmatrix}.
 \end{equation}
 Let us set
 $$
 H_{k}(w) = \frac{1}{w_{k}+1} \widetilde H_{k} \left( \frac{w-1}{w+1}
 \right) (I + f(w))
 $$
 where we use the convention \eqref{Cayley2} as usual,
 and note that
 $$ I + f(w) =  2 (I - S(\zeta))^{-1}.
 $$
 Then we compute
 \begin{align*}
     (P(\zeta) - I) \widetilde H(\zeta) ( I - S(\zeta))^{-1} & =
     \sum_{k=1}^{d} \left( \frac{w_{k}-1}{w_{k}+1} - 1 \right)
     \widetilde H_{k}(\zeta) (I - S(\zeta))^{-1} \\
     & = - \sum_{k=1}^{d} \frac{2}{w_{k}+1}  \widetilde
     H_{k}(\zeta) (I - S(\zeta))^{-1} \\
     & = - \sum_{k=1}^{d}  H_{k}(w) = -H(w)
 \end{align*}
 and similarly
 \begin{align*}
     (P(\zeta) + I) \widetilde H(\zeta) (I - S(\zeta))^{-1} & =
     \sum_{k=1}^{d} \left( \frac{w_{k}-1}{w_{k}+1} + 1 \right)
     \widetilde H_{k}(\zeta) (I - S(\zeta))^{-1}  \\
     & = \sum_{k=1}^{d} \frac{2 w_{k}}{w_{k}+1}  \widetilde
     H_{k}(\zeta) (I - S(\zeta)^{-1} \\
     & = \sum_{k=1}^{d} w_{k}  H_{k}(w) = P(w) H(w).
\end{align*}
 From these identities we see that the identity \eqref{Yeq''} is
 equivalent to
 \begin{equation}  \label{eqY'''}
 \widetilde \bY \begin{bmatrix} -H(w) \\ I \end{bmatrix} u =
 \begin{bmatrix} P(w) H(w) \\ f(w) \end{bmatrix} u.
 \end{equation}
 In particular, {\em for $u \in \cU$ then there is an $x_{u} = -H(w) u \in \cX$ so
 that $\left[ \begin{smallmatrix} x_{u} \\ u \end{smallmatrix}
 \right] \in \cD({\widetilde \bY})$}.

 As we have already observed that
 $\widetilde Y$ is skew-adjoint (see \eqref{tildeYskew}), we now have all
 the hypotheses needed in order to apply Corollary \ref{C:impconserv}
 to conclude that  $\bY =\widetilde \bY \left[ \begin{smallmatrix} -I &
 0 \\ 0 & I \end{smallmatrix} \right] =:
 \left[\begin{smallmatrix} A \& B \\ C \& D \end{smallmatrix} \right]$ is a
 $\Pi$-impedance-conservative system node.  Moreover the relation \eqref{eqY'''}
    leads to the relation
 \begin{equation*}
     \bY \begin{bmatrix} H(w) \\ I \end{bmatrix} u = \begin{bmatrix}
     P(w) H(w) \\ f(w) \end{bmatrix} u.
 \end{equation*}
 We now recover $f(w)$ as the transfer-function for the
 $\Pi$-impedance-conservative system node $\bY$ exactly as in the
 proof of (2)$\Rightarrow$(3) in Theorem \ref{T:HA-Pi}.
\end{proof}

\begin{remark} {\em The implication (2)$\Rightarrow$(1) in Theorem
    \ref{T:HA-Pi'} is due essentially to Agler--McCarthy--Young
    \cite{AMcCY} (without explicit reference to
    $\Pi$-impedance-conservative system nodes), where the result
    is worked out for the scalar-valued case in the
Nevan\-linna--Agler (rather than Herglotz--Agler) setting.
   } \end{remark}

\section{Bessmertny\u{\i} long resolvent representations for
Herglotz--Agler functions}  \label{S:Bes}

Bessmertny\u{\i} long-resolvent  representations were introduced
by Bessmertny\u{\i} in connection with the study of general
rational matrix functions of several variables, with a special
symmetrized form of such a representation handling functions $f$
in the $\Pi^{d}$-Herglotz--Agler class having an extension to
$\Omega_{d}: = \bigcup_{\lambda \in {\mathbb T}} (\lambda \Pi)^{d}
\subset {\mathbb C}^{d}$ satisfying additional symmetry conditions
(see \cite{Bes, Bes1, Bes2, Bes3, Bes4, K-V} and our companion
paper \cite{BK-V} for more detail).  In \cite{B}, a relaxation of
the symmetrized Bessmertny\u{\i} long-resolvent representation was
proposed which handles more general Herglotz--Agler-class
functions (i.e., the homogeneity property is discarded).

Given a $2 \times 2$-block operator pencil
   \begin{equation}   \label{Bes-pen}
   \begin{bmatrix}  \bV_{11}(w) & \bV_{12}(w) \\ \bV_{21}(w) & \bV_{22}(w)
  \end{bmatrix} =: \bV(w) = \bV_{0} + w_{1} \bV_{1} + \cdots +
  w_{d}\bV_{d} \in \cL(\cU \oplus \cX)
   \end{equation}
we define the {\em transfer function} $f_{\bV}(w)$ associated with the
operator pencil $\bV$ by
 \begin{equation}   \label{Bes-charfunc}
    f_{\bV}(w): = \bV_{11}(w) - \bV_{12}(w) \bV_{22}(w)^{-1} \bV_{21}(w).
  \end{equation}
wherever the formula makes sense.  Let us say that a representation
$f(w) = f_{\bV}(w)$ of the form
\eqref{Bes-charfunc} for a given $f$
is a {\em $\cB$-realization}  if the operator pencil $\bV(w)$ satisfies
(1) the {\em $\cB$-symmetry condition} $\bV(w) =
-\bV(-\overline{w})^{*}$, namely
 \begin{equation}   \label{constraint1}
       \bV_{0} + \bV_{0}^{*} = 0, \quad \bV_{k} = \bV_{k}^{*} \text{
       for } k = 1, \dots, d,
    \end{equation}
and (2) the $\cB$-positivity condition
 \begin{equation}   \label{constraint2}
    \bV_{k} = \bV_{k}^{*} \ge 0 \text{ for } k=1, \dots, d \text{
    with } \sum_{k=1}^{d} \bV_{22,k} \text{ strictly positive
    definite.}
    \end{equation}
It was shown in \cite{B} that any function $f$ of the form
\eqref{Bes-charfunc} with pencil $\bV(w)$ satisfying conditions
\eqref{constraint1} and \eqref{constraint2} is  in the
Herglotz--Agler class $\mathcal{HA}(\Pi^{d}, \cL(\cU))$.

 We note that the realization \eqref{real''} in
  Theorem \ref{T:HA-Pi} can formally be considered as
  a $\cB$-realization with operator pencil $\bV(w) = \bV_{0} +
  w_{1} \bV_{1} + \cdots + w_{d} \bV_{d}$ given by
  $$
  \bV_{0} = \begin{bmatrix}  D \& C \\-( B \& A) \end{bmatrix} =
  -\bV_{0}^{*}, \quad \bV_{k} = \begin{bmatrix} 0 & 0 \\ 0 & Y_{k}
\end{bmatrix} = \bV_{k}^{*} \ge 0 \text{ for } k = 1, \dots, d
$$
meeting all the constraints
\eqref{constraint1}--\eqref{constraint2} for a $\cB$-realization
of $f$ with the exception that $\bV_{0}$ is unbounded.  We now make
precise how a general Herglotz--Agler-class function (i.e., with
the growth condition at infinity \eqref{growth} removed) can be
completely characterized in terms of possession of a
nonhomogeneous Bessmertny\u{\i}-type representation of the form
\eqref{Bes-charfunc} subject to \eqref{constraint1} and
\eqref{constraint2} but with possibly unbounded skew-adjoint
operator $\bV_{0}$.

We shall consider unbounded pencils of the form
\eqref{Bes-pen} but with the constant term
$$
\bV_{0} = \begin{bmatrix} \bV_{0,11} & \bV_{0,12} \\ \bV_{0,21} & \bV_{0,22}
\end{bmatrix}
$$
 a possibly unbounded  operator on  $\left[
\begin{matrix} \cU \\ \cX \end{matrix} \right]$ satisfying
    what we shall call the {\em Herglotz--Agler system node}
    properties:
  \begin{enumerate}
      \item[{}] (HA1) {\em $\bV_{0}$ is skew-adjoint on $\left[
      \begin{smallmatrix} \cU \\ \cX \end{smallmatrix} \right]$}.
    \item[{}] (HA2) {\em For each $u \in \cU$ there is an
    $x_{u} \in \cX$ so that $\left[ \begin{smallmatrix} u \\ x_{u}
    \end{smallmatrix} \right] \in \cD(\bV_{0})$.}
 \end{enumerate}
 The gist of conditions (HA1) and (HA2) is that the reorganized
 colligation matrix
 $$ \begin{bmatrix} A \& B \\ C \& D \end{bmatrix} =
 \begin{bmatrix} 0 & - I_{\cX} \\ I_{\cU} & 0 \end{bmatrix}
     \bV_{0} \begin{bmatrix} 0 & I_{\cU} \\ I_{\cX} & 0 \end{bmatrix}
 $$
 is an impedance-conservative system node as discussed in Example
 \ref{E:impconserv}.  To reduce the number of subscripts and to
 suggest the connection with system nodes in the work of
 Staffans et al. \cite{MSW,Staffans-ND, Staffans, BS}, we shall use
 the notation encoded in the following formal definition.

 \begin{definition} \label{D:HApencil}
     We shall say that the pencil \eqref{Bes-pen} is a {\em
     Herglotz--Agler operator pencil} if the following conditions hold:
  \begin{enumerate}
      \item $\bV_{0}$ has the form
      $$
         \bV_{0} = \begin{bmatrix} D \& C \\ -(B \& A) \end{bmatrix} :
     = \begin{bmatrix} 0 & I \\ - I & 0 \end{bmatrix}
     \begin{bmatrix} A \& B \\ C \& D \end{bmatrix}
         \begin{bmatrix} 0 & I \\ I & 0 \end{bmatrix}
   $$
   where $\left[ \begin{matrix} A \& B \\ C \& D
\end{matrix} \right]$ is an impedance-conservative system node
as in Example~\ref{E:impconserv}.

\item The homogeneous part of the pencil,  $\bV_{\rm
H}(w):=\sum_{k=1}^{d} w_{k} \bV_{k}$ has  each
$$  \bV_{k} = \begin{bmatrix} \bV_{k,11} & \bV_{k,12} \\
\bV_{k,21} & \bV_{k,22} \end{bmatrix}
$$
 a bounded positive semidefinite operator on $\cU \oplus \cX$
with the sum having the form
\begin{equation}   \label{HApen2}
\bV_{\rm H}(\be)=\sum_{k=1}^{d} \begin{bmatrix}  \bV_{k,11} & \bV_{k,12} \\
\bV_{k,21} & \bV_{k,22} \end{bmatrix} = \begin{bmatrix} \bV_{{\rm
H},11}(\be) & 0 \\ 0 & I
\end{bmatrix}
\end{equation}
where we use the notation $\be = (1, \dots, 1)$ as in \eqref{growth}.
\end{enumerate}
 \end{definition}

Let us suppose that
$$
\bV(w) =\bV_0+\bV_{\rm H}(w)= \begin{bmatrix} D \& C \\ - (B \& A)
\end{bmatrix} + \begin{bmatrix} \bV_{{\rm H},11}(w) & \bV_{{\rm
H},12}(w) \\ \bV_{{\rm H},21}(w) & \bV_{{\rm H},22}(w) \end{bmatrix}
$$
is a Herglotz--Agler operator pencil.  Thus in particular $\left[
\begin{matrix} 0 & -I \\ I & 0 \end{matrix} \right] \bV_{0}
    \left[
\begin{matrix} 0 & I \\ I & 0 \end{matrix} \right]$ has all
    the properties delineated in Definition \ref{D:sysnode}, with the
    additional property that $\bV_{0} = -\bV_{0}^{*}$ (see Example
    \ref{E:impconserv}).  In particular we may write
 $$
  B \& A  = \left. \begin{bmatrix} B &
 A_{|\cX}\end{bmatrix}\right|_{\cD(\bV_{0})}.
 $$
 Here $A_{|\cX} \colon \cX \to \cX_{-1}$ is the extension of the
 skew-adjoint operator $A$ defined by
 $$
 \cD(A) = \left\{ x \colon \begin{bmatrix} 0 \\ x \end{bmatrix} \in
 \cD(\bV_{0}) \right\}
 \text{ with } Ax  = -\bV_{0} \begin{bmatrix} 0 \\ x
\end{bmatrix}
 $$
and with $B \colon \cU \to \cX_{-1}$ constructed via
$$
  B u =  B \& A \begin{bmatrix} u \\ x_{u} \end{bmatrix} - A_{|\cX} x
$$
where $x_{u}$ is any choice of vector in $\cX$ such that $\left[
\begin{smallmatrix} u \\ x_{u} \end{smallmatrix} \right] \in
    \cD(\bV_{0})$; it can be shown that $B$ is well defined, i.e., the
    formula for $Bu$ is independent of the choice of $x_{u} \in \cX$.
    Furthermore, the domain of $\bV_{0}$ is the same as the domain of $
    B \& A$ and has the precise characterization
$$
\cD(\bV_{0}) = \cD(  B \& A ) = \left\{ \begin{bmatrix} u \\ x
\end{bmatrix} \in \begin{bmatrix} \cU \\ \cX \end{bmatrix} \colon Bu
+ A_{|\cX} x \in \cX \right \}.
$$
The following proposition gives some additional key properties for
Herglotz--Agler operator pencils.

\begin{proposition}   \label{P:HApen}
    Suppose that
 $$
   \bV(w) = \bV_{0} + \bV_{\rm H}(w) : = \begin{bmatrix} D \& C \\ -(B \& A)
   \end{bmatrix} +
   \begin{bmatrix} \bV_{{\rm H},11}(w) & \bV_{{\rm H},12}(w) \\ \bV_{{\rm H},21}(w) &
\bV_{{\rm H},22}(w) \end{bmatrix}
$$
is a Herglotz--Agler operator pencil.  Then, for a given $w \in
\Pi^{d}$, the formal Bessmertny\u{\i} transfer function
 \eqref{Bes-charfunc},
\begin{align*}
    f_{\bV}(w) & = \bV_{11}(w) - \bV_{12}(w)
    \bV_{22}(w)^{-1}\bV_{21}(w) \\
    & = ( D + \bV_{{\rm H},11}(w)) - (C + \bV_{{\rm H},12}(w))
    (\bV_{{\rm H},22}(w) -
    A)^{-1} (-B + \bV_{{\rm H},21}(w)),
\end{align*}
can be interpreted as a bounded operator on $\cU$ equal to the sum of five
well-defined bounded terms
\begin{equation}  \label{fVj}
  f_{\bV}(w) = f_{\bV,1}(w) + f_{\bV,2}(w) + f_{\bV,3}(w) +
  f_{\bV,4}(w) + f_{\bV, 5}(w)
\end{equation}
where
\begin{align}
    f_{\bV,1}(w) & = D \& C \begin{bmatrix} I \\ (\bV_{{\rm H},22}(w) -
    A)^{-1} B \end{bmatrix}, \notag  \\
 f_{\bV,2}(w) & = \bV_{{\rm H},12}(w) \left( (\bV_{{\rm H},22}(w) -
 A)_{|\cX}\right)^{-1} \bV_{{\rm H},21}(w), \notag \\
 f_{\bV,3}(w) & = - C (\bV_{\rm {}H,22}(w) - A)^{-1} \bV_{{\rm H},21}(w), \notag \\
f_{\bV,4}(w) & = - \bV_{{\rm H}, 12}(w) (\bV_{{\rm H},22}(w) - A)^{-1}
\bV_{{\rm H},21}(w), \notag \\
f_{\bV,5}(w)  & = \bV_{{\rm H},11}(w).
\label{fVj'}
\end{align}

Directly in terms of the Bessmertny\u{\i} pencil $\bV(w)$, we have,
for each $w \in \Pi^{d}$,
\begin{align}
  &  \begin{bmatrix} I \\ -\bV_{22}(w)^{-1} \bV_{21}(w) \end{bmatrix}
    \in \cL(\cU,  \cD(\bV_{0})),  \label{Bes1} \\
 & \bV_{11}(w) \& \bV_{12}(w) \in \cL(\cD(\bV_{0}), \cU)   \label{Bes2},
\end{align}
and we recover $f_{\bV}(w)$ as the composition of bounded operators
\begin{equation}   \label{Bes3}
  f_{\bV}(w) = \bV_{11}(w) \& \bV_{12}(w) \cdot
  \begin{bmatrix} I \\ -\bV_{22}(w)^{-1} \bV_{21}(w) \end{bmatrix}
      \in \cL(\cU).
 \end{equation}
\end{proposition}

 \begin{proof}

 {\em Analysis of $f_{\bV,1}$:}  It follows from part (4) of
 Proposition \ref{P:strucresol} that the operator $\sbm{ \left(
 (\bV_{{\rm H},22}(w) -
 A)_{|\cX}\right)^{-1} B  \\ I_{\cU} }$ maps  $\cD(\bV_{0})$ into
 $\cD(\bV_{0})$.    Furthermore, one can check that  $\sbm{ \left(
 (\bV_{{\rm H},22}(w) -
 A)_{|\cX}\right)^{-1} B  \\ I_{\cU} }$ is bounded as an operator
 from $\cU$ to $\cD(\bV_{0})$ (with $\cD(\bV_{0})$ equipped with the
 graph norm of $\bV_{{\rm H},22}(w)$).  As $D \& C$ has domain equal
 to $\cD(\bV_{0})$, we see that $D \& C$ maps $\cD(\bV_{0})$ into
 $\cU$.  Moreover, part (4) of Definition \ref{D:sysnode} assures us
 that  $D \& C$ is bounded as an operator from $\cD(\bV_{0})$
 into $\cU$.  We conclude that $f_{\bV,1}(w) = D \& C \cdot
 \sbm{ \left( (\bV_{{\rm H},22}(w) -
 A)_{|\cX}\right)^{-1} B  \\ I_{\cU} } \in \cL(\cU)$.

 \smallskip

 {\em Analysis of $f_{\bV,2}$:}  By part (2) of Proposition
 \ref{P:strucresol}, $\left( (\bV_{{\rm H},22}(w) - A)_{|\cX}\right)^{-1}
 B$ maps  $\cU$ boundedly into $\cX$.  As
 $\bV_{{\rm H},12}(w)$ is bounded as an operator from $\cX$ into $\cU$, it
 follows that $f_{\bU,2}(w) = \bV_{{\rm H},12}(w) \cdot \left(
 (\bV_{{\rm H},22}(w) - A)_{|\cX}\right)^{-1}
 B$ is bounded as an operator on $\cU$.

 \smallskip

 {\em Analysis of $f_{\bV,3}$:}  It is a consequence of part(2) of
 Proposition \ref{P:strucresol} that the operator $(\bV_{{\rm H},22}(w) - A)^{-1}$ is
 bounded from $\cX$ into $\cX_{1}$.  From the definition
 \eqref{Cdef} of $C$, we see that $C$ maps $\cX_{1}$ boundedly into
 $\cU$.  Since also $\bV_{{\rm H},21}(w)$ is bounded as an operator from
 $\cU$ to $\cX$,  it follows that $f_{\bV,3}(w) = -C \cdot (\bV_{{\rm H},22}(w) - A)^{-1}
 \cdot \bV_{{\rm H},21}(w)$ defines a bounded operator on $\cU$.

 \smallskip

 {\em Analysis of $f_{\bV,4}$:}  Note that $\bV_{{\rm H},21}(w) \in
 \cL(\cU, \cX)$,  by part (3) of Proposition \ref{P:strucresol}
 $(\bV_{{\rm H},22}(w) - A)^{-1}$ is bounded as an operator from $\cX$ into
 $\cX_{1}$ and hence as an operator from $\cX$ into itself, and
 $\bV_{{\rm H},12}(w)$ is bounded as an operator from $\cX$ into $\cU$.  It
 follows that $f_{\bV,4}(w)$, as a composition of bounded operators,
 is bounded as an operator on $\cU$.

 \smallskip

 {\em Analysis of $f_{\bV,5}$:}  This is the easiest term:
 $\bV_{{\rm H},11}(w)$ is a bounded operator on $\cU$ from the definition
 of Herglotz-Agler pencil (Definition \ref{D:HApencil}).

 \smallskip

{\em Verification of formula \eqref{Bes3}:}
We first write out $\bV(w)$ in terms of constant term and homogeneous
part:
$$
  \bV(w) = \bV_{0} + \bV_{{\rm H},k}(w) =
  \begin{bmatrix} D \& C \\ -(B \& A) \end{bmatrix} +
      \begin{bmatrix} \bV_{{\rm H},11}(w) & \bV_{{\rm H},12}(w) \\
      \bV_{{\rm H},21}(w) & \bV_{{\rm H},22}(w) \end{bmatrix}.
$$
Thus
$$
\bV_{11}(w) \& \bV_{12}(w) = D \& C + \begin{bmatrix} \bV_{{\rm
H},11}(w) & \bV_{{\rm H}, 12}(w) \end{bmatrix}.
$$
The first term maps $\cD(\bU_{0})$ boundedly into $\cU$ while the
second term is bounded from the larger space $\sbm{ \cU \\ \cX}$ into
$\cU$.  It follows that the sum indeed is bounded from $\cD(\bU_{0})$
into $\cU$, verifying property \eqref{Bes2}.  Similarly,
\begin{align*}
    \begin{bmatrix} I \\ - \bV_{22}(w)^{-1} \bV_{21}(w) \end{bmatrix} & =
    \begin{bmatrix} I \\ (\bV_{{\rm H}, 22}(w) - A)^{-1}(B -
        \bV_{{\rm H},21}(w) \end{bmatrix} \\
& = \begin{bmatrix} I \\ (\bV_{{\rm H},22}(w) - A)^{-1}B \end{bmatrix}
     - \begin{bmatrix} 0 \\ (\bV_{{\rm H},22}(w) - A)^{-1} \bV_{{\rm H},21}(w) \end{bmatrix}.
\end{align*}
The first term maps $\cU$ boundedly into $\cD(\bU_{0})$ as a
consequence of part (4) of Proposition \ref{P:strucresol} while the
second term maps $\cU$ boundedly into $\sbm{ 0 \\ \cX_{1}} \subset
\cD(\bU_{0})$ (and hence also boundedly into $\cD(\bU_{0})$) by part
(2) of Proposition \ref{P:strucresol}; this verifies property
\eqref{Bes1}.  Thus the composition in the formula \eqref{Bes3}
defines a bounded operator on $\cU$.  Working out the various pieces
in detail, we see that the result agrees with the formula for
$f_{\bV}(w)$ in \eqref{fVj}.
\end{proof}

 The following is the main result of this section.

 \begin{theorem}  \label{T:HA-Pi''}  Given a function $f \colon
     \Pi^{d} \to \cL(\cU)$, the following are equivalent:
     \begin{enumerate}
\item $f$ is in the Herglotz--Agler class $\mathcal{HA}(\Pi^{d},
\cL(\cU))$.

\item $f$ has a $\Pi^{d}$-Herglotz--Agler decomposition, i.e.,
there exist $\cL(\cU)$-valued positive kernels $K_{1}, \dots,
K_{d}$ on $\Pi^{d}$ such that
$$
  f(z)^{*} +  f(w) = \sum_{k=1}^{d} (\overline{z}_{k} + w_{k})
  K_{k}(z,w).
$$

\item There exists a Hilbert space $\cX$ and a Herglotz--Agler
pencil
$$
  \bV(w) = \bV_{0} + \bV_{\rm H}(w) =  \begin{bmatrix} D \& C \\ -(B \& A)
\end{bmatrix} + \begin{bmatrix} \bV_{{\rm H},11}(w) & \bV_{{\rm
H},12}(w) \\ \bV_{{\rm H},21}(w) &
\bV_{{\rm H},22}(w) \end{bmatrix}
$$
such that $f(w) = f_{\bV}(w)$ (with $f_{\bV}(w)$ as in \eqref{fVj}--\eqref{fVj'}
or \eqref{Bes3}).
 \end{enumerate}
 \end{theorem}

 \begin{proof}

     (1) $\Leftrightarrow$ (2):  This is already done in the proof of
     Theorem \ref{T:HA-Pi}.

     \smallskip

    (1) or (2) $\Rightarrow$ (3):  We start with the representation
    for a Herglotz--Agler function $F$ over the polydisk ${\mathbb
    D}^{d}$:
    \begin{equation}   \label{DdHerreal3}
F(\zeta) = R + V^{*} (U - P(\zeta))^{-1} (U + P(\zeta)) V
    \end{equation}
    where $U$ is unitary, $P(\zeta) = \zeta_{1} P_{1} + \cdots + \zeta_{d} P_{d}$
    is a spectral decomposition of the identity on the state space
    $\cX$, $R = \frac{F(0) - F(0)^{*}}{2}$ and $V^{*}V = \frac{
    F(0) + F(0)^{*}}{2}$ as in formula \eqref{DdHerreal2}.
   We use the tuple version of the mutually inverse Cayley changes of
    variable \eqref{Cayley}:
 $$
 \zeta \in {\mathbb D}^{d} \mapsto  \frac{1+\zeta}{1-\zeta} \in
    \Pi^{d}, \quad
  w \in \Pi^{d} \mapsto \frac{w-1}{w+1} \in {\mathbb D}^{d},
$$
 with the conventions \eqref{Cayley1} and \eqref{Cayley2} in force.
 If we set $F(\zeta) = f\left( \frac{1+\zeta}{1- \zeta} \right)$, then $F$ is in
 the Herglotz--Agler class over the polydisk ${\mathbb D}^{d}$ and
 hence we can represent $F(\zeta)$ as in \eqref{DdHerreal3}.  Moreover,
 we recover $f(w)$ from $F(\zeta)$ via $f(w) = F\left(\frac{w-1}{w+1}
 \right)$.  This leads to the formula
 \begin{equation}   \label{fform}
   f(w) = R + V^{*} M(w) V
 \end{equation}
 where we set
 $$
   M(w) = \left( U - P\left( \frac{w-1}{w+1} \right) \right)^{-1}
   \left( U + P\left(\frac{w-1}{w+1} \right) \right).
 $$
 We compute further
 \begin{align}
     M(w) & = \left( U - (P(w) - I) (P(w) + I)^{-1} \right)^{-1}
     \left(U + (P(w) - I) (P(w) + I)^{-1}  \right) \notag \\
     & = \left( P(w) U + U - P(w) + I \right)^{-1} \left( P(w) U + U
     + P(w) - I \right) \notag  \\
     & = \left( P(w) (U - I) + (U + I) \right)^{-1} \left(P(w) (U + I) +
     (U - I) \right)
     \label{compute1}
 \end{align}
 Let us split out the eigenspace of $U$ for eigenvalue $z = 1$ (if
 any) by writing $U$ in the form
 $$
   U = \begin{bmatrix} I & 0 \\ 0 & U_{0} \end{bmatrix}
  $$
  with respect to the decomposition $\cX = \cX^{(1)} \oplus  \cX^{(0)}$
  ($\cX^{(1)}$ equal to the $1$-eigenspace for $U$ and $\cX^{(0)}$
  equal to the orthogonal complement of $\cX^{(1)}$ in $\cX$).
 Then $U_{0}$ is unitary but does not have $1$ as an eigenvalue.
 We have
 $$
   U - I = \begin{bmatrix} 0 & 0 \\ 0 & U_{0} - I \end{bmatrix}, \quad
   U + I = \begin{bmatrix} 2 I & 0 \\ 0 & U_{0} + I \end{bmatrix}.
 $$
 and hence
\begin{multline*}
 M(w) = \left( P(w) \left[ \begin{matrix} 0 & 0 \\ 0 & U_{0}-I
 \end{matrix} \right] + \left[ \begin{matrix} 2I & 0 \\ 0 &
 U_{0} + I \end{matrix} \right] \right)^{-1}\\
\cdot \left( P(w) \left[ \begin{matrix} 2I & 0 \\ 0 & U_{0} +I
\end{matrix} \right] + \left[ \begin{matrix} 0 & 0 \\ 0 &
U_{0}-I \end{matrix} \right] \right)
\end{multline*}
Let us set
$$
  T = (I + U_{0}) (I - U_{0})^{-1}.
$$
Then $T$ is  a possibly unbounded skew-adjoint operator on $\cX_{0}$
(by Proposition \ref{P:unitary/skewadj})
and we have the following relations between $T$ and $U_{0}$:
\begin{align*}
    U_{0} & = (T-I) (T+I)^{-1} = I - 2 (T+I)^{-1} \\
    & \Rightarrow U_{0} - I = -2(T+I)^{-1}, \quad (U_{0} - I)^{-1} =
    - \frac{1}{2} (T + I).
\end{align*}
We may then continue the computation \eqref{compute1} to get
\begin{align}
    M(w) & = \left[ \begin{matrix} \frac{1}{2} I & 0 \\ 0 &
    (U_{0} - I)^{-1} \end{matrix} \right] \left( P(w) \left[
    \begin{matrix} 0 & 0 \\ 0 & I \end{matrix} \right] +
    \left[ \begin{matrix} I & 0 \\ 0 & (U_{0} + I) (U_{0} -
    I)^{-1} \end{matrix} \right] \right)^{-1} \cdot \notag  \\
& \quad \cdot  \left( P(w) \left[ \begin{matrix} I & 0 \\ 0 &
 (U_{0}+I) (U_{0} - I)^{-1} \end{matrix} \right] + \left[
 \begin{matrix} 0 & 0 \\ 0 & I \end{matrix} \right] \right)
     \left[ \begin{matrix} 2 I & 0 \\ 0 & U_{0}-I
 \end{matrix} \right]  \notag \\
 & = \left[ \begin{matrix} I & 0 \\ 0 & -(T+I) \end{matrix}
 \right] \left( P(w) \left[ \begin{matrix} 0 & 0 \\ 0 & I
\end{matrix} \right] + \left[ \begin{matrix} I & 0 \\ 0 &
-T \end{matrix} \right] \right)^{-1} \cdot \notag \\
& \quad \cdot \left( P(w) \left[ \begin{matrix} I & 0 \\ 0 & -T
\end{matrix} \right] + \left[ \begin{matrix} 0 & 0 \\ 0 & I
\end{matrix} \right] \right) \left[ \begin{matrix} I & 0
\\ 0 & -(T + I)^{-1} \end{matrix} \right] \notag \\
& = \left[ \begin{matrix} I & 0 \\ 0 & -(T+I) \end{matrix} \right]
N(w) \left[ \begin{matrix} I & 0 \\ 0 & -(T+I)^{*}
\end{matrix} \right]
\label{M1}
 \end{align}
 where, due to the identity $-(T+I)^{*} = T - I$ arising from $T =
 -T^{*}$,  $N(w)$ is given by
 \begin{multline}
 N(w) =  \left(P(w) \left[ \begin{matrix} 0 & 0 \\ 0 & I
\end{matrix} \right] +
\left[ \begin{matrix} I & 0 \\ 0 & -T \end{matrix} \right]
\right)^{-1} \\
\cdot \left(P(w) \left[ \begin{matrix} I & 0 \\ 0 & -T
\end{matrix} \right] + \left[ \begin{matrix} 0 & 0 \\ 0 & I
\end{matrix} \right] \right) \cdot \left[ \begin{matrix} I
& 0 \\ 0 & (I - T^{2})^{-1} \end{matrix} \right].\label{N1}
\end{multline}
If we write out the block matrix decomposition
$$
P(w) = \begin{bmatrix} P_{11}(w) & P_{10}(w) \\ P_{01}(w) & P_{00}(w) \end{bmatrix}
= \sum_{k=1}^{d} w_{k} \begin{bmatrix} P_{k,11} & P_{k,10}\\ P_{k,01} & P_{k,00} \end{bmatrix}
$$
of $P(w)$ with respect to the decomposition $\cX = \left[
\begin{matrix} \cX^{(1)} \\ \cX^{(0)} \end{matrix}\right]$
    of $\cX$, we can write out more explicitly
    \begin{multline*}
    \left( P(w) \begin{bmatrix} 0 & 0 \\ 0 & I \end{bmatrix} +
    \begin{bmatrix} I & 0 \\ 0 & -T \end{bmatrix} \right)^{-1}=
        \begin{bmatrix} I & P_{10}(w) \\ 0 & P_{00}(w) - T
        \end{bmatrix}^{-1}\\ =
\begin{bmatrix} I & - P_{10}(w) (P_{00}(w) - T)^{-1} \\ 0 &
    (P_{00}(w) - T)^{-1} \end{bmatrix},
    \end{multline*}
\begin{equation*}  P(w) \begin{bmatrix} I & 0 \\ 0 & -T \end{bmatrix}  +
  \begin{bmatrix}  0 & 0 \\ 0 & I \end{bmatrix}  =
      \begin{bmatrix} P_{11}(w) & -P_{10}(w) T \\ P_{01}(w) & I -
      P_{00}(w) T \end{bmatrix}
\end{equation*}
and from \eqref{N1} we see that $N(w)$ is given by
\begin{equation}   \label{N2}
  N(w) = \begin{bmatrix} I & -P_{10}(w) (P_{00}(w) - T)^{-1} \\ 0 &
  (P_{00}(w) - T)^{-1} \end{bmatrix}
  \begin{bmatrix} P_{11}(w)  & -P_{10}(w) T (I - T^{2})^{-1}  \\
      P_{01}(w) & (I - P_{00}(w) T) (I- T^{2})^{-1} \end{bmatrix}.
\end{equation}
At this stage it is convenient to introduce the Gelfand triple (or
rigging) of $\cX^{(0)}$ associated with the (possibly unbounded)
skew-adjoint operator $T$:
\begin{align*}
    & \cX_{1}^{(0)}: = \operatorname{Dom} T = \operatorname{Ran} (I -
    T)^{-1}, \\
  & \cX^{(0)}_{-1}:= \text{completion of $\cX^{(0)}$ in
  $\cX^{(0)}_{-1}$-norm:} \ \| x\|_{-1} = \| (I - T)^{-1} x
  \|_{\cX^{(0)}} \text{ for } x \in \cX^{(0)}.
 \end{align*}
 Then we see that $(I-T)^{-1}$ is well defined as an element of $\cL(\cX^{(0)},
 \cX^{(0)}_{1})$ and of $\cL(\cX^{(0)}_{-1}, \cX^{(0)})$.
 A careful inspection of the formula \eqref{N2} for $N(w)$ shows that
 \begin{equation}  \label{Ndomran}
   N(w) \colon \begin{bmatrix} \cX^{(1)} \\ \cX^{(0)}_{-1}
\end{bmatrix} \to \begin{bmatrix} \cX^{(1)} \\ \cX^{(0)}_{1}
\end{bmatrix}
\end{equation}
from which it follows that the formula \eqref{M1} gives sense for
$M(w)$ as an element of $\cL(\cX^{(1)} \oplus \cX^{(0)}) $.
However the formula \eqref{N2} (and \eqref{N1}) lacks symmetry. To
fix this we introduce the operator $J$ by
\begin{equation}   \label{defJ}
    J: = -T(I - T^{2})^{-1},
\end{equation}
Thus $J = -J^{*}$ and we can consider $J$ as an element of
$\cL(\cX^{(0)}, \cX^{(0)}_{1})$ as well as $\cL(\cX^{(0)}_{-1},
\cX^{(0)})$.  While the operator $I - P_{00}(w) T$ makes sense as an
element of $\cL(\cX^{(0)}_{1}, \cX^{(0)})$ (as well as
$\cL(\cX^{(0)}, \cX^{(0)}_{-1})$), the individual terms in the additive
decomposition
\begin{equation}   \label{decomposition}
  I - P_{00}(w) T = (I - T^{2}) + (T^{2} - P_{22}(w) T)
\end{equation}
make sense only as elements in $\cL(\cX^{(0)}_{1}, \cX^{(0)}_{-1})$.
Nevertheless, we proceed to get a more symmetric formula for $N(w)$
as follows.  Note first that the decomposition \eqref{decomposition}
leads to
\begin{align*}
    (I - P_{00}(w) T) (I - T^{2})^{-1} & =
    I - (P_{00}(w) - T) T (I - T^{2})^{-1} \\
    & = I + (P_{00}(w) - T) J \colon \cX^{(0)} \to \cX^{(0)}.
\end{align*}
From \eqref{N2} and the definition \eqref{defJ} of $J$, we then have
\begin{align}
    N(w) & = \begin{bmatrix} I & -P_{10}(w) (P_{00}(w) - T)^{-1} \\ 0
    & (P_{00}(w) - T)^{-1} \end{bmatrix} \begin{bmatrix} P_{11}(w) &
    P_{10}(w) J \\ P_{01}(w) & I + (P_{00}(w) - T) J \end{bmatrix}
    \notag \\
    & = \begin{bmatrix} P_{11}(w) - P_{10}(w) (P_{00}(w) - T)^{-1}
    P_{01}(w) & - P_{10}(w) (P_{00}(w) - T)^{-1} \\ (P_{00}(w) -
    T)^{-1} P_{01}(w) & (P_{00}(w) - T)^{-1} + J \end{bmatrix}
    \label{N3}
\end{align}
which a priori makes sense only as an operator from $\left[
\begin{matrix} \cX^{(1)} \\ \cX^{(0)}_{-1} \end{matrix} \right]$
    to $\left[ \begin{matrix} \cX^{(1)} \\ \cX^{(0)}
\end{matrix} \right]$ rather than to  $\left[
\begin{matrix} \cX^{(1)} \\ \cX^{(0)}_{1}
\end{matrix} \right]$ as in \eqref{Ndomran}, due to the
decoupling of an $\infty - \infty$ cancellation occurring in the application of the
decomposition \eqref{decomposition}.  This in turn leads to
difficulties in understanding $M(w)$ as a bounded operator on
$\cX^{(1)} \oplus \cX^{(0)}$ from the formula \eqref{M1}.

\smallskip

\textbf{Continuation of the analysis with an extra assumption:}
{\em Assuming for the moment that $T$ is bounded} (as is the case for the
situation where the state space $\cX$ is finite-dimensional as
in the setting discussed in \cite{BK-V}), this difficulty does
not occur and we may continue as follows.  From \eqref{N3} we
see that
\begin{equation}   \label{N4}
N(w) = \begin{bmatrix} P_{11}(w) & 0 \\ 0 & J \end{bmatrix} -
\begin{bmatrix} P_{10}(w) \\ - I \end{bmatrix} (P_{00}(w) - T)^{-1}
    \begin{bmatrix} P_{01}(w) & I \end{bmatrix}.
\end{equation}
If we block-decompose the operator $V \colon \cU \to \cX = \left[
\begin{matrix} \cX^{(1)} \\ \cX^{(0)} \end{matrix} \right]$ as $V =
    \left[ \begin{matrix} V_{1} \\ V_{0} \end{matrix}
    \right]$ and then combine \eqref{fform} with \eqref{M1} and
    \eqref{N4} while noting the simplification
    \begin{align*}
    (I+T) J (I + T)^{*} &  = (I+T) [-(I+T)^{-1} T (I-T)^{-1}]
    (I+T)^{*} \text{ (by \eqref{defJ})}\\
    & = -T  \text{ (since $T^{*} = -T$)}.
    \end{align*}
 we arrive at
\begin{align}
  &  f(w)  = R + V_{1}^{*} P_{11}(w) V_{1} - V_{0}^{*} T V_{0} \notag \\
    & -[ V_{1}^{*}P_{10}(w) + V_{0}^{*}(I + T) ] (P_{00}(w) - T)^{-1}
    [P_{01}(w) V_{1} - (I + T)^{*} V_{0}].
  \label{fform'}
    \end{align}
    We have arrived at a Bessmertny\u{\i} long-resolvent representation
    for $f$
    $$ f(w) = f_{\bV}(w)
    $$
    where the operator pencil $\bV(w) = \left[ \begin{matrix}
    \bV_{11}(w) & \bV_{12}(w) \\ \bV_{21}(w) & \bV_{22}(w) \end{matrix}
    \right]$ is given by
    \begin{align*}
    \bV_{11}(w) & = R + V_{0}^{*} (I + T) J (I + T)^{*} V_{0} +
    V_{1}^{*} P_{11}(w) V_{1}, \\
    \bV_{12}(w) & = V_{0}^{*} (I + T) + V_{1}^{*} P_{10}(w), \\
    \bV_{21}(w) & = -(I + T)^{*} V_{0} + P_{01}(w) V_{1}, \\
    \bV_{22}(w) & = -T + P_{00}(w).
\end{align*}
Thus the associated linear pencil
\begin{equation}    \label{A(w)}
  \bV(w) = \begin{bmatrix} \bV_{11}(w) & \bV_{12}(w) \\ \bV_{21}(w) &
  \bV_{22}(w) \end{bmatrix} = \bV_{0} + w_{1} \bV_{1} + \cdots +
  w_{d} \bV_{d}
\end{equation}
has coefficients
\begin{align}
    \bV_{0}& = \begin{bmatrix} R - V_{0}^{*}T V_{0} & V_{0}^{*} (I + T)
    \\ -(I + T)^{*} V_{0} & -T \end{bmatrix}  \notag \\
    & =
    \begin{bmatrix} R & 0 \\ 0 & 0 \end{bmatrix} +
    \begin{bmatrix} V_{0}^{*} & 0 \\ 0 & I \end{bmatrix}
        \begin{bmatrix} -T & I+T \\ -(I+T)^{*} & -T \end{bmatrix}
        \begin{bmatrix} V_{0} & 0 \\ 0 & I \end{bmatrix},
    \label{A0} \\
    \bV_{k} & = \begin{bmatrix} V_{1}^{*} P_{k,11} V_{1} & V_{1}^{*}
    P_{k,10} \\ P_{k,01} V_{1} & P_{k,00} \end{bmatrix} \text{ for }
    k=1, \dots, d.  \label{Ak}
\end{align}
Moreover, it is easily checked that $\bV_{0}$ is skew-adjoint and
that $P_{k} \ge 0$ for each $k$ with $\sum_{k=1}^{d} P_{k,00} =
I_{\cX^{(0)}}$ (since $\sum_{k=1}^{d} P_{k} = I_{\cX}$), and hence
$\bV(w)$ is a Herglotz--Agler pencil and Theorem \ref{T:HA-Pi''} is
completely proved in case $T = (I + U_{0}) (I - U_{0})^{-1}$ is
bounded on $\cX^{(0)}$.

\smallskip

\textbf{Back to the general case:}
For the general case (where $T$ is allowed to be unbounded), the
formula \eqref{Ak} for $\bV_{k}$ still makes good sense and the
$\bV_{k}$'s meet property (2) in Definition \ref{D:HApencil}. The next
step is to  make sense of the formula \eqref{A0} for $\bV_{0}$.

The first term in the formula \eqref{A0} for $\bV_{0}$ can always be
added in later so we focus on the second term $\bV_{0}'$:
\begin{equation*}
    \bV_{0}' = \begin{bmatrix} V_{0}^{*}  & 0 \\ 0 & I\end{bmatrix}
    \begin{bmatrix} -T & I+T \\ -(I+T)^{*} & -T \end{bmatrix}
    \begin{bmatrix} V_{0} & 0 \\ 0 & I \end{bmatrix}.
\end{equation*}
We view $\bV_{0}'$ as a possibly unbounded operator with dense
domain in $\left[ \begin{matrix} \cU \\ \cX \end{matrix} \right]$
given by
$$
 \cD(\bV_{0}') = \left\{ \begin{bmatrix} u \\ x \end{bmatrix} \in
 \begin{bmatrix} \cU \\ \cX \end{bmatrix} \colon x - V_{0} u \in
     \cD(T) \right\}.
$$Then the flip of $\bV_{0}'$, namely
$$
 \begin{bmatrix} 0 & -I \\ I & 0 \end{bmatrix} \bV_{0}'
     \begin{bmatrix} 0 & I \\ I & 0 \end{bmatrix}
 = \begin{bmatrix} I  & 0 \\ 0 &  V_{0}^{*}\end{bmatrix}
 \begin{bmatrix} T & I-T \\ I+T & -T \end{bmatrix}
     \begin{bmatrix} I & 0 \\ 0 & V_{0} \end{bmatrix}
$$
has exactly the form of the model $\Pi$-impedance-conservative
system node given in Proposition \ref{P:impconserv}. We can now
conclude that $\bV(w)$ given by \eqref{A(w)}, \eqref{A0}, \eqref{Ak}
is indeed a Herglotz--Agler pencil.  It remains only to verify
that $f(w) = f_{\bV}(w)$.

From the formula \eqref{fform} for $f$ combined with the formula
\eqref{M1} for $M(w)$ and the formula \eqref{N3} for $N(w)$, we know
that $f(w)$ has the representation
\begin{align}
f(w) & = R + V_{1}^{*}P_{11}(w) V_{1} - V_{1}^{*}P_{10}(w)
 (P_{00}(w) - T)^{-1} P_{01}(w) V_{1} \notag  \\
& \quad -V_{0}^{*} (I+T)(P_{00}(w) - T)^{-1} P_{01}(w) V_{1}
 \notag \\
& \quad + V_{1}^{*}P_{10}(w) [ T (I+T)^{-1}
+ (P_{00}(w) -T)^{-1}(I - P_{00}(w)T) (I+T)^{-1}] V_{0}  \notag \\
& \quad +V_{0}^{*}(I+T) (P_{00}(w) - T)^{-1} (I - P_{00}(w) T) (I+T)^{-1} V_{0}.
    \label{fform''}
\end{align}
On the other hand the Bessmertny\u{\i} transfer function
associated with the pencil $\bV$ \eqref{A(w)} can be written as
\begin{align}
    f_{\bV}(w) & = [\bV_{11}(w) \& \bV_{12}(w) ] \begin{bmatrix} I
    \\ - \bV_{22}(w)^{-1} \bV_{21}(w) \end{bmatrix} \notag \\
    & =
     R + V_{0}^{*}\left[ -T + (I+T) (P_{00}(w) -
    T)^{-1} (I-T) \right]V_{0}  \notag \\
    & \quad  + V_{1}^{*} P_{11}(w) V_{1} - V_{1}^{*} P_{10}(w)
    (P_{00}(w) - T)^{-1} P_{01}(w) V_{1} \notag \\
    & \quad - V_{0}^{*}(I+T) (P_{00}(w) - T)^{-1} P_{01}(w) V_{1}
    \notag \\
    &\quad  + V_{1}^{*} P_{10}(w) (P_{00}(w) - T)^{-1} (I-T)V_{0}.
    \label{fA}
\end{align}
Note that care must be taken in writing the first term of the
expression after the $R$ term:  the individual expressions
$-V_{0}^{*} T V_{0}$ and $V_{0}^{*} (I + T) (P_{00}(w) -
T)^{-1}(I-T) V_{0}$ make no sense since the operators $-T$ and $(I
+ T) (P_{00}(w) - T)^{-1}(I-T)$ map   $\cX^{(0)}$ into
$\cX^{(0)}_{-1}$ and $V_{0}^{*} \in \cL(\cX^{(0)}, \cU)$ has no
extension to $\cX^{(0)}_{-1}$; as we shall see in detail below,
the combination $-T + (I + T) (P_{00}(w) - T)^{-1}(I-T)$
fortuitously maps $\cX^{(0)}$ back into itself so that the
combined term $V_{0}^{*}(-T + (I + T) (P_{00}(w) -
T)^{-1}(I-T))V_{0}$ makes good sense as a bounded operator on
$\cU$ for each $w \in \Pi^{d}$; roughly speaking, this is where we
couple back together the $\infty - \infty$ cancellation introduced
earlier to make our formulas once again make sense.  Note also
that the formula \eqref{fA} agrees with the formula \eqref{fform'}
once one takes care to rearrange the terms in \eqref{fform'} so
that the result makes sense as a well-defined bounded operator on
$\cU$ defining the operator $f(w)$.

By the analysis done above with the extra assumption imposed, we
see that the two expressions \eqref{fform''} and \eqref{fA} agree
in the special case where the skew-adjoint operator $T$ is
bounded. Once the operator $-T + (I+T) (P_{00}(w) - T)^{-1}
P_{01}(w) (I-T)$ is exhibited more explicitly as a bounded
operator on $\cX^{(0)}$ (even in the case where $T$ itself is
unbounded), it is possible to verify the equality of the two
expressions \eqref{fform''} and \eqref{fA} by approximating the
unbounded case by the bounded case and then taking limits.  As
this is really about algebra, however, perhaps more satisfying is
to verify the equality between \eqref{fform''} and \eqref{fA}
directly by brute-force algebra.

Toward this goal, we note that each term in \eqref{fform''} can be
paired with an identical term in \eqref{fA} once we establish the
validity of the two identities:
\begin{align}
 & (I+T)(P_{00}(w) - T)^{-1} (I - P_{00}(w) T) (I + T)^{-1}  \notag
 \\  & \quad =
 -T + (I+T) (P_{00}(w) - T)^{-1} (I-T)  \label{id1''}  \\
 & P_{10}(w)\,[T(I+T)^{-1} + (P_{00}(w) - T)^{-1} (I - P_{00}(w) T) (I
 + T)^{-1}] \notag  \\
 & \quad  = P_{10}(w)(P_{00}(w) - T)^{-1} (I-T).
\notag
\end{align}
In particular, \eqref{id1''} demonstrates how the expression $-T +
(I+T) (P_{00}(w) - T)^{-1} (I-T)$ actually defines a bounded operator
on $\cX$.  These two identities can be verified directly by
brute-force algebra; we leave the details to the reader (or as an
exercise for {\tt MATHEMATICA}).  This concludes the proof of (1) or
(2) $\Rightarrow$ (3) in Theorem \ref{T:HA-Pi''}.

     \smallskip

     (3) $\Rightarrow$ (2):  We assume that $f(w)  = f_{\bV}(w)$ for a
     Herglotz--Agler pencil
     $$
     \bV(w) = \begin{bmatrix} \bV_{11}(w) \& \bV_{12}(w) \\ \bV_{21}(w) \&
     \bV_{22}(w) \end{bmatrix} =
     \begin{bmatrix} D \& C \\ - (B \& A) \end{bmatrix} +
     \begin{bmatrix} \bV_{{\rm H},11}(w) & \bV_{{\rm H},12}(w) \\ \bV_{{\rm H},21}(w) &
         \bV_{{\rm H},22}(w) \end{bmatrix}.
$$
Thus, as explained in Proposition \ref{P:HApen}, if we set $x_{w} = - \bV_{22}(w)^{-1} \bV_{21}(w)$
for $w \in \Pi^{d}$, then for each $u \in \cU$ we have
$$
 \begin{bmatrix} u \\ x_{w}u \end{bmatrix} \in \cD\left(
     \begin{bmatrix} \bV_{11}(w) \& \bV_{12}(w) \\ \bV_{21}(w) \& \bV_{22}(w)
     \end{bmatrix} \right)
$$
and
\begin{equation}   \label{note1}
    f(w) u = \left( \bV_{11}(w) \& \bV_{12}(w) \right) \begin{bmatrix} u
    \\ x_{w} u \end{bmatrix}
 \end{equation}
 We may also compute
 \begin{align*}
& \left( \bV_{21}(w) \& \bV_{22}(w) \right) \begin{bmatrix} u \\
 -\bV_{22}(w)^{-1} \bV_{21}(w) u \end{bmatrix}  \\
 &  =
  \begin{bmatrix}  \bV_{21}(w) & \bV_{22}(w)
\end{bmatrix}  \begin{bmatrix} u \\
 -\bV_{22}(w)^{-1} \bV_{21}(w) u \end{bmatrix}
 \text{ (as a vector in $\cX^{(0)}_{-1}$)}  \\
 & = \bV_{21}(w) u - \bV_{21}(w) u = 0.
 \end{align*}
 Thus \eqref{note1} can be expanded to the identity
 \begin{equation}  \label{note2}
     \begin{bmatrix} \bV_{11}(w) \& \bV_{12}(w) \\ \bV_{21}(w) \& \bV_{22}(w)
     \end{bmatrix} \begin{bmatrix} u \\ x_{w} u \end{bmatrix} =
     \begin{bmatrix} f(w) \\ 0 \end{bmatrix} u.
\end{equation}
In addition to $u \in \cU$ and $w \in \Pi^{d}$, choose another pair
$u' \in \cU$ and $z \in \Pi^{d}$ and consider the sesquilinear form
\begin{align}
& {\mathbb Q}(z,w)[u,u'] : =   \notag   \\
& \quad \left\langle \begin{bmatrix} \bV_{11}(w) \& \bV_{12}(w) \\ \bV_{21}(w)
    \& \bV_{22}(w) \end{bmatrix} \begin{bmatrix} u \\ x_{w} u
\end{bmatrix}, \, \begin{bmatrix} u' \\ x_{z} u' \end{bmatrix}
\right\rangle + \left\langle \begin{bmatrix} u \\
x_{w} u \end{bmatrix}, \, \begin{bmatrix} \bV_{11}(z) \& \bV_{12}(z) \\
\bV_{21}(z) \& \bV_{22}(z) \end{bmatrix} \begin{bmatrix} u' \\ x_{z} u'
\end{bmatrix} \right \rangle.
 \notag
\end{align}
As a consequence of \eqref{note2} we see that
\begin{align}
    {\mathbb Q}(z,w)[u,u'] &  = \left\langle \begin{bmatrix} f(w) u \\ 0
    \end{bmatrix}, \begin{bmatrix} u' \\ x_{z} u' \end{bmatrix}
    \right \rangle + \left\langle \begin{bmatrix} u \\ x_{w} u
\end{bmatrix}, \, \begin{bmatrix} f(z) u' \\ 0 \end{bmatrix} \right
\rangle  \notag \\
& = \langle \left( f(z)^{*} + f(w) \right) u, u' \rangle.
\label{note3}
\end{align}
On the other hand, from the decomposition of $\bV(w)$ as $\bV(w) =
\bV_{0} + \bV_{\rm H}(w)$ with $\bV_{0} = - \bV_{0}^{*}$, we have
\begin{multline}
    {\mathbb Q}(z,w)[u,u']  =
\Big \langle  \bV(w)
\begin{bmatrix} u \\ x_{w} u \end{bmatrix}, \,
\begin{bmatrix} u' \\ x_{z} u' \end{bmatrix} \Big \rangle +
    \Big \langle \begin{bmatrix} u \\ x_{w} u \end{bmatrix}, \,
    \bV(z) \begin{bmatrix} u' \\ x_{z} u'
\end{bmatrix} \Big \rangle   \\
 = \Big \langle \left( \bV_{\rm H}(z)^{*} + \bV_{\rm H}(w) \right) \begin{bmatrix} I \\
- \bV_{22}(w)^{-1} \bV_{21}(w) \end{bmatrix} u, \,
\begin{bmatrix} I \\ - \bV_{22}(z)^{-1} \bV_{21}(z) \end{bmatrix} u'
    \Big \rangle  \\
  = \sum_{j=1}^{d} (\overline{z}_{j} + w_{j}) \langle H(z)^{*}
    \bV_{j} H(w) u, u' \rangle
    \label{note4}
 \end{multline}
 where we set
 $H(w) = \begin{bmatrix} I \\ -\bV_{22}(w)^{-1} \bV_{21}(w) \end{bmatrix}
 \in \cL(\cU, \cU \oplus \cX)$.  Combining \eqref{note3} and
 \eqref{note4} gives us
 $$
   f(z)^{*} + f(w) = \sum_{j=1}^{d} (\overline{z}_{j} + w_{j})
   H(z)^{*} \bV_{j} H(w)
 $$
 where $\bV_{j} \ge 0$ on $\cU \oplus \cX$ by assumption, and (2)
 follows.
 \end{proof}

 We next illustrate Theorem \ref{T:HA-Pi''} by looking at some
 special cases.

 \smallskip

 \noindent
 \textbf{Special case 1:  $V_{1} = 0$.}  We note that the case
 $V_{1}=0$ in the proof of Theorem \ref{T:HA-Pi''} is exactly the
 case where the representation \eqref{fA} for $f$ collapses to
 \begin{equation}  \label{fAproper}
 f(w) = R + V_{0}^{*}[-T + (I+T) (P_{00}(w) - T)^{-1} (I-T)]
 V_{0}.
 \end{equation}
 where $P_{00}(w) = w_{1} P_{1,00} + \cdots + w_{d} P_{d,00}$ is a
 positive decomposition of $I_{\cX^{(0)}}$, i.e., $f$ has
 a representation exactly as in part (3) of Theorem \ref{T:HA-Pi}.
 In general, from the property \eqref{HApen2} for a Herglotz--Agler
 pencil, we have
 $$
   \bV_{\rm H}(t \be) = t \bV_{\rm H}(\be) = t \begin{bmatrix} \bV_{{\rm H},11}(\be) & 0 \\ 0 & I
\end{bmatrix}.
$$
Thus, in the general representation \eqref{fA} for $f$, $f(t \be)$
takes on the simplified form
$$
f(t \be) = R + V_{0}^{*}[ -T + (I+T) (tI - T)^{-1} (I-T) ] V_{0}
 + t V_{1}^{*} V_{1}.
$$
We have already seen that
$$
\lim_{t \to +\infty} \frac{1}{t} [ R + V_{0}^{*}(-T + (I+T)
(tI -T)^{-1} (I-T))V_0 ] = 0.
$$
We conclude that in general
$$
\lim_{t \to +\infty} \frac{1}{t} f(t \be) =  V_{1}^{*} V_{1}.
$$
Thus the growth condition at $\infty$ \eqref{growth} is equivalent
to the condition that $V_{1} = 0$.  In this way we arrive at Theorem
\ref{T:HA-Pi} as a corollary of Theorem \ref{T:HA-Pi''}.

\smallskip

\noindent \textbf{Special case 2:  $\cX^{(1)} = \{0\}$.}   This
corresponds to the case   where $P_{k} = P_{k,00}$ and $P_{00}(w)
= w_{1} P_{1,00} + \cdots + w_{d} P_{d,00}$ is a spectral
decomposition (not just a positive decomposition) of
$I_{\cX^{(0)}}$.  Then the representation \eqref{fA} collapses
again to \eqref{fAproper}, but this time with the stronger
property that $w_{1} P_{1,00} + \cdots + P_{d,00} w_{d}$ is a
spectral decomposition of $I_{\cX^{(0)}}$, i.e., exactly the
conclusion of part (1) of Theorem \ref{T:HA-Pi'}.  On the one
hand, the condition $\cX^{(1)} = \{0\}$ means that the unitary
operator $U$ in the Herglotz representation \eqref{DdHerreal2} for
the function $F \in \mathcal{HA}({\mathbb D}^{d}, \cU)$ given by
$$
  F(\zeta) = f\left(\frac{1+\zeta}{1- \zeta} \right)
$$
does not have $1$ as an eigenvalue.  On the other hand, condition
(1) in Theorem \ref{T:HA-Pi''} is that the colligation matrix $\bU
= \left[ \begin{matrix} \bA & \bB \\ \bC & \bD
\end{matrix} \right]$ in the unitary Givone--Roesser
representation
\begin{equation}   \label{GR-S}
  S(\zeta) = \bD + \bC (I - P(\zeta) \bA)^{-1} P(\zeta) \bB
\end{equation}
for the $\mathcal{SA}({\mathbb D}^{d}, \cU)$-class function
\begin{equation}  \label{SfromForf}
  S(\zeta) = \left( f\left( \frac{1+\zeta}{1- \zeta} \right) - I
  \right)
  \left( f\left( \frac{1+\zeta}{1-\zeta} \right) + I \right)^{-1}
 = (F(\zeta) - I) (F(\zeta) + I)^{-1}
\end{equation}
does not have $1$ as an eigenvalue.  To see that these conditions
match up, we recall from the discussion in Remark \ref{R:sysnodes}
(note formulas \eqref{DdHerreal2'}, \eqref{Agler-rep} and
\eqref{DdHerreal2}) that we have the following connection between
the unitary operator $U$ in the Herglotz representation
\eqref{DdHerreal2} for $F(\zeta)$  and the unitary colligation matrix
$\bU = \left[ \begin{matrix} \bA & \bB \\ \bC & \bD
\end{matrix} \right]$ generating the  Givone--Roesser
representation \eqref{GR-S} for $S(\zeta)$ given by \eqref{SfromForf}:
$$
   U^{*} = \bU_{0}: = \bA - \bB(I - \bD)^{-1} \bC.
$$
As was observed in Remark \ref{R:HAdisk},  $I - \bD$ is injective
since $I-\bD = I - S(0) = 2 (F(0) + I)^{-1}$.  By the theory of Schur
complements, given that $I- \bD$ is injective, then  $I- \bU$ is
injective if and only if its Schur complement $\bU_{0} = I - \bB (I -
\bD)^{-1} \bC$ is injective.  As $\bU_{0}$ is unitary, this in turn
is equivalent to $U = (\bU_{0})^{*}$ not having $1$ as an
eigenvalue.  In this way we recover Theorem \ref{T:HA-Pi'} as a
corollary of Theorem \ref{T:HA-Pi''}.

\smallskip
\noindent \textbf{Special case 3: $X^{(0)}=\{0\}$.} In this case,
$M(w)=P(w)$, and the representation \eqref{fA} collapses to
$$f(w)=\bV(w)=R+V^*P(w)V.$$

 \smallskip

 \noindent
 \textbf{Special case 4: $R=0$, $U=U^*$, $V_0=0$.} In this
 case, $T=T^*=-T^*$, hence $T=0$, and the linear pencil $A(w)$ is
 homogeneous: $$\bV(w)=\bV_{\rm H}(w)=\begin{bmatrix}
 V_1^* & 0\\
 0 & I
 \end{bmatrix}P(w)\begin{bmatrix}
 V_1 & 0\\
 0 & I
 \end{bmatrix}.$$
 Moreover, $\bV(\mathbf{e})=\begin{bmatrix}
 V_1^*V_1 & 0\\
 0 & I
 \end{bmatrix}$, as in \eqref{HApen2}.
 The representation
\begin{align}
f(w) &=V_1^*(P_{11}(w)-P_{10}(w)P_{00}(w)^{-1}P_{01}(w))V_1 \label{E:Bessm1}\\
&=\bV_{{\rm H},11}(w)-\bV_{{\rm H},12}(w) \bV_{{\rm H},22}(w)^{-1} \bV_{{\rm
H},21}\label{E:Bessm2}
\end{align}
is then Bessmertny\u{\i}'s long-resolvent representation in the
infinite-dimensional setting, as in \cite{K-V}, i.e., $f$ belongs
to the \emph{Bessmertny\u{\i} class} $\mathcal{B}_d(\mathcal{U})$.

\begin{remark}\label{rem:pencil-normal}
\emph{Notice that for a function $f\in\mathcal{B}_d(\mathcal{U})$
one can always find a homogeneous Bessmertny\u{\i} pencil $\bV_{\rm
H}(w)$ satisfying the condition \eqref{HApen2}. First of all,
$\bV_{{\rm H},22}(\be)$ must be invertible. Indeed, if $\ker \bV_{{\rm
H},22}(\be)\neq\{0\}$, then the maximum principle and positivity
of the coefficients $\bV_k$, $k=1,\ldots,d$, force $\bV_{\rm H}(w)$ to
have the form
$$\bV_{\rm H}(w)=\begin{bmatrix}
\bV_{{\rm H},11}(w) & \widetilde{\bV}_{{\rm H},12}(w) & 0\\
\widetilde{\bV}_{{\rm H},21}(w) & \widetilde{\bV}_{{\rm H},22}(w) &
0\\
0 & 0 & 0
\end{bmatrix}\colon\begin{bmatrix}
\mathcal{U}\\
 \overline{\rm Ran}\, {\bV}_{{\rm H},22}(\be)\\
{\rm Ker}\, \bV_{{\rm H},22}(\be)
\end{bmatrix}\to\begin{bmatrix}
\mathcal{U}\\
 \overline{\rm Ran} \, {\bV}_{{\rm H},22}(\be)\\
{\rm Ker}\, \bV_{{\rm H},22}(\be)
\end{bmatrix},$$
in contradiction with condition \eqref{HApen2} for a Herglotz-Agler
pencil. Therefore, one can replace the pencil $\bV_{\rm
H}(w)$
by the pencil $$\widehat{\bV}_{\rm H}(w)=\begin{bmatrix} I & 0\\
0 & \bV_{{\rm H},22}(\be)^{-1}
\end{bmatrix} \bV_{\rm H}(w)\begin{bmatrix} I & 0\\
0 & \bV_{{\rm H},22}(\be)^{-1}
\end{bmatrix},$$
which satisfies the condition \eqref{HApen2} and provides another
Bessmertny\u{\i} transfer-function realization of $f$:
$$f(w)=\widehat{\bV}_{{\rm H},11}(w)-\widehat{\bV}_{{\rm
H},12}(w)\widehat{\bV}_{{\rm H},22}(w)^{-1}\widehat{\bV}_{{\rm
H},21}.$$ }
\end{remark}

In fact, the special case described in Remark \ref{rem:pencil-normal} covers the whole
class $\mathcal{B}_d(\mathcal{U})$.
In the following theorem, we collect the characterizations of this
class from \cite{K-V}, together with the two additional
characterizations: via its image in the class
$\mathcal{HA}(\mathbb{D}^d,\mathcal{L(U)})$ under the Cayley
transform over the variables, and as this special case of Theorem
\ref{T:HA-Pi''}.
\begin{theorem}[]\label{T:Bessm}
  Given a function $f \colon
     \Pi^{d} \to \cL(\cU)$, the following are equivalent:
     \begin{enumerate}
\item $f$ is in the Herglotz--Agler class $\mathcal{HA}(\Pi^{d},
\cL(\cU))$ and can be extended to a
    holomorphic $\mathcal{L(U)}$-valued
    function on
    $\Omega_d:=\bigcup\limits_{\lambda\in\mathbb{T}}(\lambda\Pi)^d$
    which satisfies the following conditions: \begin{itemize}
    \item[(a)] \emph{Homogeneity:}  $f(\lambda
    w_1,\ldots,\lambda w_2)=\lambda f(w_1,\ldots,w_d)$ for every
    $\lambda\in\mathbb{C}\setminus\{0\}$ and
    $w=(w_1,\ldots,w_d)\in\Omega_d$.
    \item[(b)] \emph{Real Symmetry:}
    $f(\overline{w}_1,\ldots,\overline{w}_d)=f(w_1,\ldots,w_d)^*$  for every
    $w\in\Omega_d$.
\end{itemize}

\item $f$ has a $\Pi^{d}$-Bessmertny\u{\i} decomposition, i.e., there
exist $\cL(\cU)$-valued positive kernels $K_{1}, \dots, K_{d}$ on
$\Pi^{d}$ such that the identity
$$
 f(w) = \sum_{k=1}^{d} w_{k} K_{k}(z,w),\quad z,w\in\Pi^d,
$$
holds, or equivalently, the following two identities hold:
$$
 f(z)^*\pm f(w) = \sum_{k=1}^{d} (\overline{z}_k\pm w_{k}) K_{k}(z,w),\quad
 z,w\in\Pi^d.
$$

\item $f$ belongs to the Bessmertny\u{\i} class
$\mathcal{B}_d(\mathcal{U})$, i.e., there exist a Hilbert space
$\cX$ and a Bessmertny\u{\i} pencil
$$
  \bV(w) = \bV_{\rm H}(w) = \begin{bmatrix}
\bV_{{\rm H},11}(w) & \bV_{{\rm H},12}(w)\\
\bV_{{\rm H},21}(w) & \bV_{{\rm H},22}(w)
\end{bmatrix}=
  \sum_{k=1}^dw_k \bV_k,
$$
where $\bV_k\in\mathcal{L}(\mathcal{U}\oplus\mathcal{X})$ are
positive semidefinite operators,  so that \eqref{HApen2} holds and
$f(w) = f_{\bV}(w)$ (with $f_{\bV}(w)$ as in \eqref{E:Bessm2}).

\item The double Cayley transform of $f$,
$$S(\zeta)=\mathcal{C}(f):=\Big[f\Big(\frac{1+\zeta}{1-\zeta}\Big)-I\Big]
\Big[f\Big(\frac{1+\zeta}{1-\zeta}\Big)+I\Big]^{-1},$$ belongs to
the Schur--Agler class $\mathcal{SA}(\mathbb{D}^d,\mathcal{L(U)})$
and has a unitary Givone--Roesser realization \eqref{SAreal} with
a unitary and self-adjoint colligation matrix
$\bU=\bU^*=\bU^{-1}$.

\item The Cayley transform of $f$ over the variables,
$$
F(\zeta)=f\Big(\frac{1+\zeta}{1-\zeta}\Big),
$$
belongs to the
Herglotz--Agler class $\mathcal{HA}(\mathbb{D}^d,\mathcal{L(U)})$
and has a representation \eqref{DdHerreal3} with $R=0$,
$U=U^{-1}=U^*$, and ${\rm Ran}\,V$ contained in the 1-eigenspace of
$U$.

\item There exist a Hilbert space $\mathcal{X}$,  its subspaces
$\cX^{(0)}$, $\cX^{(1)}$ with $\cX=\cX^{(1)}\oplus\cX^{(0)}$, a
decomposition of $I_{\cX}$,
$P(w)=w_1P_1+\cdots+w_dP_d=\begin{bmatrix} P_{11}(w) &
P_{10}(w)\\
P_{01}(w) & P_{00}(w) \end{bmatrix}$ (with respect to the two-fold
decomposition of $\cX$), and $V_1\in\mathcal{L}(\cU,\cX^{(1)})$,
such that $f$ has the form \eqref{E:Bessm1}.
 \end{enumerate}
\end{theorem}
\begin{proof}
The equivalence of statements (1), (2), (3), and (4) has been
proved in \cite{K-V}.

(5)$\Rightarrow$(6) has been shown above, in the first paragraph
of Special case 4. This is an application of the construction in
the proof of Theorem \ref{T:HA-Pi''} to this special case.

 (6)$\Rightarrow$(3) is obvious.

 (4)$\Rightarrow$(5). The function $F$ is related to
 $S=\mathcal{C}(f)$ as in \eqref{E:F-S}. Let $\bU=\begin{bmatrix}
 \bA & \bB\\
 \bC & \bD
 \end{bmatrix}$ be a unitary colligation matrix
 providing the transfer-function realization for $S$ as in (4),
 i.e., $\bU^*=\bU$, together with the spectral decomposition of $I_{\cX}$, $P(\zeta)=\zeta_1P_1+\cdots+\zeta_dP_d$.
  As it was shown in Remark \ref{R:HAdisk}, $F$
 has a representation \eqref{DdHerreal2} with skew-adjoint $R$, unitary
 $U=\bU_0^*=\bA^*-\bC^*(I-\bD^*)^{-1}\bB^*$
  and
$V=\frac{1}{\sqrt{2}}\bB$. Now we also have that
$R=R^*=-R^*$, which means that $R=0$, and that
$U=U^*=\bA-\bB(I-\bD)^{-1}\bB^*$. It follows that the space $\cX$
has an orthogonal decomposition $\cX=\cX^{(1)}\oplus\cX^{(0)}$,
where $\cX^{(1)}$ is the 1-eigenspace of $U$ and $\cX^{(0)}$ is
the (-1)-eigenspace of $U$. With respect to this decomposition,
let us write
$$U=\bU_0=\begin{bmatrix}
I & 0\\
 0 & -I
\end{bmatrix},\quad
  \bA=\begin{bmatrix}
\bA_{11} & \bA_{10}\\
\bA_{01} & \bA_{00}
\end{bmatrix},\quad \bB=\begin{bmatrix}
\bB_1\\
\bB_0
\end{bmatrix}.$$
Since $\bA_{00}$ is a self-adjoint and contractive, the identity
$$-I=\bA_{00}+\bB_0(I-\bD)^{-1}\bB_0^*$$
is possible only if $\bB_0=0$. Thus ${\rm im}\,V\subseteq
\cX^{(1)}$, which completes the proof.
\end{proof}

 \smallskip

 \noindent
 \textbf{Special case 5: the single-variable case.}
 When we specialize
      Theorem \ref{T:HA-Pi''} to the single-variable case, some
      simplifications occur.  In particular the Herglotz pencil
      \eqref{A(w)} involves only two operators, namely the flip of a
      $\Pi$-impedance-conservative system node which can be assumed
      to be the canonical form \eqref{A0}
      $$  \bU_{0} = \begin{bmatrix} R - V_{0}^{*}T V & V_{0}^{*}
      (I+T) \\ -(I+T)^{*} V_{0} & -T \end{bmatrix},
      $$
      along with a single positive operator necessarily of the
      diagonal form
      $$
      \bU_{1} = \begin{bmatrix} V_{1}^{*} V_{1} & 0 \\ 0 & I
  \end{bmatrix}.
  $$
  Then the pencil $\bU(w)$ has the form
  $$
  \bU(w) = \begin{bmatrix}  R - V_{0}^{*}T V_{0} & V_{0}^{*} (I+T) \\
  -(I+T)^{*} V_{0} & -T \end{bmatrix} + w \begin{bmatrix}  V_{1}^{*}
  V_{1} & 0 \\ 0 & I \end{bmatrix}
  $$
  and the Bessmertny\u{\i} transfer-function realization for $f \in
  \cH(\Pi, \cL(\cU))$ becomes
  \begin{equation}   \label{Bess-real}
  f(w) = R + w V_{1}^{*} V_{1} +  V_{0}^{*}(-T + (I+T) (wI - T)^{-1} (I+T)^{*} ) V_{0}.
  \end{equation}
  For simplicity, let us now assume that $\cU = {\mathbb C}$ (so $f$
  is scalar-valued). Then $R \in {\mathbb C}$ is just a purely
  imaginary number and $V_{0}^{*} V_{0}$ is an operator on ${\mathbb
  C}$ and so can be identified with a nonnegative real number
  $\alpha$ (the image of the operator $V_{0}^{*} V_{0}$ acting on $1 \in {\mathbb
  C}$).  From the representation \eqref{Bess-real}, we see that there
  is no harm in cutting the state space $\cX^{(0)}$ (on which $T$ is
  acting) down to the smallest subspace reducing for $T$ which
  contains the range of the rank-1 operator $V_{0}$, i.e., we may
  assume that $V_{0} \cdot 1$ is a cyclic vector for $T$.
    Then the spectral theorem tells us that there is a
  measure $\nu$ on the imaginary line $i{\mathbb R}$ so that $T$ is
  unitarily equivalent to
  $$
    T= M_{-\zeta} \colon f(\zeta) \mapsto -\zeta f(\zeta)
  $$
  acting on $L^{2}(\nu)$.  Without loss of generality we take the
  cyclic vector $V_{0} \cdot 1$ to be the function $\frac{1}{1 -
  \zeta}$.  As this function must be in $L^{2}(\nu)$, we conclude that
   $\frac{1}{1-\zeta} \in L^{2}(\nu)$,
  i.e., that $\frac{1}{1 + |\zeta|^{2}} {\tt d}\nu(\zeta)$ is
  a finite measure.  When this is done then we see that the adjoint
  operator $V_{0}^{*} \colon L^{2}(\nu) \mapsto {\mathbb C}$ is given
  by
  $$
  V_{0}^{*} \colon f(\zeta) \mapsto \int_{i{\mathbb R}}
  \frac{1}{1-\zeta} f(\zeta)  \, {\tt d}\nu(\zeta).
  $$
  Then the Bessmertny\u{\i} realization
  \eqref{Bess-real} for $f(w)$  collapses to the integral
  representation formula
  \begin{equation*}
  f(w)   = \alpha w + R + \int_{i {\mathbb R}} \left[ \frac{\zeta}{1 +
  |\zeta|^{2}} + \frac{1}{\zeta + w} \right] \, {\tt d} \nu(\zeta).
  \end{equation*}
  This agrees with the classical Nevanlinna integral representation
  for holomorphic functions taking the right halfplane into itself.
  Actually the formula is usually stated for holomorphic functions
  taking the upper halfplane into itself (see \cite[Theorem 1 page
  20]{Donoghue}); however the correspondence $\widetilde f(\omega)
  \mapsto f(w):= -i \widetilde f(iw)$ between $\widetilde f$ in the Nevanlinna class and $f$ in the
 Herglotz class enables one to easily convert one integral
 representation to the other.  We also point out that our proof
 (starting with the Herglotz representation \eqref{DdHerreal2} for
 the Herglotz function $F(\zeta)$ on the disk and then separating out
 the $1$-eigenspace of the unitary operator $U$ in that
 representation) is just an operator-theoretic analogue of the proof
 of the integral representation formula in \cite{Donoghue}, where one
 starts with the integral Herglotz representation
 $$
 F(\zeta) = R + \int_{\mathbb T} \frac{ t + \zeta}{t - \zeta}\, {\tt d} \mu(t)
 $$
 and then separates out any point mass of $\mu$ at the point $1$ on
 the circle.

 \begin{remark}  \label{R:TullyDoyle} {\em  In recent work
     \cite{ATDY1}, Agler--Tully-Doyle--Young  obtain a realization
     formula for the most general scalar-valued Nevanlinna--Agler
     function on the upper polyhalfplane.   It is a straightforward matter to adjust the
     formulas to the right polyhalfplane setting which we have here
     and to extend the results to the operator-valued case.  The
     result amounts to combining our formulas \eqref{fform},
     \eqref{M1} and \eqref{N1}, i.e.,
     \begin{align}
 f(w) =  & R + V^{*} \left[\begin{matrix} I & 0 \\ 0 & -(I+T)
    \end{matrix}\right]  \left( P(w) \left[ \begin{matrix}
    0 & 0 \\ 0 & I \end{matrix} \right] + \left[
    \begin{matrix} I & 0 \\ 0 & -T \end{matrix} \right]
    \right)^{-1}  \cdot  \notag \\
& \quad \cdot   \left( -P(w) \left[ \begin{matrix} I & 0 \\ 0 & -T
    \end{matrix} \right] + \left[ \begin{matrix} 0 & 0 \\ 0
    & I \end{matrix} \right] \right) \left[ \begin{matrix}
    I & 0 \\ 0 & -(I-T)^{-1} \end{matrix} \right] V.
    \label{ATDY}
  \end{align}
  However the associated Bessmertny\u{\i} pencil
  $$\bV(w) = \begin{bmatrix} R &
  V^{*} \left[ \begin{smallmatrix} I
  & 0 \\ 0 & -(I+T) \end{smallmatrix} \right] \\
  ( -P(w) \left[ \begin{smallmatrix} I & 0 \\ 0 & -T \end{smallmatrix}
  \right] + \left[ \begin{smallmatrix} 0 & 0 \\ 0 & I
\end{smallmatrix} \right] ) \left[ \begin{smallmatrix} I & 0 \\ 0 &
-(I-T)^{-1} \end{smallmatrix} \right] V  & P(w) \left[ \begin{smallmatrix} 0 & 0 \\
0 & I \end{smallmatrix} \right] + \left[ \begin{smallmatrix} I & 0 \\
0 & -T \end{smallmatrix} \right] \end{bmatrix}
$$
lacks the symmetry properties of what we are calling a
Herglotz--Agler operator pencil (see Definition \ref{D:HApencil}).
Hence there is no easy analogue of the proof of (3) $\Rightarrow$
(2) in Theorem \ref{T:HA-Pi''} and it is not all transparent from
the presentation \eqref{ATDY} why the resulting function
\eqref{ATDY} has a $\Pi^d$-Herglotz--Agler decomposition
\eqref{PiHerAgdecom};  indeed, it takes several pages of
calculations in \cite{ATDY1} (see Propositions 3.4 and 3.5 there)
to arrive at this result.

There are other results in \cite{ATDY1} and \cite{ATDY2} using
realization theory to characterize various types of boundary
behavior of the function $f$ at  infinity; we do not go into this
topic here. }\end{remark}


\begin{thebibliography}{99}

\bibitem{Agler1990} J.~Agler,
\newblock {\em On the representation of certain holomorphic functions
defined on a polydisc},
\newblock In  Topics in operator theory: Ernst D. Hellinger Memorial
Volume,  Oper. Theory Adv. Appl.,
 Vol.~\textbf{48}, pp. 47--66,
  Birkh\"auser, Basel, 1990.

\bibitem{AMcC99}  J.~Agler and J.E.~McCarthy, {\em Nevanlinna--Pick
interpolation on the bidisk}, \newblock J. Reine Angew.~Math. {\bf
506} (1999), 191--204.

\bibitem{AMcCY} J.~Agler, J.E.~McCarthy, and N.J.~Young,  {\em
Operator monotone functions and L\"owner functions of several
variables}, Annals of Mathematics \textbf{176} (2012), 1783--1826.

\bibitem{ATDY1} J.~Agler, R.~Tully-Doyle, and N.J.~Young, {\em
Nevanlinna representations in several variables},
arXiv:1203.2261v2.

\bibitem{ATDY2} J.~Agler, R.~Tully-Doyle, and N.J.~Young, {\em
Boundary behavior of analytic functions of two variables via generalized
models}, Indag.~Math.~(N.S.) \textbf{23} (2012) no.~4, 995--1027.

\bibitem{ABT} Y.~Arlinskii, S.~Belyi, and E.~Tsekanovskii, {\em
Conservative Realizations of Herglotz--Nevanlinna Functions},
Operator Theory: Advances and Applications \textbf{217},
Birkh\"auser, 2011.

\bibitem{AN} D.Z.~Arov and M.A.~Nudelman, {\em Passive linear
stationary dynamical scattering systems with continuous time},
Integral Equations and Operator Theory \textbf{24} (1996), 1--45.

\bibitem{AI} T.A.\ Azizov and I..S.\ Iokhvidov, {\em Linear Operators
in Spaces with an Indefinite Metric}, Wiley, 1980.
\bibitem{B}
J. A. Ball.
\newblock {\em Multidimensional circuit synthesis and multivariable
dilation theory},  Multidimens.~Syst.~Signal Process \textbf{22}
(2011) no.~1-3, 27--44.

\bibitem{BB-Montreal}  J.A.~Ball and V.~Bolotnikov,
 {\em Canonical de Branges--Rovnyak
model transfer-function
realization for multivariable Schur-class functions}, in: Hilbert
Spaces of Analytic Functions (Ed. J.~Mashreghi, T.~Ransford, and
K.~Seip), CRM Proceedings \& Lecture Notes {\bf 51},
Amer.~Math.~Soc.,
Providence, 2010.


\bibitem{BK-V} J.A.~Ball and D.S.~Kaliuzhnyi-Verbovetskyi,
{\em Rational Cayley inner Herglotz--Agler functions:
positive-kernel decompositions and transfer-function
realizations}.  Linear Algebra Appl. \textbf{456} (2014),
138--156.

\bibitem{BS} J.A.~Ball and O.J.~Staffans, {\em Conservative
state-space realizations of dissipative system behaviors}, Integral
Equations and Operator Theory \textbf{54} (2006), 151--213.

\bibitem{BT}
J.~A. Ball and T.~Trent,
{\em Unitary colligations, reproducing kernel Hilbert spaces, and
  Nevanlinna-Pick interpolation in several variables},
 J.~Funct.~Anal. \textbf{157} (1998), 1--61.

 \bibitem{BHdS2008} J.~Behrndt, S.~Hassi, and H.~de Snoo, {\em
Functional
 models for Nevanlinna families}, Opuscula Mathematica \textbf{28}
 (2008) no.~3, 233--245.

  \bibitem{BHdS2009} J.~Behrndt, S.~Hassi, and H.~de Snoo, {\em
  Boundary relations, unitary colligations, and functional models},
  Complex Analysis and Operator Theory \textbf{3} (2009), 57--98.

 \bibitem{BHdST} S.~Belyi, S.~Hassi, H.~de Snoo, and E.~Tsekanovskii,
 {\em A general realization theorem for matrix-valued
 Herglotz--Nevanlinna functions}, Linear Algebra and its Applications
 \textbf{419} (2006) no.~2--3, 331--358.

 \bibitem{Bes}
M.~F. Bessmertny\u{\i}.
\newblock {\em Functions of Several Variables in the Theory of Finite
Linear
  Structures}.
\newblock Ph. {D. Thesis}, Kharkov University, Kharkov, 1982.
\newblock (Russian).

\bibitem{Bes1}
M.~F. Bessmertny\u{\i}.
\newblock {\em On realizations of rational matrix functions of
several complex variables.}
\newblock In: Interpolation Theory, Systems Theory and Related
Topics: The Harry Dym Anniversary Volume
 (D.~Alpay, I.~Gohberg, and V.~Vinnikov, eds.),
Oper. Theory Adv. Appl., Vol.~\textbf{134}, pp. 157--185,
Birkh\"{a}user-Verlag, Basel, 2002.

\bibitem{Bes2}
 M. F. Bessmertny\u\i.
 \newblock {\em On realizations
of rational matrix functions of several complex variables. II.}
Translated from the Russian by
V. Katsnelson.
\newblock In:  Reproducing Kernel Spaces and Applications,
{\em Oper. Theory Adv. Appl.}, Vol.~\textbf{143}, pp.  135--146,
Birkh\"{a}user, Basel, 2003.

\bibitem{Bes3}
 M. F. Bessmertny\u\i.
 \newblock {\em On realizations of rational matrix functions of
several variables. III.}
\newblock In: Current Trends in Operator Theory and Its
Applications, Oper. Theory Adv. Appl., Vol.~\textbf{149}, pp.
133--138, Birkh\"auser, Basel, 2004.

\bibitem{Bes4}
 M. F. Bessmertny\u\i.
 \newblock  {\em Functions of several variables in the theory of
finite linear
structures. I. Analysis.}
\newblock In: Operator Theory, Systems Theory and Scattering Theory:
Multidimensional Generalizations,
 Oper. Theory Adv. Appl.~\textbf{157}, pp. 91--106, Birkh\"auser,
Basel, 2005.

\bibitem{Bognar} J.\ Bogn\'ar, {\em Indefinite Inner Product Spaces},
Springer-Verlag, New York-Heidelberg-Berlin, 1974.

\bibitem{BM} D.~Bors and M.~Majewski, {\em On the existence of an
optimal solution of the Mayer problem governed by 2D continuous
counterpart of the Fornasini-Marchesini model},
 Multidimens.~Syst.~Signal Process \textbf{24}
(2013) no.~4, 657--665.

\bibitem{Donoghue}  W.F.~Donoghue, {\em Monotone Matrix Functions and
Analytic Continuation}, Die Grundlehren der mathematischen
Wissenshaften in Einzeldarstellungen Band \textbf{207}, Springer,
New York, 1974.

\bibitem{FM} E.~Fornasini and G.~Marchesini, {\em Doubly-indexed
dynamical systems:  state-space models and structural properties},
Math.~Systems Theory \textbf{12} (1978(, 59--72.

\bibitem{Fuhrmann} P.A.~Fuhrmann, {\em Linear Operators and Systems
in Hilbert Space}, McGraw-Hill, New York, New York, 1981.


 \bibitem{Helton} J.W.~Helton, {\em Systems with infinite-dimensional
    state space: the Hilbert space approach}, Recent Trends in System
    Theory, Proc.~IEEE \textbf{64} (1976) no.~1, 145--160.

\bibitem{K-V}
D. S. Kalyuzhny\u{\i}-Verbovetzki\u{\i}.
\newblock {\em On the Bessmertny\u{\i} class of homogeneous positive
holomorphic
functions of several variables.}
\newblock In:  Current Trends in Operator Theory and Its
Applications,  Oper. Theory Adv. Appl.~\textbf{149}, pp.  255-289,
Birkh\"{a}user, Basel, 2004.

\bibitem{MSW} J.~Malinen, O.J.~Staffans and G.~Weiss, {\em When is a linear
system conservative?} Quart.~Appl.~Math. \textbf{64} (2006) no.~1,
61--91.


\bibitem{Paulsen}  V.~Paulsen, {\em Completely Bounded Maps and
Operator Algebras}, Cambridge Studies in Advanced Mathematics
\textbf{78}, Cambridge University Press, 2002.

\bibitem{Phillips59} R.S.\ Phillips, {\em Dissipative operators and
hyperbolic systems of partial differential equations}, Trans.\ Amer.\
Math.\ Soc.\ \textbf{90} (1959), 193--254.

\bibitem{Sala87} {D.~Salamon}, Infinite dimensional linear
systems with unbounded control and observation: a functional analytic
approach, {\em Trans. Amer. Math.  Soc.} \textbf{300} (1987), pp.~383--431.

\bibitem{Smuljan} Y.L.~\v{S}muljan, Invariant subspaces of semigroups
and the Lax-Phillips scheme, deposited in VINITI, N 8009-1386, Odessa,
49 pp., 1986.

\bibitem{StaffansMCSS}  O.J.~Staffans, {\em Passive and conservative
continuous-time impedance and scatterings systems. Part I: well-posed
systems}, Math.~Control Signals Systems \textbf{15} (2002), 291--315.



\bibitem{Staffans-ND}  O.J.~Staffans, {\em Passive and conservative
infinite-dimensional impedance and scattering systems (from a
personal point of view)}, in Mathematical Systems in Biology, Communication,
Computation and Finance (MTNS2002 Notre Dame, Indiana) (ed.
J.~Rosenthal and D.S.~Gilliam) pp. 373--414, IMA Volume \textbf{314}
Springer, 2003.


\bibitem{Staffans}  O.J.~Staffans, {\em Well-posed Linear Systems},
Encyclopedia of Mathematics and Its Applications \textbf{103},
Cambridge University Press, 2005.

\bibitem{Staffans2013}  O.J.~Staffans, {\em On scattering passive
system nodes and maximal scattering dissipative operators},
Proc.~Amer.~Math.~Soc. \textbf{141} no.~4 (2013), 1377--1383.

\bibitem{SW2012} O.J.\ Staffans and G.\ Weiss, {\em A physically motivated class of
scattering passive linear systems}, SIAM J.\ Control Optim.\
\textbf{50} No.\ 5 (2012), 3083--3112.

\bibitem{WeissTAMS94}  G.~Weiss, {\em Transfer functions of regular linear
systems. Part I: characterizations of regularity},
 {\em Trans. Amer. Math.  Soc.} \textbf{3342} (1994), pp.~827--854.

 \bibitem{WS2013} G.\ Weiss and O.J.\ Staffans,  {\em Maxwell's
 equations as a scattering passive linear system}, SIAM J.\ Control Optim.\
 \textbf{51} (2013), 3722--3756.



\end{thebibliography}
\end{document}